\documentclass[12pt]{amsart}

\usepackage[hmargin=0.8in,height=8.6in]{geometry}
\usepackage{amssymb,amsthm, times}
\usepackage{delarray,verbatim}
\usepackage{ifpdf}
\ifpdf
\usepackage[pdftex]{graphicx}
 \else
\usepackage[dvips]{graphicx}
 \fi

\usepackage{bm}

\usepackage{ifpdf}
\usepackage{color}
\definecolor{webgreen}{rgb}{0,.5,0}
\definecolor{webbrown}{rgb}{.8,0,0}
\definecolor{emphcolor}{rgb}{0.5,0.95,0.95}

\usepackage{hyperref}
\hypersetup{%
          colorlinks=true,
          linkcolor=webbrown,
          filecolor=webbrown,
          citecolor=webgreen,
          breaklinks=true}
\ifpdf \hypersetup{pdftex,
            pdfstartview=FitH, %%Fit, FitB, FitH
            bookmarksopen=true,
            bookmarksnumbered=true
} \else \hypersetup{dvips} \fi

\newcommand {\ud}{{\rm d}}

\numberwithin{equation}{section}

\newtheorem{proposition}{Proposition}[section]

\newtheorem{remark}{Remark}[section]
\newtheorem{lemma}{Lemma}[section]

\newtheorem{assump}{Assumption}[section]

\numberwithin{remark}{section} \numberwithin{proposition}{section}
\numberwithin{corollary}{section}
\newcommand {\R}{\mathbb{R}}

\newcommand {\p}{\mathbb{P}}

\newcommand {\E}{\mathbb{E}}

\newcommand{\diff}{{\rm d}}

\newcommand{\lev}{L\'{e}vy }

\newcommand{\e}{\mathbb{E}}

\begin{document}
\title{On optimal periodic dividend strategies  for L\'evy risk processes}

	\thanks{This version: \today.   }

\author[K. Noba]{Kei Noba$^*$}

\author[J. L. P\'erez]{Jos\'e-Luis P\'erez$^\dag$}
%\address[J. L. P\'erez]{.}
%\email{ }
\thanks{$*$\, Department of Mathematics, Graduate School of Science, Kyoto University
Sakyo-ku, Kyoto 606-8502, Japan. Email: knoba@math.kyoto-u.ac.jp (K. Noba), kyano@math.kyoto-u.ac.jp (K. Yano)}
\thanks{$\dagger$\, Department of Probability and Statistics, Centro de Investigaci\'on en Matem\'aticas A.C. Calle Jalisco s/n. C.P. 36240, Guanajuato, Mexico. Email: jluis.garmendia@cimat.mx.  }
\author[K. Yamazaki]{Kazutoshi Yamazaki$^\ddagger$}
%\address[K. Yamazaki]{}
\thanks{$\ddagger$\, (corresponding author) Department of Mathematics,
Faculty of Engineering Science, Kansai University, 3-3-35 Yamate-cho, Suita-shi, Osaka 564-8680, Japan. Email: kyamazak@kansai-u.ac.jp.   }

\author[K. Yano]{Kouji Yano$^*$}
%\email{}
\date{}

\begin{abstract} %\red{[added this]}
In this paper, we revisit the optimal periodic dividend problem, in which dividend
payments can only be made at the jump times of an independent Poisson
process.  In the dual (spectrally positive L\'evy)
model, recent results have shown the optimality of a periodic barrier strategy, which pays dividends at Poissonian dividend-decision times, if and only if the surplus is
above some level.  In this paper, we show the optimality of this strategy for a
spectrally negative \lev process whose dual has a completely monotone \lev
density.  The optimal strategies and value functions are concisely
written in terms of the scale functions. Numerical results are also provided.
\\
\noindent \small{\noindent  AMS 2010 Subject Classifications: 60G51, 93E20, 91B30
% \red{[added this. can you check?]} \\ %C44, C61, G24, G32, G35 \\
}% \red{[added this. can you check?]}}
\end{abstract}

\maketitle

\section{Introduction} %\red{[added this. can you check?]}} 

In the classical de Finetti's optimal dividend problem, the expected total discounted dividends accumulated until ruin are maximized. To model the surplus of an insurance company that increases by premium and decreases by insurance payments, a compound Poisson process with downward jumps or more generally  a spectrally negative \lev process is used.  Nowadays, fluctuation theory and scale functions are known to be useful, particularly if the optimal strategy is guessed to be a \emph{barrier strategy} reflecting the underlying process at an upper barrier.  Numerous computations are possible for the reflected \lev process, and these can be used to solve the problem in a straightforward manner.

Despite the analytical tractability of the classical continuous-time model, the barrier strategies are unfortunately not implementable in practice. On the other hand,  while the models with deterministic discrete payment times are ideal, they lack analytical tractability, and numerical methods are required to solve them. Recently, with the aim of developing a more realistic  yet analytically tractable model, random discrete payment times were considered.  For example, in the research by Albrecher et al.\
 \cite{ABT, ACT}, if the random times are suitably chosen, analytical approaches can be used to compute various identities of interest. Random observation times are also suggested in various economic literatures.  See, for example, the discussion in the introduction of \cite{SS} motivated by  \emph{rational inattention}  \cite{Sims} in macroeconomics literature. See also the discussions given in \cite{Lempa} and \cite{BL} for  real option problems with random intervention times.

In this paper, we focus on the \emph{periodic barrier strategy} and its optimality when dividend
payments can only be made at the jump times of an independent Poisson
process.   In this context,  Avanzi et al.\ \cite{ATW} solved the case with positive hyperexponential jumps; this case was later generalized by P\'erez and Yamazaki \cite{PY} and Zhao et al.\ \cite{ZCY} for a general spectrally positive \lev process. 
% In this paper, we focus on the spectrally negative \lev case, which is more suitable in the insurance context.
%Detailed motivation of this model is given in \cite{ATW}.
%While the classical continuous-time model assumes that dividend payments can be made instantaneously  at any time, in reality it can be done only at intervals. 
By assuming that the intervals are independent exponential random variables, 
%the problem still posses a one-dimensional Markovian structure.  \red{How about ``
we can still formulate it as a one-dimensional Markovian problem. It is also known (see, e.g., \cite{leung2015analytic} in the context of finance) that this can give an approximation for the case of constant intervals.  See also the Erlang(2) interarrival time case recently considered by Avanzi et al. \cite{ATW_erlang}. 

 %For more detailed motivation of this model, see TODO.
%To our best knowledge, 
%In this paper, we focus on the spectrally negative \lev model, that is more suitable in the insurance context.

%the existing results are  to our best knowledge limited to the dual (spectrally positive \lev) model.   In particular,

%revisit the periodic version, where dividend
%payments can only be made at the jump times of an independent Poisson
%process.  

%In particular, Avanzi et al.\ TODO solved the case with positive hyperexponential jumps, and was later generalized by P\'erez and Yamazaki TODO for a general spectrally positive \lev process. The solution can be shown to be again of barrier type that pushed the process down to a certain barrier at each dividend decision time at which the process is above it.

We consider the spectrally negative \lev case, which is suitable in the insurance context. 
%\red{[remove the following sentence and combine with the next paragraph?]}\blue{[Agreed]} Differently from the dual (spectrally positive L\'evy) model, it is known, even in the classical model, that the optimality of a barrier strategy can fail 
% depending on the choice of the \lev measure.  
 In the spectrally negative case, the surplus process can instantaneously jump to the liquidation region where the value function flattens suddenly, and for this reason the analysis is sensitive to the choice of the \lev measure. On the other hand, this never happens in the spectrally positive case and the liquidation region can be ignored.  For these reasons, the proof of optimality is significantly more difficult in the spectrally negative model than in the spectrally positive model. In particular, in order to show the variational inequalities, the spectrally positive \lev case can be handled using known general results on the scale function.  On the other hand, for the spectrally negative case, these are not sufficient and further properties of the scale functions  that hold only for a subset of spectrally negative \lev processes need to be considered (see Theorem 2 and Corollary 1 in \cite{L2}).
% \\ 
% the verification argument does not require
% The main difficulty in the spectrally negative case appears in the verification step. 
% \\
% Intuitively, it becomes difficult because in the spectrally negative \lev case, the process can jump from the controlling region (at which ) to the waiting and liquidation region.  In particular, as the value becomes zero in the liquidation region, the optimality becomes sensitive to the choice of the jump measure. On the other hand, for the spectrally positive case, it never jumps below zero and optimality tends to hold for any choice of the \lev measure. 

 One currently known sufficient condition is the \emph{completely monotone} assumption on the \lev density, under which the scale function %\red{[singular form is more natural here?]} 
 can be written as the difference of an exponential function and a completely monotone function (see Remark \ref{remark_scale_function_completely_monotone}).  Accordingly, Loeffen \cite{L} (see also Yin et al. \cite{YW} for an analytic approach) and Kyprianou et al.\ \cite{KLP} showed the optimality of a barrier strategy in the classical case and that of a threshold strategy under the absolutely continuous assumption on the dividend strategy, respectively.  In this paper, we show that the completely monotone assumption is again a sufficient condition for the optimality of a periodic barrier strategy in the considered problem.

The class of \lev  risk processes with completely monotone \lev densities include a variety of important processes.
% such as  
%gamma, tempered stable, and meromorphic \lev processes. \blue{[How about "...important processes. 
To name a few, we have the spectrally negative $\alpha$-stable process used in \cite{F}, the (one-sided) gamma process considered in \cite{DG1}, the (one-sided) inverse Gaussian process used in \cite{DG}, and finally Cr\'amer-Lundberg processes with heavy-tailed Weibull, Pareto, and hyperexponential jumps (see \cite{A} Chapter 1.2).  %\red{[How about ``with heavy-tailed Weibull, Pareto, and hyperexponential distributions"?]}
%\blue{ The meromorphic \lev process are particularly..."} \red{[what can we say about meromorphic \lev processes?]}

By Bernstein's theorem, a completely monotone function has the form
\begin{align*}
f(t) = \int_0^\infty e^{-tx} \lambda(\diff x),
\end{align*}
for a possibly infinite measure $\lambda$.
From this, it can be seen that a \lev measure with a completely monotone density is roughly a mixture of exponential distributions, implying that larger jumps are less frequent.  We refer the reader to, e.g., Feldmann and Whitt \cite{FW} regarding the approximation of a completely monotone distribution using mixtures of exponential distributions. The empirical results shown in \cite{CGMY} suggest that  financial models should be modeled using \lev processes with completely monotone \lev densities.

%The Fine Structure of Asset Returns: An Empirical Investigation" by CGMY gives a positive answer to the hypothesis of completely monotone returns as in  pages 328 and also 332 (conclusion).
% It is particularly important in the insurance context, because it can model the arrivals of (possibly infinite) mixtures of exponential claims with possibly different means.
% \red{[remove the rest?]}\blue{[Agreed]}--  it is suitable to use if individual claims follow exponential distributions with different means.} 

 Under the periodic barrier strategy, the surplus is pushed down to a given barrier at each Poisson arrival time at which it is above the barrier.   The controlled process is precisely the \emph{Parisian-reflected \lev process} considered in \cite{APY, PYM}.  Its fluctuation identities can be used efficiently to conduct the following ``\emph{guess and verify}" procedure: 
%Thanks to the developments of the so-called Parisian reflected process TODO that is pushed down to a given barrier at each Poisson observation times at which the process is above it,  we can take efficiently the ``guess and verify" procedure in a straightforward manner. 
%In order to solve the problem, we take the following steps.
\begin{enumerate}
\item The expected net present value (NPV) of dividends under the periodic barrier strategy can be written in terms of generalizations of the scale functions.
% by directly using the results in TODO. 
The candidate barrier, which we call $b^*$, is chosen so that the corresponding (candidate) value function, if $b^* > 0$,  \emph{becomes smoother at the barrier}.  In particular, this candidate barrier is chosen so that it becomes $C^2(0, \infty)$ (resp.\ $C^3(0, \infty)$) for the case of bounded (resp.\ unbounded) variation.  
\item We then analyze the existence of $b^*$ such that the expected NPV satisfies the smoothness condition. To this end, we use the special property of the scale function under the completely monotone assumption that its derivative first decreases and then increases.  We will achieve $b^* > 0$ or $b^*=0$, where, for the case $b^* = 0$, the future prospect is negative and it should be liquidated as quickly as possible.

% Under a certain condition at which the future prospect is negative, one wants to liquidate as quickly as possible; in this case, we set $b^*=0$ and the smoothness condition is not necessarily satisfied.
\item By using the selected candidate barrier $b^* \geq 0$, its optimality is confirmed through a verification lemma requiring the analysis of the harmonic property and the slope of the candidate value function.  By using the known property of the scale functions, the harmonic property can be analyzed easily.  In contrast, the analysis of the slope above the barrier $b^*$ is a great challenge.
%On the other hand, the biggest challenge is to analyze the slope above the barrier $b^*$. 
Motivated by the technique used in Kyprianou et al.\ \cite{KLP}, we manage this by using the decomposition of the scale function to an exponential function and a completely monotone function.
\end{enumerate}

%In order to see 
To observe the link with the classical case, we also  analyze the convergence as the Poisson arrival rate increases to infinity.  In particular, we show that the optimal barrier $b^*$ and  the value function converge to those in the classical case \cite{APP2007, L}.  The analytical results are further confirmed through  a
 sequence of numerical experiments. By using a simple case with  i.i.d.\ exponentially distributed jumps, we confirm the optimality and conduct sensitivity analysis with respect to the parameters describing the problem.

The remainder of the paper is organized as follows. In Section \ref{section_preliminaries}, we review the spectrally negative \lev process and present the mathematical model.  In Section \ref{PR}, we review the periodic barrier strategy and obtain the expected NPV of dividends by using the scale functions.  Section \ref{sf_d} states a condition of the candidate barrier $b^*$ and shows its existence.  The optimality of the selected strategy is confirmed in Section \ref{exist_thres}.  Section \ref{section_convergence} shows the analysis of the convergence as the Poisson arrival rate goes to infinity. Finally, Section \ref{section_numerics} concludes the paper with numerical results.

	\section{Preliminaries} \label{section_preliminaries}
\subsection{Spectrally negative \lev processes} %\red{[changed from $X_t$ to $X(t)$]}
Let $X=(X(t); t\geq 0)$ be a L\'evy process defined on a  probability space $(\Omega, \mathcal{F}, \p)$.  For $x\in \R$, we denote the law of $X$ by $\p_x$ when it starts at $x$ and, for convenience, write  $\p$ in place of $\p_0$. Accordingly, we write $\e_x$ and $\e$ for the associated expectation operators. %\red{[added this because it is used later] 
%We let $(\mathcal{F}_t)_{t \geq 0}$ be the filtration generated by $X$. 
Throughout the paper, we assume that $X$ is \textit{spectrally negative},   meaning here that it has no positive jumps and that it is not the negative of a subordinator. We define the Laplace exponent %\green{[In the following I changed the L\'evy measure to a measure in $(0,\infty)$ as in Loeffen, to avoid the confussion when using the dual process for the completely monotone assumption. Is this ok?]}
%$\psi(\theta):[0,\infty) \to \R$, i.e.
%\[
%\e\big[{\rm e}^{\theta X(t)}\big]=:{\rm e}^{\psi(\theta)t}, \qquad t, \theta\ge 0,
%\]
%given by the \emph{L\'evy-Khintchine formula} %\red{[changed from $x$ to $z$]}
\begin{equation}\label{lk}
\psi(\theta):=\log \e\big[{\rm e}^{\theta X(1)}\big] = \gamma\theta+\frac{\eta^2}{2}\theta^2+\int_{(-\infty,0)}\big({\rm e}^{\theta z}-1 -\theta z \mathbf{1}_{\{z > -1 \}}\big)\Pi(\ud z), \quad \theta \geq 0,
\end{equation}
%\blue{[Kazu: We are using $\nu^{\pi}$ for the dividend payments, change to $\eta$?]}
%\red{[Now that $\sigma$ is used for ruin time, how about $\nu$?]}\blue{[Kazu: How about changing the ruin time to $\eta$ or $\nu$ instead, to avoid confussion with classical notation?]}
where $\gamma\in \R$, $\eta \ge 0$, and $\Pi$ is a measure on $(-\infty,0)$ called the L\'evy measure of $X$ that satisfies
\[
\int_{(-\infty,0)}(1\land z^2)\Pi(\ud z)<\infty.
\]

It is well-known that $X$ has paths of bounded variation if and only if $\eta=0$ and $\int_{(-1,0)} |z|\Pi(\mathrm{d}z)$ is finite.  In this case, its Laplace exponent is given by
\begin{equation}
	\psi(\theta) = c \theta+\int_{(-\infty,0)}\big( {\rm e}^{\theta z}-1\big)\Pi(\ud z), \quad \theta \geq 0, \label{psi_bounded_var}
\end{equation}
where 
\begin{align}
	c:=\gamma-\int_{(-1,0)} z\Pi(\mathrm{d}z). \label{def_drift_finite_var}
\end{align}
 Note that  necessarily $c>0$, since we have ruled out the case that $X$ has monotone paths.

%In this case $X$ can be written as
%\begin{equation}
%X(t)=ct-S(t), \,\,\qquad t\geq 0,\notag
%%\label{BVSNLP}
%\end{equation}
%where 
%\begin{align}
%	c:=\gamma-\int_{(-1,0)} z\Pi(\mathrm{d}z) \label{def_drift_finite_var}
%\end{align}
%and $(S(t); t\geq0)$ is a driftless subordinator. Note that  necessarily $c>0$, since we have ruled out the case that $X$ has monotone paths; its Laplace exponent is given by
%\begin{equation}
%	\psi(\theta) = c \theta+\int_{(-\infty,0)}\big( {\rm e}^{\theta z}-1\big)\Pi(\ud z), \quad \theta \geq 0. \label{psi_bounded_var}
%\end{equation}
	
	%\subsection{Formulation of the problems.}{\color{blue} Working title.}
	\subsection{The optimal dividend problem with Poissonian dividend-decision times.}\label{dividends-strategy} %\red{JL: How about something like ``Optimal dividend problem with Poissonian dividend decision times"?}
	
	We assume that the dividend payments can only be made at the arrival times $\mathcal{T}_r :=(T(i); i\geq 1 )$ of a Poisson process $N^r=( N^r(t); t\geq 0) $ with intensity $r>0$, which is independent of  $X$.  %\red{[moved this sentence here]}
%	The set of dividend-decision times is denoted by $\mathcal{T}_r :=(T(i); i\geq 1 )$, where $T(i)$, for each $i\geq 1$, represents the $i^{\textrm{th}}$ arrival time of the Poisson process $N^r$. This implies that $T(i)-T(i-1)$, $i\geq 1$ (with $T(0) := 0$) are exponentially distributed with mean $1/r$. 
 Let $\mathbb{F} := (\mathcal{F}(t); t \geq 0)$ be the filtration generated by the processes $(X, N^r)$. 

In this setting, a strategy  $\pi := \left( L^{\pi}(t); t \geq 0 \right)$ is a nondecreasing, right-continuous, and $\mathbb{F}$-adapted process such that  the cumulative amount of dividends $L^{\pi}$ admits the form
	\[
	L^{\pi}(t)=\int_{[0,t]}\nu^{\pi}(s)\diff N^r(s),\qquad\text{$t\geq0$,}
	\]
	for some $\mathbb{F}$-adapted c\`agl\`ad process $\nu^{\pi}$.
	%Given that we can \red{pay dividends only} at the arrival times of the process $N^r$, we have that the process $L^{\pi}$ takes the following form
%	Here, for each $t\geq0$, $\nu^\pi(t)$ represents the dividend payment at time $t$ (if $\Delta N^r (t) > 0$) associated with the strategy $\pi$. In particular, the dividend payment at time $T(i)$  is given by $\nu^\pi(T(i))$ for each $i\geq 1$. 
For a more detailed description of the problem, we refer the reader to the spectrally positive case studied in \cite{PY}.
%In other words, a strategy $\pi$ can be essentially determined by a set $\mathcal{V}^{\pi}:=\{\nu^{\pi}(T(i)):i\geq 1\}$; we suppose, without loss of generality, the c\`agl\`ad property for covenience in verification lemma. \red{[it may be just better not to say the last sentence. For some people, it is quite counterintuitive why one needs to assume some condition for the whole path. We probably should not spell out this assumption too much.]}\blue{I agree Kazu.}
	
	% represents one particular dividend strategy. 

	The surplus process $U^\pi$ after dividends are deducted is such that %$U^\pi(0-) := x$ \red{[I guess we can remove $U^\pi(0-) := x$? because control at time $0$ is not possible? (bail-out case would be different when $x < 0$ though)]} and 
	\[
	U^\pi(t) := X(t)- L^\pi (t) =X(t)-\sum_{i=1}^{\infty}\nu^{\pi}(T(i))1_{\{T(i)\leq t\}}, \qquad\text{$0\leq t\leq \sigma_0^{\pi}$}, 
	\]
	where 
	\begin{align*}
		\sigma_0^{\pi}:=\inf\{t>0:U^{\pi}(t)<0\}
	\end{align*}
	is the corresponding (continuously-monitored) ruin time.  Here and throughout, let $\inf \varnothing = \infty$. 
%While the payment of dividends is allowed to cause immediate ruin, it cannot exceed the amount of surplus currently available. In other words, we also assume that
As in \cite{PY}, the payment cannot exceed the available surplus and hence
	\begin{align} \label{surplus_constraint}
		0\leq \nu^{\pi}(s) \leq U^\pi(s-), \quad s \geq 0.
	\end{align}

	%\red{[we can define $\Delta$ here?  It is currently defined in the appendix.]}
	%\red{[moved here]
	%where we use the following notation: $\Delta L^{\pi}(T(i)):=L^{\pi}(T(i))-L^{\pi}(T(i)-)$.
	Over  the set of all admissible strategies $\mathcal{A}_r$ that satisfy all the constraints described above,
	%\par Then we can express the surplus process $U^{\pi}$ in the following form
	%\[
	%U^{\pi}(t)=X(t)-\sum_{i=1}^{\infty}\nu^{\pi}(T(i))1_{\{T(i)\leq t\}}, \qquad\text{$0\leq t\leq \tau_0^{\pi}$}.
	%\]
	we need to maximize, for $q > 0$, the expected NPV of dividends paid until ruin: 
%associated with the strategy $\pi\in \mathcal{A}_r$, defined as
	\begin{align*}
		v_{\pi} (x) := \mathbb{E}_x \Big( \int_{[0,\sigma_0^{\pi}]} e^{-q t} \diff L^{\pi}(t)\Big) = \mathbb{E}_x \Big(\int_{[0,\sigma_0^{\pi}]} e^{-q t}  \nu^{\pi}(t)\diff N^r(t) \Big), \quad x \geq 0.
	\end{align*}
	Hence, the problem is to compute the value function
	\begin{equation*}
		v(x):=\sup_{\pi \in \mathcal{A}_r}v_{\pi}(x), \quad x \geq 0,
	\end{equation*}
	and obtain the optimal strategy $\pi^*$ that attains it, if such a strategy exists.
	
	%For the rest of the paper, 
	Hereafter,  we mean, by $\mathbf{e}_p$, %\blue{[How about "Hereafter $\mathbf{e}_p$, denotes..." ]} 
	an exponential random variable with parameter $p > 0$, independent of the process $X$ so that we can write $T(1) = \mathbf{e}_r$. 
	
		\section{%L\'evy processes Reflected at Poissonian times
			Periodic barrier strategies}\label{PR}
			
		%	\red{[to be shortened because it is essentially the same as the SP case.]}\blue{Agreed.}
		
		%\red{[JL: Moved this paragraph here]} 
		As in the spectrally positive case \cite{PY}, our objective is to show the optimality of the \emph{periodic barrier strategy}, %reflection strategy at Poissonian times
		say $\pi^{b} = (L_r^b(t); t \geq 0)$, with the resulting controlled process $U_r^{b}$ being the \emph{\lev process with Parisian reflection above} given as follows. 
%Namely, at each Poissonian dividend-decision time, dividends are paid whenever the surplus process is above $b$ and is pushed down so that the remaining surplus becomes $b$. %\red{[shortened and reorganized a bit below. Is this ok?]}
		%\red{[put some sentences back.]} 
%		The controlled process $U_r^{(b)}$ is precisely the \emph{\lev process with Parisian reflection above} given as follows. 
		We have
	\begin{align} \label{X_X_r_the_same}
		U_r^{b}(t) = X(t), \quad 0 \leq t < T_b^+(1)
	\end{align}
	where
	\begin{align} T_b^+(1) := \inf\{T(i):\; X(T(i)) >b\}. \label{def_T_0_1}
	\end{align}
	The process then jumps downward by $X(T_b^+(1))-b$ so that $U_r^{b}(T_b^+(1)) = b$. For $T_b^+(1) \leq t < T_b^+(2)  := \inf\{ S \in  \mathcal{T}_r: S > T_b^+(1),  U^{b}_r( S-) > b\}$, we have $U_r^{b}(t) = X(t) - (X(T_b^+(1))-b)$.  The process $U_r^{b}$ can be constructed by repeating this procedure.

%		\begin{align*}
%			U_r^{b}(t) = X(t) - L_r^b(t), \quad t \geq 0,
%		\end{align*}
%		%	\red{JL: Maybe better to use $U_r^{b}$ instead of $X^{r,b}$ because $U$ is the controlled process?}
%		with
%		\begin{align}
%			L_r^b(t) := \sum_{T_b^+(i) \leq t} \left(U_r^{b}(T_b^+(i)-)-b\right), \quad t \geq 0, \label{def_L_r}
%		\end{align}
%		where 
%			\begin{align} T_b^+(1) := \inf\{T(i):\; X(T(i)) >b\}, \label{def_T_0_1}
%			\end{align}
%			and
%		$(T_b^+(n); n \geq 1)$ can be constructed inductively by \eqref{def_T_0_1} and
%		\begin{eqnarray*}T_b^+(n+1) := \inf\{T(i) > T_b^+(n):\; U_r^{b}(T(i)-) >b\}, \quad n \geq 1.
%		\end{eqnarray*}

%\red{[it was too simplified and so added some sentences back.]}

	%\red{JL: Similarly to $U_r^b$, how about writing $L_r^b$?}
	
	Suppose $L_r^b(t)$ is the cumulative amount of (Parisian) reflection until time $t \geq 0$. Then we have
	\begin{align*}
		U_r^{b}(t) = X(t) - L_r^b(t), \quad t \geq 0,
	\end{align*}
	%	\red{JL: Maybe better to use $U_r^{b}$ instead of $X^{r,b}$ because $U$ is the controlled process?}
	with
	\begin{align}
		L_r^b(t) = \sum_{T_b^+(i) \leq t} \left(U_r^{b}(T_b^+(i)-)-b\right), \quad t \geq 0, \label{def_L_r}
	\end{align}
	where $(T_b^+(n); n \geq 1)$ can be constructed inductively by using \eqref{def_T_0_1} and
	\begin{eqnarray*}T_b^+(n+1) := \inf\{ S \in  \mathcal{T}_r: S > T_b^+(n),  U_r^{b}(S-) >b\}, \quad n \geq 1.
	\end{eqnarray*}

		It is clear that the strategy $\pi^b := (L_r^b(t); t \geq 0 )$, for $b \geq 0$, is admissible with $\nu^{\pi^b}(t) = (U_r^b(t-) - b) \vee 0$. We denote its expected NPV of dividends by 
		\begin{align}\label{vf}
			v_b(x):=\E_x\Big(\int_{[0,\sigma_0^b]}e^{-qt}\diff L_r^b(t)\Big), \quad x \geq 0,
		\end{align}
		where 
		\begin{align*}
			\sigma_0^b := \inf \{ t > 0: U_r^b(t) < 0 \}.
		\end{align*}
		\subsection{Computation of the expected NPV \eqref{vf}}
		%\section{Review of Scale Functions and some Fluctuation Identities.}
		%\subsection{Review of scale functions}  \label{section_scale_functions}
		The expected NPV of dividends as in \eqref{vf} can be computed directly by using the fluctuation theory.  Toward this end, we first review the scale functions.
		
		%In this section we review here the scale function and its applications on the spectrally positive \lev process. As we need to deal with the fluctuation of the processes  $X$, we define the scale function here.
		Fix $q > 0$. We use $W^{(q)}: \R \to [0, \infty)$ for the scale function of the spectrally negative \lev process $X$, 
% This is the mapping from $\R$ to $[0, \infty)$ that 
which takes the value zero on the negative half-line, while on the positive half-line, it is a continuous and strictly increasing function such that
%that is defined by its Laplace transform:
		\begin{align} \label{scale_function_laplace}
			\begin{split}
				\int_0^\infty  \mathrm{e}^{-\theta x} W^{(q)}(x) \diff x &= \frac 1 {\psi(\theta)-q}, \quad \theta > \Phi(q),
			\end{split}
		\end{align}
		where $\psi$ is as defined in \eqref{lk} and
		\begin{align}
			\begin{split}
				\Phi(q) := \sup \{ \lambda \geq 0: \psi(\lambda) = q\} . 
			\end{split}
			\label{def_varphi}
		\end{align}
		%	In particular, when $q=0$, we shall drop the superscript. \red{[if we assume $q > 0$, we can delete this?]}
		%By the strict  convexity of $\psi_Y$, we derive the inequality $\varphi(q) > \Phi(q) > 0$ for $q > 0$ and  $\varphi(q) \geq \Phi(q) \geq 0$ for $q = 0$.
		We also define, for $x \in \R$, %\red{added $\overline{Z}$}%\red{JL: removed $\overline{Z}$ because it is not used anywhere for the SP case.}
		%\red{[added $\overline{\overline{W}}$ below]}
		\begin{align*}
			\overline{W}^{(q)}(x) &:=  \int_0^x W^{(q)}(y) \diff y, \\ 
			\overline{\overline{W}}^{(q)}(x) &:=  \int_0^x \overline{W}^{(q)}(y) \diff y, \\ 
			Z^{(q)}(x) &:= 1 + q \overline{W}^{(q)}(x), \\
			\overline{Z}^{(q)}(x) &:= \int_0^x Z^{(q)} (z) \diff z = x + q \int_0^x \int_0^z W^{(q)} (w) \diff w \diff z.
		\end{align*}
		Because $W^{(q)}(x) = 0$ for $-\infty < x < 0$, we have
		\begin{align}
			\overline{W}^{(q)}(x) = 0,\quad \overline{\overline{W}}^{(q)}(x) = 0,\quad Z^{(q)}(x) = 1,
			\quad \textrm{and} \quad \overline{Z}^{(q)}(x) = x, \quad x \leq 0.  \label{z_below_zero}
		\end{align}
		%\red{[JL: remoe the second one?  It is I guess not used anywhere.]}
		%\red{[moved the second line here]}
		%that can be obtained by performing an integration of identity (\ref{RLqp}).
		%	which can be proven by showing that the Laplace transforms on both sides are equal.
		
		%\red{[added this.  Used later.]
		%\red{[should we change from $\tau$ to $T$ or something as $\tau$ is used for the ruin time?]}\blue{[$T$ is related to the Poissonian reflection times, so maybe something else?]} \red{[How about $\sigma$ for $U^\pi$ and keep $\tau$ for $X$?]}
		If we define $\tau_0^- := \inf \left\{ t \geq 0: X(t) < 0 \right\}$ and $\tau_b^+ := \inf \left\{ t \geq 0: X(t) >  b \right\}$ for any $b > 0$,
		then
		\begin{align}
			\begin{split}
				\E_x \left( e^{-q \tau_b^+} 1_{\left\{ \tau_b^+ < \tau_0^- \right\}}\right) &= \frac {W^{(q)}(x)}  {W^{(q)}(b)}, \\
				 \E_x \left( e^{-q \tau_0^-} 1_{\left\{ \tau_b^+ > \tau_0^- \right\}}\right) &= Z^{(q)}(x) -  Z^{(q)}(b) \frac {W^{(q)}(x)}  {W^{(q)}(b)}.
			\end{split}
			\label{laplace_in_terms_of_z}
		\end{align}
		
		\begin{remark} \label{remark_smoothness_zero}
			%\begin{enumerate}
			%	\item If $X$ is of unbounded variation or the \lev measure is atomless, it is known that $W^{(q)}$ is $C^1(\R \backslash \{0\})$; see, e.g.,\ \cite[Theorem 3]{Chan2011}. \red{[remove this item? It is not needed for this paper. Only the continuity of $W$ will do.]}
			%\item 
			Regarding the asymptotic behaviors near zero, as in Lemmas 3.1 and 3.2 of \cite{KKR},
			\begin{align}\label{eq:Wqp0}
				\begin{split}
					W^{(q)} (0) &= \left\{ \begin{array}{ll} 0 & \textrm{if $X$ is of unbounded
							variation,} \\ \frac 1 {c} & \textrm{if $X$ is of bounded variation,}
					\end{array} \right. \\
					W^{(q)\prime} (0+) &:= \lim_{x \downarrow 0}W^{(q)\prime} (x) =
					\left\{ \begin{array}{ll}  \frac 2 {\eta^2} & \textrm{if }\eta > 0, \\
						\infty & \textrm{if }\eta = 0 \; \textrm{and} \; \Pi(-\infty,0)= \infty, \\
						\frac {q + \Pi(-\infty,0)} {c^2} &  \textrm{if }\eta = 0 \; \textrm{and} \; \Pi(-\infty, 0) < \infty.
					\end{array} \right.
				\end{split}
			\end{align}
			%On the other hand,
			In addition, according to
			 Lemma 3.3 of \cite{KKR},
			\begin{align}
				\begin{split}
					e^{-\Phi(q) x}W^{(q)} (x) \nearrow \psi'(\Phi(q))^{-1}, \quad \textrm{as } x \uparrow \infty.
				\end{split}
				\label{W_q_limit}
			\end{align}
		\end{remark}
		
		For the expression of \eqref{vf}, we also use the scale function $W^{(q+r)}$ and $\Phi(q+r)$  defined by \eqref{scale_function_laplace} and \eqref{def_varphi}, respectively, with $q$ replaced with $q+r$.  Note that 
		\begin{align}
		\Phi(q+r) > \Phi(q), \quad r > 0 \label{big_phi_monotonicity}
		\end{align}
		 and  from identity (5) in \cite{LRZ} 
	\begin{align}
		W^{(q+r)}(x)-W^{(q)}(x)=r\int_0^xW^{(q+r)}(u)W^{(q)}(x-u) \diff u, \quad x \in \R. \label{W_q_W_q_r_relation}
	\end{align}
		
		We also define, for $q, r > 0$ and $x \in \R$,
		\begin{align}
			Z^{(q)}(x,\Phi(q+r)) &:=e^{\Phi(q+r) x} \left( 1 -r \int_0^{x} e^{-\Phi(q+r) z} W^{(q)}(z) \diff z	\right) \notag \\
			&=r \int_0^{\infty} e^{-\Phi(q+r) z} W^{(q)}(z+x) \diff z	 > 0, \label{z1}\end{align}
			where the second equality holds because \eqref{scale_function_laplace} gives $\int_0^\infty  \mathrm{e}^{-\Phi(q+r) x} W^{(q)}(x) \diff x = r^{-1}$.
			By differentiating this with respect to the first argument,
					\begin{align}
			Z^{(q) \prime}(x,\Phi(q+r)) &:= \frac \partial {\partial x}Z^{(q)}(x,\Phi(q+r))  = \Phi(q+r) Z^{(q)}(x,\Phi(q+r))	- r W^{(q)}(x), \quad x > 0. \label{Z_prime_Phi}
			\end{align}

		%	\red{[added the middle equality. We can delete the reason below?]}
		%where the positivity holds because, by \eqref{scale_function_laplace},
		%{\color{blue}[Kazu: Following your suggestion I changed $Z^{(q,r)}(x)$ to $J^{(q,r)}(x)$, is this ok?]}
		%\begin{align*}
		%	r\int_0^xe^{-\Phi(q+r)z}W^{(q)}(z)\diff z<r\int_0^{\infty}e^{-\Phi(q+r)z}W^{(q)}(z) \diff z=\frac{r}{(q+r)-q}=1,
	%	\end{align*}
		%and define
		%\red{JL: Should we just define ${\color{blue}{\color{blue}J^{(q,r)}(x)}}$ as the general case with $\theta$ is not used anywhere? In that case, we may just use something like $Z^{(q)} (x, r)$? It is just that ${\color{blue}{\color{blue}J^{(q,r)}(x)}}$ is too long and  hard to see?}
	%	\begin{align}
		%	Z^{(q,r)}(x) &:=\frac{r}{r+q} Z^{(q)}(x)
		%	+\frac{q}{r+q} J^{(q,r)}(x) \label{z2}.
	%	\end{align}
		%\red{[added this] 
	%	Note that
	%	\begin{align}
		%	Z^{(q,r)\prime}(x) = \frac {q} {r+q} \Phi(q+r) J^{(q,r)}(x), \quad x \in \R. \label{Z_q_r_der}
	%	\end{align}%}
	%	\red{[I guess $Z^{(q,r)}$ and $J$ are not used in this paper. Delete?]}
		
		%\red{[JL: I just realized that the derivative of $Z^{(q,r)}$ appears and so it makes more sense to use $Z^{(q,r) \prime}$.  Sorry, is it ok to change back instead of $\frac {q} {r+q} \Phi(q+r) Z^{(q)}(x,r)$ below and use \eqref{Z_q_r_der} when necessary?]}
		
		%\red{JL: I made this a lemma.}
		
		%\red{[simply say the results are immediate application of  Corollary 3.1 (ii) in \cite{PYM} and delete the following paragraph?]}
		The following results, related to the computation of the expected NPV under a periodic barrier strategy at the level $b$, are immediate applications of  Corollary 3.1 (ii) in \cite{PYM}. Note that only the case $b > 0$ is covered in Corollary 3.1 (ii) in \cite{PYM}, but can be extended to the case $b=0$ by monotone convergence by taking a decreasing sequence of down-crossing times.
	%	The expected NPV \eqref{vf} can be written concisely by the scale functions defined above.  All the fluctuation identities required here are essentially computed in \cite{APY}; they studied the spectrally negative \lev case where it is reflected from below at Poisson arrival times, and also its variation with additional classical reflection from above. Our processes $U_r^b$ and $U^{0,b}_r$ are the dual of these processes and hence their results can be directly used.
		%Now we will use the results in \cite{APY} to give some of the fluctuation identities that will be used throughout this work. 
		\begin{lemma} For all $b \geq 0$ and $x \in \R$, %\red{[rewrote to avoid the division by $W(b)$ which can be zero when $b = 0$]}
			\begin{align}\label{vf3}
			\begin{split}
				v_b(x)&=r \frac{ W^{(q)}(x)+r\int_0^{x-b}W^{(q+r)}(x-b-y)W^{(q)}(y+b) \diff y -r  W^{(q)}(b)  \overline{W}^{(q+r)}(x-b) }{\Phi(q+r)Z^{(q)\prime}(b,\Phi(q+r))} \\ &-r\overline{\overline{W}}^{(q+r)}(x-b).
				\end{split}
				%\end{align}
				%and
				%\begin{align}
			\end{align}
			
		%	\red{JL: Should we just write here 
		%	$\int_0^{x-b } W^{(q + r)} (y) W^{(q)}(x - y ) \diff y$ instead of $\int_0^{x-b } W^{(q + r)} (x- b - y) W^{(q)}(y +b ) \diff y$?}
			
			%\red{JL: Should we just write $h'$ as $r Z^{(q)} (x) / (r+q)$?}
		%	where %\red{[the reader may wonder why there should be tilde. How about using something like $H^{(q,r)}$?]}
		%	\begin{align*}
		%		W^{(q,r)}_{-b}(y)&:=W^{(q)}(y+b)+r\int_0^yW^{(q+r)}(x-y)W^{(q)}(y+b) \diff y, \quad y \in \R. %\\
				%Z^{(q,r)}(x)&=\frac{r}{(r+q)}Z^{(q)}(x)+\frac{q}{(q+r)}{\color{blue}{\color{blue}J^{(q,r)}(x)}}.\\
				%{\color{blue}{\color{blue}J^{(q,r)}(x)}}&=e^{\Phi(q+r)x}\left(1-r\int_0^xe^{-\Phi(q+r)z}W^{(q)}(z)dz\right).
			%\end{align*}
			%\red{JL: If $W^{(q,r)}_{-b}(y)$ does not appear so frequently, I guess we can avoid using this short-hand notation? If the reader is not familiar with \cite{PYM}, the minus sign for $b$ may look odd (and may think that it is a typoe).}
			
			%\red{[JL: It is better to write $\frac q {r+q} \Phi(q+r) Z^{(q)}(x,r)$ instead of $Z^{(q,r)'}$?]}
		\end{lemma}

		It is noted that the expression \eqref{vf3} also hold for $-\infty < x \leq b$ with %\red{[simplified the expression a bit below]}
		\begin{align*}
			v_b(x) &=\frac{r}{\Phi(q+r)}\frac{W^{(q)}(x)}{Z^{(q)\prime}(b,\Phi(q+r))}. 
			%u_b(x) &=(r+q)^{-1} \Big[\frac{1}{Z^{(q,r) \prime}(b)}\left(\frac{r Z^{(q)} (b)}{ r+q}-\beta\right)\Big( r
			%+q e^{(b-x)\Phi(q+r)} \Big) -r\left(b-x+\frac{\psi'(0+)}{q}\right) \Big].
			\end{align*}
			Moreover, for the case $b = 0$, by \eqref{W_q_W_q_r_relation}
			%, \blue{[remove the previous comma?]} 
			the expression \eqref{vf3} is simplified as follows:
			\begin{align} \label{vf0}
				v_0(x)&=\frac{r \big( W^{(q+r)}(x)-r  W^{(q)}(0)  \overline{W}^{(q+r)}(x) \big)}{\Phi(q+r)(\Phi(q+r)-r W^{(q)}(0)) }-r\overline{\overline{W}}^{(q+r)}(x).
			\end{align}
			
\subsection{Completely monotone case} %\red{[added this subsection]}
In the remainder of the paper, we assume the following.
%the completely monotone assumption given as follows.
\begin{assump} \label{assump_completely_monotone} %\green{[Shortened OK?]} 
The \lev measure $\overline{\Pi}$ of the dual process $-X$ has a completely monotone density.  That is, $\overline{\Pi}$ has a density $\pi$ whose $n^{th}$ derivative $ \pi^{(n)}$ exists for all $n \geq 1$ and satisfies
\begin{align*}
(-1)^n \pi^{(n)} (x) \geq 0, \quad x > 0.
\end{align*}
\end{assump}
This assumption is known to be a sufficient condition of optimality for the classical spectrally negative case by Loeffen \cite{L} and for the absolutely continuous case by Loeffen et al. \cite{KLP}.

%that the \lev measure The class of hyperexponential distributions is important as it is dense in the class of all positive-valued distributions with completely monotone densities.  %We refer the reader to \cite{Albrecher_2010, Feldmann_1998, Kammler_1976} for approximation methods.

%In this case, it is known that the scale function $W^{(q)}$ satisfies the following.
\begin{remark} \label{remark_scale_function_completely_monotone}
Under Assumption \ref{assump_completely_monotone}, we have the following.\begin{enumerate}
\item  As in Theorem 2 of \cite{L2}, 
the scale function $W^{(q)}(x)$ is infinitely differentiable and can be written as \[
	W^{(q)}(x)=\Phi'(q)e^{\Phi(q)x}-\int_0^\infty e^{-x t} \mu^{(q)}(\diff t),\qquad x>0,
	\]
	%where $f(x) := \int_0^\infty e^{-xs} \mu(\diff x)$ 
	for some finite measure $\mu^{(q)}$. 
	%\red{[how about changing to $\mu^{(q)}$ here? the one we are actually using later is $\mu^{(q+r)}$]}\blue{I agree Kazu.} %\red{[JL: Maybe we can just use $\mu$ and avoid using $f$?]}
\item As in the proof of Theorem 3 of \cite{L}, we have $W^{(q)\prime \prime \prime}(x) > 0$ for all $x > 0$, and hence there exists $\bar{b} \in [0, \infty)$
% \red{[let's check if it can be $0$ or $\infty$. Probably it can be zero, but it cannot be $\infty$ (the scale function increases exponential in the tail)]}\blue{Agreed.} 
 such that $W^{(q)\prime \prime} < 0$ on $(0, \bar{b})$ and $W^{(q)\prime \prime} > 0$ on $(\bar{b}, \infty)$. 
\item As shown in Loeffen \cite{L}, the optimal solution for the classical case is to reflect (in the classical sense) from above at $\bar{b}$; the value function is given by
	\begin{equation} \label{classical_value_function}
		\bar{v}(x) := \sup_{\pi \in \mathcal{A}_\infty}\mathbb{E}_x \Big( \int_{[0,\sigma_0^{\pi}]} e^{-q t} \diff L^{\pi}(t)\Big) = \left\{ \begin{array}{ll}\frac {W^{(q)}(x)} {W^{(q) \prime}(\bar{b})} &  x \leq \bar{b}, \\ \frac {W^{(q)}(\bar{b})} {W^{(q) \prime}(\bar{b})} + x-\bar{b} & x > \bar{b},  \end{array} \right.
	\end{equation}
where $\mathcal{A}_\infty$ is the set of  nondecreasing, right-continuous, and $\mathbb{F}$-adapted processes, as a relaxation of $\mathcal{A}_r$.%\blue{[Kazu: Perhaps say what the optimal threshold $\bar{b}$ is?]} \red{[defined in (1) above?]}

\end{enumerate} 
\end{remark}
			
	\section{Selection of optimal barrier $b^*$}\label{sf_d} %\red{[I changed it so that we refrain from using $b^*$ until we define as the candidate threshold. OK?]}

In this section, we focus on the above-mentioned periodic barrier strategy and choose a candidate barrier $b^*$, and show its existence.
	
	\subsection{Smooth fit}
	Motivated by many relevant studies, we choose the barrier $b^*$ so that \emph{the degree of smoothness of $v_{b^*}$ at $b^*$ increases by one} (if $b^* > 0$).  Unlike the classical model in \cite{APP2007} and \cite{L}, we see in our model that $v_{b^*}$ becomes $C^2 (0, \infty)$ (resp.\ $C^3 (0, \infty)$) for the case $X$ is of bounded (resp.\ unbounded) variation. 
	
	%\red{[moved here and named the condition. OK?] 
	Here, we shall show that the desired smoothness at $b > 0$ is satisfied on condition that 
	\begin{equation}\label{opt_thres}
	\mathfrak{C}_b: W^{(q)\prime}(b)=\frac{\Phi(q+r)}{r}Z^{(q)\prime}(b,\Phi(q+r)),
	\end{equation}
	where $Z^{(q)\prime}(b,\Phi(q+r))$ is given as in \eqref{Z_prime_Phi}.
	%holds for $b^* = b$.
	
	%\red{[JL: I am thinking we should just add Kei's Lemma 0.3 here. His differentiation is probably more careful. ] }\blue{[Kazu: Tried to adapt Kei's Lemma to the paper.]}
To this end, we first compute the derivatives of \eqref{vf3}. Recall the smoothness of the scale function as in Remark \ref{remark_scale_function_completely_monotone}.
	\begin{lemma}
		For $b > 0$ and  $x \in (0, \infty)\backslash \{b\}$,  %\green{[deleted the counters (i), (ii), (iii)]}%\blue {[OK]}
		\begin{align}
			%\frac{\partial}{\partial x} v_b(x)
			v_b^{\prime}(x)
			&= r\left(\frac{W^{(q) \prime }(x)
				+r \int_0^{x- b} W^{(q+r)}(y)W^{(q)\prime} (x- y) \diff y}{ \Phi(q+ r)   Z^{(q) \prime}(b, \Phi(q + r))}  
			-{\overline{W}}^{(q+r)}(x-b)\right), \label{015} \\
			%\frac{\partial^2}{\partial x^2} v_b(x)
			v_b^{\prime \prime}(x)
			&= r\left(\frac{W^{(q) \prime \prime} (x )
				+r W^{(q + r)}(x - b)W^{(q) \prime}(b) 
				+ r\int_0^{x- b} W^{(q+r)}(y)W^{(q)\prime \prime} (x- y) \diff y }{ \Phi(q+ r)   Z^{(q) \prime}(b, \Phi(q + r))}  
			-{W}^{(q+r)}(x-b) \right) \label{016}, \\
			%\frac{\partial^3}{\partial x^3} v_b(x)
			v_b^{\prime \prime \prime}(x)
			&=\frac{r}{{ \Phi(q+ r)   Z^{(q) \prime}(b, \Phi(q + r))}}\Big{(}
			W^{(q) \prime \prime \prime} (x)
			+r W^{(q + r)\prime}(x - b)W^{(q) \prime}(b) \label{018} \\
			&+rW^{(q+r)}(x-b)W^{(q) \prime \prime} (b) 
			+ r\int_0^{x- b} W^{(q+r)}(y)W^{(q)\prime \prime \prime} (x- y) \diff y\Big{)}
			-rW^{(q+r) \prime}(x - b). \notag
		\end{align}
	\end{lemma}
	\begin{proof}
	%\red{[JL: OK to write (i) etc to be consistent with other parts?]}
		(i)  %\red{[ok to remove this following sentence?]}First we will compute $v_b'(x)$. \red{[shortened here]} 
		By $\int_0^{x-b } W^{(q + r)} (x- b - y) W^{(q)}(y +b ) \diff y
			= \int_0^{x-b } W^{(q + r)} (y) W^{(q)}(x - y ) \diff y$ and dominated convergence,
%		\begin{align*}
%			%W_{-b}^{(q, r)} (x-b)
%			%&= 
%			W^{(q)} (x ) + r \int_0^{x-b } &W^{(q + r)} (x- b - y) W^{(q)}(y +b ) \diff y\\
%			&=W^{(q)} (x ) + r \int_0^{x-b } W^{(q + r)} (y) W^{(q)}(x - y ) dy
%		\end{align*}
%		then by dominated convergence, 
		we have 
		\begin{align}  \label{021}
		\begin{split}
			\frac{\partial }{ \partial x}&\left(W^{(q)} (x ) + r \int_0^{x-b } W^{(q + r)} (y) W^{(q)}(x - y ) \diff y -r  W^{(q)}(b)  \overline{W}^{(q+r)}(x-b)\right) \\
			&=W^{(q) \prime } (x ) + 
			r \lim_{\epsilon \rightarrow 0 }
			%\frac{\int_0^{x+ \epsilon -b } W^{(q + r)} (x+\epsilon - b - y) W^{(q)}(y +b ) dy
			%	- \int_0^{x-b } W^{(q + r)} (x- b - y) W^{(q)}(y +b ) dy}{\epsilon} \\
			%&=W^{(q) \prime } (x )
			%+ rW^{(q + r)}(0)W^{(q)}(x)
			%+ r\int_0^{x-b} W^{(q+r) \prime} (x-b-y)W^{(q)}(y+b) dy. 
			\frac{\int_0^{x+\epsilon-b } W^{(q + r)} (y) W^{(q)}(x + \epsilon- y ) \diff y 
				-\int_0^{x-b } W^{(q + r)} (y) W^{(q)}(x - y ) \diff y}{\epsilon} \\
				&-rW^{(q + r)}(x - b)W^{(q)}(b) 
				\\&=W^{(q) \prime } (x )
			+r \int_0^{x- b} W^{(q+r)}(y)W^{(q)\prime} (x- y) \diff y.
			\end{split}
		\end{align}
		Hence, by differentiating \eqref{vf3}, we obtain \eqref{015}. 
%		[and remove the rest?]}
%		So using the above identity and \eqref{vf3} we obtain
%		\begin{align}
%			\frac{\partial}{\partial x}&\left(W^{(q)}(x)+r\int_0^{x-b}W^{(q+r)}(x-b-y)W^{(q)}(y+b) \diff y -r  W^{(q)}(b)  \overline{W}^{(q+r)}(x-b)\right)\notag\\
%		%	&=\eqref{019}-rW^{(q+r)}(x-b)W^{(q)}(b) \\
%			&=W^{(q) \prime }(x)
%			+r \int_0^{x- b} W^{(q+r)}(y)W^{(q)\prime} (x- y) dy . \label{021}
%		\end{align}
%		So by differentiating \eqref{vf3} and using \eqref{021} we obtain \eqref{015}. 
%		\par

		(ii) %Now we will compute $v_b''(x)$. 
		By differentiating \eqref{021},
		the dominated convergence theorem gives
		\begin{align}
\begin{split}
			\frac{\partial^2}{\partial x^2}&\left(W^{(q)}(x)+r\int_0^{x-b}W^{(q+r)}(x-b-y)W^{(q)}(y+b) \diff y -r  W^{(q)}(b)  \overline{W}^{(q+r)}(x-b)\right) \\
			&= W^{(q) \prime \prime} (x )
			+r W^{(q + r)}(x - b)W^{(q) \prime}(b) 
			+ r\int_0^{x- b} W^{(q+r)}(y)W^{(q)\prime \prime} (x- y) \diff y.
\end{split} \label{028}
		\end{align}
	%	We have 
	%	\begin{align}
	%		{\overline{W}}^{(q+r) \prime} (x-b)= W^{(q+r)}(x- b), 
	%	\end{align}
		Hence, by differentiating \eqref{015} and using \eqref{028}, we obtain \eqref{016}. \par
		(iii) %Finally we will compute $v_b'''(x)$. 
		By further differentiating \eqref{028}, the dominated convergence theorem gives 
		\begin{align*}
			\frac{\partial^3}{\partial x^3}&\left(W^{(q)}(x)+r\int_0^{x-b}W^{(q+r)}(x-b-y)W^{(q)}(y+b) \diff y -r  W^{(q)}(b)  \overline{W}^{(q+r)}(x-b)\right)\notag\\
			&= W^{(q) \prime \prime \prime} (x )
			+r W^{(q + r)\prime}(x - b)W^{(q) \prime}(b) 
			+rW^{(q+r)}(x-b)W^{(q) \prime \prime} (b) \notag\\
			&~~+ r\int_0^{x- b} W^{(q+r)}(y)W^{(q)\prime \prime \prime} (x- y) \diff y . 
		\end{align*}
		Hence, by using the aforementioned identity and by differentiating \eqref{016}, we obtain \eqref{018}. 
	\end{proof}
	%In this section we will obtain the candidate $b^*$ for the optimal threshold. In order to do it we will look for smoothness of the value function. To this end we note that for $x<b$
	
By the smoothness of the scale function on $\R \backslash \{0\}$ as in Remark \ref{remark_scale_function_completely_monotone}, the derivatives \eqref{015}, \eqref{016}, and \eqref{018} are continuous on $(0, \infty) \backslash \{b\}$. Now, we analyze their continuity at $b$ for the cases of bounded and unbounded variation.
% \blue{[variation? for example see 236 of \cite{K}]}.  
 Recall the behaviors of the scale function around zero as in Remark \ref{remark_smoothness_zero}. 
%\green{[moved the following here because it holds for both unbounded/bounded var cases.]} 
Based on \eqref{vf3} and \eqref{015} we determine that $v_b$ and $v_b^{\prime}$ are continuous functions on $(0 , \infty)$ regardless of the choice of $b$.  In addition, we note that
\begin{align}\label{cond_smooth_1}
	v_b''(b+)-v_b''(b-)
	=rW^{(q + r)}(0) \left(\frac{r W^{(q) \prime} (b)}{{ \Phi(q+ r)   Z^{(q) \prime}(b, \Phi(q + r))}} -1  \right). 
\end{align}

a) Let us assume that $X$ has bounded variation paths. 
By \eqref{cond_smooth_1} and the fact that $W^{(q+r)}(0)>0$ as in \eqref{eq:Wqp0}, $v''_b$ is continuous on $(0 , \infty)$ if and only if $\mathfrak{C}_b$ in \eqref{opt_thres} holds.
\par
b) On the other hand  let us assume that $X$ has unbounded variation paths. 
Then using, in \eqref{cond_smooth_1},  the fact that $W^{(q+r)}(0)=0$ as in \eqref{eq:Wqp0}, 
%we have that 
$v''_b$
is continuous on $(0 , \infty)$ regardless of the choice of $b$. %[Remove the rest, to avoid using $W^{(q+r)\prime}(0+)$?]}On the other hand, 
%\begin{align*}
%	v_b'''(b+)-v_b'''(b-)
%	=rW^{(q+r) \prime}(0+) \left( \frac{rW^{(q) \prime}(b)}{ \Phi(q+ r)   Z^{(q) \prime}(b, \Phi(q + r))}  -1\right). 
%\end{align*}
%Therefore $v'''_b$ is continuous on $(0 , \infty)$ if and only if the condition $\mathfrak{C}_b$ holds.
%\red{JL: The continuity of $v'''$ still holds and so we should not delete this part completely.  But we don't have to say ``if and only if". How about this? ``
In addition,  if $\mathfrak{C}_b$ is satisfied, we note by \eqref{018} that
\begin{align*}
	v_b^{\prime \prime \prime}(x)
			&=\frac{r}{{ \Phi(q+ r)   Z^{(q) \prime}(b, \Phi(q + r))}}\Big{(}
			W^{(q) \prime \prime \prime} (x) 
			\\&+rW^{(q+r)}(x-b)W^{(q) \prime \prime} (b)+ r\int_0^{x- b} W^{(q+r)}(y)W^{(q)\prime \prime \prime} (x- y) \diff y\Big{)},
			\end{align*}
			and hence $v'''_b$ is continuous on $(0 , \infty)$.
	%\red{[I added a lemma below. ] 
%	We shall now summarize the results obtained above.
In the following, we summarize the obtained  results.

	\begin{lemma} \label{smooth_fit_prob1}Suppose $b > 0$ satisfies $\mathfrak{C}_b$ in \eqref{opt_thres}. Then, $v_{b}$ is $C^2 (0, \infty)$ for the case $X$ is of bounded variation, while it is $C^3 (0, \infty)$ for the case $X$ is of unbounded variation.
	\end{lemma}
	
	\begin{remark} The verification lemma (Lemma \ref{verificationlemma}) requires only the $C^1 (0, \infty)$ and $C^2 (0, \infty)$ conditions for the cases of bounded and unbounded 
	variation, respectively. The extra smoothness obtained in Lemma \ref{smooth_fit_prob1} will not be directly used for the proof of optimality.  
%However, the barrier selected by this smoothness criteria will be shown in Proposition  \ref{prop_slope_condition} that the desired slope condition of the value function is satisfied.
		 However, it will be shown in Proposition  \ref{prop_slope_condition}, that the barrier selected by this smoothness criteria satisfies the desired slope condition of the value function.
	%The smoothness in Lemma  \ref{smooth_fit_prob1} is more than necessa
%	 The smoothness attained here is more than necessary for the verification lemma which only requires $C^2$ and $C^1$ for the unbounded and bounded variation cases, respectively.  However, it is indeed a natural approach to attain a candidate barrier $b^*$; it is analogous to the classical smooth fit principle where the barrier is chosen so that the value function gets smoother. 
	\end{remark}

\subsection{Existence of $b^*$} 

Here we pursue the existence of $b^* > 0$ such that the condition $\mathfrak{C}_b$ for $b = b^*$ holds.

%\red{[I guess we can just define $e^{-\Phi (q + r) b} g(b)$ as our current $g(b)$ is not used? We are using $g$ later and so how about using $h$?]}

\par Define, for $b > 0$,
\begin{align} \label{h_b_def}
\begin{split}
	h(b) &:=  e^{-\Phi(q+r) b}\left[r W^{(q) \prime} (b) - \Phi(q+ r)   Z^{(q) \prime}(b, \Phi(q + r)) \right] \\
	&=e^{-\Phi(q+r) b}\left[r W^{(q) \prime} (b)- \Phi(q+ r)\left\{ \Phi(q+ r)   Z^{(q)}(b, \Phi(q + r))
		- r W^{(q) } (b)\right\}\right], 
		\end{split}
\end{align}
%\red{How about just defining
%\begin{align*}
%	h(b) &:= e^{-\Phi(q+r) b} \Big( r W^{(q) \prime} (b) - \Phi(q+ r)   Z^{(q) \prime}(b, \Phi(q + r)) \Big)?
%\end{align*}
%}
with its initial value %\red{[simplified below]}
\begin{align*}
h(0+)&=r [W^{(q)\prime}(0+)+ \Phi(q+r)W^{(q)}(0)]-\Phi^2(q+r).
\end{align*}
It is clear that $\mathfrak{C}_b$ is satisfied if and only if $h(b) = 0$.

By differentiating \eqref{h_b_def}, we get
%\begin{align}
%	\blue{h}^\prime(b)=r W^{(q) \prime  \prime} (b) + \Phi (q + r) g(b)
%\end{align}
%and 
\begin{align}
	h'(b)
	= r e^{-\Phi (q+r)b} W^{(q) \prime  \prime} (b), \quad b > 0. \label{038}
\end{align}
In addition, by \eqref{W_q_limit}, \eqref{big_phi_monotonicity}, and \eqref{z1} and because $\int_0^{b} e^{-\Phi(q+r) z} W^{(q)}(z) \diff z \xrightarrow{b \uparrow \infty} \int_0^{\infty} e^{-\Phi(q+r) z} W^{(q)}(z) \diff z = r^{-1}$ in view of \eqref{scale_function_laplace}, 
\begin{align}\label{ot_cond}
\begin{split}
	h(b) 
	&=r e^{-\Phi (q + r) b}  W^{(q) \prime} (b)- \Phi(q+ r)\left(\Phi(q+ r)   \left( 1 -r \int_0^{b} e^{-\Phi(q+r) z} W^{(q)}(z) \diff z	\right)  
		- r e^{-\Phi (q + r) b} W^{(q) } (b)\right) \\
		&\xrightarrow{b \uparrow \infty} 0.
		\end{split}
\end{align}

%By \eqref{038}, we can write
Hence, we can write
%\blue{[Kazu: Remove the following equality?]}
%\begin{align*}
%\blue{h}(b) = -r \int_0^\infty e^{-\Phi(q+r) y} W^{(q)\prime \prime} (y+b) \diff y.
%\end{align*}
%This implies that
\begin{equation}
h(b)=-r\int_b^\infty e^{-\Phi(q+r) y} W^{(q)\prime \prime} (y) \diff y, \quad b > 0.\label{039}
\end{equation}
%It is known for the case of completely monotone \lev density, for some $\bar{b}$ (I don't know if it can be $0$ or $\infty$),  $W^{(q)\prime \prime} \leq 0$ on $(0, \bar{b})$ and $W^{(q)\prime \prime} \geq 0$ on $(\bar{b}, \infty)$.
Based on \eqref{038} and \eqref{039} and Remark \ref{remark_scale_function_completely_monotone} (2), the function %$b \mapsto e^{\Phi(q+r)b}h(b)$ 
$h$
decreases on $(0, \bar{b}$) with $h(\bar{b})<0$. It then increases on $(\bar{b},\infty)$ and converges to $0$. (See the plots in Figure \ref{figure_h} in Section \ref{section_numerics}.)

%In addition, we can show easily that $e^{-\Phi(q+r) b} g(b) \xrightarrow{b \rightarrow \infty} 0$. 

The above argument and the continuity of $h(b)$ imply that there exists $0 < b^*<\bar{b}$ such that $h(b^*)=0$ (or equivalently $\mathfrak{C}_b$ for $b = b^*$) if and only if 
\begin{align} \label{b_zero_criteria}
h(0+)>0 \Longleftrightarrow r \Big(W^{(q)\prime}(0 +)+\Phi(q+r) W^{(q)}(0) \Big)>\Phi^2(q+r).
\end{align}
For the case $h(0+) \leq 0$, we set $b^* = 0$.

We summarize the results in the following proposition.
\begin{proposition} \label{prop_b_existence}
(i) If \eqref{b_zero_criteria} holds, then  there exists 
\begin{align}
0 < b^* < \bar{b} \label{b_star_b_bar}
\end{align}
 such that $\mathfrak{C}_b$ for $b = b^*$ is satisfied and $h(b) \geq 0$ if and only if $b \in (0,b^*]$.
 
(ii) Otherwise, $h(b) \leq 0$ for all $b > 0$.
\end{proposition}

\begin{remark} \label{remark_cond_b_zero}
%\red{[made this a  remark]} \red{[JL: I guess it is more succinct to write for the condition $b^*=0$ instead?]}
	%\blue{Kazu: Changed condition to $b^*=0$ as you suggested.}
By using \eqref{eq:Wqp0}, $b^* = 0$ if and only if one of the following holds:
\begin{itemize}
	%\item[(i)] $\sigma=0$ and $\Pi(-\infty,0)=\infty$,
	\item[(i)] $\eta>0$ and $\displaystyle r\leq \frac{\eta^2}{2}\Phi^2(q+r)$ or
	\item[(ii)] $\eta=0$, $\Pi(-\infty,0)<\infty$, and $\displaystyle\Phi^2(q+r)\geq r\left(\frac{q+\Pi(-\infty,0)}{c^2}+\frac{\Phi(q+r)}{c}\right)$.
	\end{itemize}
\end{remark}
	
%	\red{[JL: Just to confirm, I added this argument.]}
\begin{remark}  %\red{[JL: Just to confirm, I added this argument.]} 
	In view of (i) of Remark \ref{remark_cond_b_zero} and by using \eqref{lk},
\begin{align*}
r - \frac{\eta^2}{2}\Phi^2(q+r) = -q + \Phi(q+r) \gamma  +\int_{(-\infty,0)} (e^{\Phi(q+r) z}-1-\Phi(q+r) z 1_{\{z > -1\}}) \Pi(\diff z).
\end{align*}
%\red{[Just realized I made a mistake here.  I fixed as follows.]}
For the case $\eta > 0$, we have $\Phi(q+r) \sim \sqrt{2r} /\eta$ as $r \rightarrow \infty$.  Hence,
\begin{align*}
\frac {\eta} {\sqrt{2r}} \Big(r - \frac{\eta^2}{2}\Phi^2(q+r) \Big) &= \frac {\eta} {\sqrt{2r}}  \Big(-q + \Phi(q+r) \gamma +\int_{(-\infty,0)} (e^{\Phi(q+r) z}-1-\Phi(q+r) z 1_{\{z > -1\}}) \Pi(\diff z)\Big) \\
 &\xrightarrow{r \uparrow \infty}  \gamma -\int_{(-\infty,0)} z1_{\{z > -1\}} \Pi(\diff z)=c.
\end{align*}
Hence, we obtain  $b^*>0$ for a large enough $r$, if $c>0$. 
%Therefore the criteria reduces to 
This is consistent with 
the classical case as given in Theorem 1 in \cite{L2}. %\red{[JL: Theorem 1 in \cite{L2} does not exist?]}
%\blue{[Kazu: In page 87 of \cite{L2} I found Theorem 1.]}
%\red{This is consistent with the case of unbounded variation case with $\eta = 0$ where $c=\infty$, in which case $b^* > 0$ all the time.} \blue{Kazu: In which part of the above argument are you taking $\eta=0$?} \red{[I guess what I said is confusing.  Just want to make sure the case $\sigma > 0$ and the case of unbounded variation jump is not so different.]}
%Hence the criteria reduces to the classical case for the case $\sigma > 0$ \blue{[Kazu: Maybe provide reference for the classical case?]}. \red{[actually is this a classical result? In the dual model, we always get this criteria. I will look into more references in primal models.]}
\end{remark}

%\red{JL: I just realized that the definition of the taking all money and run strategy in \cite{L2} is slightly different (for the case of bounded variation).  It is not exactly the strategy $\pi_0$.  In \cite{L2}, with sufficient terminal payoff, one wants to liquidate as quickly as possible and so one basically reduces the process to strictly below zero.  So, I guess throughout the paper, we should avoid using the term ``taking all money and run"?  }
%\blue{[Added this.]

	%\section{Some computations related to verification.}
	\section{Verification of optimality}
	\label{exist_thres}
	%\red{[moved this here]} 
	With $b^* \geq 0$ defined above, we now show the optimality of the obtained periodic barrier strategy %reflection strategy at Poissonian dividend decision times 
	$\pi^{b^*}$.
	
	For the case $b^* > 0$, because $\mathfrak{C}_b$ holds for $b = b^*$, the expected NPV \eqref{vf3} can be succinctly written as
	\begin{align}\label{vf5}
	\begin{split}
		v_{b^*}(x)&=\frac{W^{(q)}(x)+r\int_0^{x-b^*}W^{(q+r)}(x-b^*-y)W^{(q)}(y+b^*) \diff y -rW^{(q)}(b^*)\overline{W}^{(q+r)}(x-b^*)}{W^{(q)\prime}(b^*)} \\&-r\overline{\overline{W}}^{(q+r)}(x-b^*),\qquad\text{for $x \in \R$},
		\end{split}
	\end{align}
where, in particular, for $x<b^*$,
	\begin{align}
	v_{b^*}(x)=\frac{W^{(q)}(x)}{W^{(q)\prime}(b^*)}. \label{value_function_below_b_star}
	\end{align}
	
	In contrast, for the case $b^*=0$, the expected NPV is given by \eqref{vf0}.
%	we have by \eqref{vf3} that, for $x \in \R$,
%	\begin{align}\label{vf0}
%	\begin{split}
%	v_0(x)%&=\frac{r}{\Phi(q+r)Z^{(q)\prime}(0,\Phi(q+r))}\left(W^{(q,r)}_{0}(x)-r\overline{W}^{(q+r)}(x)W^{(q)}(0)\right)-r\overline{\overline{W}}^{(q+r)}(x) \\
%	=\frac{r}{\Phi(q+r)Z^{(q)\prime}(0,\Phi(q+r))}\left(W^{(q+r)}(x)-r\overline{W}^{(q+r)}(x)W^{(q)}(0)\right)-r\overline{\overline{W}}^{(q+r)}(x). \end{split}
%	\end{align}
%	where the equality follows because identity (5) in \cite{LRZ} gives %\red{[let's add a reference]}
%	\[
%	W^{(q)}(x)+r\int_0^{x}W^{(q+r)}(x-y)W^{(q)}(y) \diff y=W^{(q+r)}(x).
%	\]

	\subsection{Verification lemma}
	
		Let $\mathcal{L}$ be the infinitesimal generator associated with
	the process $X$ applied to a $C^1$ (resp.\ $C^2$) function $f$ for the case $X$ is of bounded (resp.\ unbounded) variation:
	\begin{align} \label{generator}
		\mathcal{L} f(x) &:= \gamma f'(x) + \frac 1 2 \eta^2 f''(x) + \int_{(-\infty,0)} \left[ f(x+z) - f(x) -  f'(x) z 1_{\{-1 < z < 0\}} \right] \Pi(\diff z), \quad x  >0.
	\end{align}
	%\red{above: changed from $h$ to $f$.}
	
%	\red{[moved the verification lemma. I guess we can just refer the reader to see the SP case for the proof?]}
We now provide a verification lemma. %\red{[changed to semicolon]}
The proof is essentially the same as  Lemma 4.3 in \cite{PY} (which deals with the spectrally positive case with a terminal payoff/penalty), and is hence omitted.
%we refer its proof to that of Lemma 4.3 in \cite{PY}. %\red{[Let us make sure it is exactly the same or need to add a few sentences.]}\blue{Agreed, in this case it might be even simplier because we are not considering payoff at ruin ($\rho=0$).}
	\begin{lemma}[Verification lemma]
	\label{verificationlemma}
	Suppose $\hat{\pi} \in \mathcal{A}$ is such that $v_{\hat{\pi}}$ is $C^1(0, \infty)$ (resp.\ $C^2(0, \infty)$)  for the case $X$ is of bounded (resp.\ unbounded) variation, % , right-continuous at zero  \red{[remove right-continuity at zero?]} %with
%	\begin{align}
%		v_{\hat{\pi}}(0+) = \rho, \label{v_0_continuity}
%	\end{align}
	%differentiable at zero \red{[not so sure differentiability is needed. Currently the function is not defined for $(-\infty, 0)$]}, 
	and satisfies
	\begin{align}
		\label{HJB-inequality}
		(\mathcal{L} - q)v_{\hat{\pi}}(x)+r\max_{0\leq l\leq x}\{l+v_{\hat{\pi}}(x-l)-v_{\hat{\pi}}(x)\}\leq 0,  \quad x > 0.
	\end{align} %\red{changed to $m$ because it would be confusing with the martingale $M$.}
	Then, $v_{\hat{\pi}}(x)=v(x)$ for all $x\geq0$, and hence $\hat{\pi}$ is an optimal strategy.
\end{lemma}

	%In the following section we will include some computations related the verification of optimality
	%of the periodic barrier strategy. 
%	By Proposition 2 of \cite{APP2007} and as in the proof of Theorem 8.10 of \cite{K}, the processes 
%	\begin{align}\label{mart_prop}
%		e^{-q (t \wedge \tau^-_{0} \wedge \tau^+_B)} Z^{(q+r)}(X(t \wedge \tau^-_{0} \wedge \tau^+_B)) \quad \textrm{and} \quad e^{-q (t \wedge \tau^-_{0} \wedge \tau^+_B)}  \Big( \overline{Z}^{(q+r)}(X(t \wedge \tau^-_{0} \wedge \tau^+_B)) + \frac {\psi'(0+)} q \Big), \quad t \geq 0,
%	\end{align}
%	for any $B > 0$ are $\p_x$-martingales.  Thanks \eqref{mart_prop} and the smoothness of $Z^{(q)}$ and $\overline{Z}^{(q)}$ on $(0,\infty)$, we obtain 
%	\begin{align}
%		(\mathcal{L}-(q+r)) Z^{(q+r)}(x-b^*) =(\mathcal{L}-(q+r)) \Big(\overline{Z}^{(q+r)}(x-b^*) + \frac {\psi'(0+)} {q+r} \Big) = 0, \quad y > 0. \label{martingale_Z_R}
%	\end{align}
	
	Note that $W^{(q)}$ is infinitely differentiable on $(0, \infty)$, as pointed out in Remark \ref{remark_scale_function_completely_monotone}. By the proof of Lemma 4 of \cite{APP2007}, for any $B > 0$, $(e^{-q (t \wedge \tau^-_{0} \wedge \tau^+_B)} W^{(q)}(X(t \wedge \tau^-_{0} \wedge \tau^+_B)); t \geq 0)$
	is a martingale, and hence
	\begin{align}
	(\mathcal{L}-q)W^{(q)}(y)=0, \quad y > 0. \label{W_harmonic}
	\end{align}
	Similarly, by Proposition 2 of \cite{APP2007},
% and as in the proof of Theorem 8.10 of \cite{K}, we have  %\red{changed from $q+r$ to $q$ here. OK}
	\begin{align}
	(\mathcal{L}-q) Z^{(q)}(y)
% =(\mathcal{L}-q) \Big(\overline{Z}^{(q)}(y) + \frac {\psi'(0+)} q \Big) 
= 0, \quad y > 0. \label{martingale_Z_R}
	\end{align}
	Note that these identities also hold when $q$ is replaced with $q+r$.
	By using these identities, we show the following. %[made this lemma]
	\begin{lemma} For $y > 0$,
	\begin{align}
	(\mathcal{L}-q) W^{(q+r)}(y)  &=  r W^{(q+r)}(y), \label{W_harmonic_q_r} \\
	\label{v_2}
	(\mathcal{L}-q)\overline{W}^{(q+r)}(y) &=1+r\overline{W}^{(q+r)}(y), \\
	\label{v_3}
	(\mathcal{L}-q)\overline{\overline{W}}^{(q+r)} (y)
	&=y +r\overline{\overline{W}}^{(q+r)}(y).
	\end{align}
	\end{lemma}
	\begin{proof}
	(i) By \eqref{W_harmonic}, we have $(\mathcal{L}-q) W^{(q+r)}(y) = (\mathcal{L}-q-r) W^{(q+r)}(y) + r W^{(q+r)}(y) =  r W^{(q+r)}(y)$. \\
	(ii) By using \eqref{martingale_Z_R}, we have
	\begin{align*}
	(\mathcal{L}-q)\overline{W}^{(q+r)}(y)&=\frac{1}{r+q}(\mathcal{L}-q)\left(Z^{(q+r)}(y)-1\right) =\frac{r}{r+q}Z^{(q+r)}(y)+\frac{q}{r+q} =1+r\overline{W}^{(q+r)}(y).
	\end{align*}
	(iii) % \green{Using \eqref{generator} we note that
	%\begin{align*}
	%\mathcal{L}y=\gamma +\int_{(-\infty,-1)}z\Pi(\diff z)=\psi'(0+).
	%\end{align*}}
	%Hence, by \eqref{martingale_Z_R},
	%\begin{align*}
	%(\mathcal{L}-q)\overline{\overline{W}}^{(q+r)}(y) &=(\mathcal{L}-q)\left(\frac{1}{r+q}\Big(\overline{Z}^{(q+r)}(y) + \frac {\psi'(0+)} {r+q} \Big)-\frac{\psi'(0+)}{(r+q)^2}-\frac{y}{r+q} \right)\notag\\
	%&=\frac{r}{r+q}\Big(\overline{Z}^{(q+r)}(y) + \frac {\psi'(0+)}{r+q} \Big)+\frac{q}{(r+q)^2}\psi'(0+)+\frac{q}{r+q}y-\frac {\psi'(0+)}{r+q}\notag\\
	%&= y +r\overline{\overline{W}}^{(q+r)}(y).
	%\end{align*}
%\red{[Just realized that we are not assuming that $\psi'(0+) > -\infty$.  Hence, we should change the proof. How about this?]}
By integration by parts and  the proof of Lemma 4.5 in \cite{EY}, we have 
\begin{align*}
(\mathcal{L}-q-r)\overline{\overline{W}}^{(q+r)}(y)  = (\mathcal{L}-q-r) \int_0^y z W^{(q+r)}(y-z) \diff  z = y.
\end{align*}
Hence, we have \eqref{v_3}.

	\end{proof}

	%By Proposition 2 in \cite{APP2007} and TODO,
\begin{lemma}\label{generator_on_v}
		For  $b^* \geq 0$, we have %\red{[cleaned up the equation a bit below]}
		\begin{equation}\label{gen_on_v}
		(\mathcal{L}-q)v_{b^*}(x)=
		\begin{cases} 0 &\mbox{if } x \in(0,b^*], \\ 
		\displaystyle -r\left\{(x-b^*)+v_{b^*}(b^*)-v_{b^*}(x)\right\} & \mbox{if } x \in (b^*,\infty). \end{cases}
		\end{equation}
	\end{lemma}
	\begin{proof} (i)  %\red{[shortened the proof a bit.]}
	Suppose $b^* > 0$.
For $0 < x < b^*$, by \eqref{value_function_below_b_star} and \eqref{W_harmonic}, we have
	\begin{align*}
	(\mathcal{L}-q)v_{b^*}(x)=\frac{1}{W^{(q)\prime}(b^*)}(\mathcal{L}-q)W^{(q)}(x)=0.
	\end{align*}

%	On the other hand, $x>b^*$, we have that
%	\begin{align}\label{v_1}
%	v_{b^*}(x)&=\frac{1}{W^{(q)\prime}(b^*)}\left(W^{(q)}(x)+r\int_0^{x-b^*}W^{(q+r)}(x-b^*-y)W^{(q)}(y+b^*)\diff y\right)\notag\\&-r\frac{W^{(q)}(b^*)}{W^{(q)\prime}(b^*)}\overline{W}^{(q+r)}(x-b^*)-r\overline{\overline{W}}^{(q+r)}(x-b^*).
%	\end{align}
Now suppose $x > b^*$.
%	Now we note that we can write
%	\begin{align*}
%	\overline{W}^{(q+r)}(x-b^*)&=\frac{1}{q+r}\left(Z^{(q+r)}(x-b^*)-1\right), \\
%	\overline{\overline{W}}^{(q+r)}(x-b^*)&=\frac{1}{r+q}\Big(\overline{Z}^{(q+r)}(x-b^*) + \frac {\psi'(0+)}{r+q} \Big)-\frac{\psi'(0+)}{(r+q)^2}-\frac{1}{r+q}(x-b^*).
%	\end{align*}
%	Using \eqref{martingale_Z_R} we have that
%	\begin{align}\label{v_2}
%	\begin{split}
%	(\mathcal{L}-q)\overline{W}^{(q+r)}(x-b^*)&=\frac{1}{(q+r)}(\mathcal{L}-q)\left(Z^{(q+r)}(x-b^*)-1\right) \\
%	&=\frac{r}{r+q}Z^{(q+r)}(x-b^*)+\frac{q}{r+q} =1+r\overline{W}^{(q+r)}(x-b^*),
%	\end{split}
%	\end{align}
%	and
%	%On the other hand using \eqref{martingale_Z_R} we obtain
%	\begin{align}\label{v_3}
%	(\mathcal{L}-q)\overline{\overline{W}}^{(q+r)}&(x-b^*)=(\mathcal{L}-q)\left(\frac{1}{r+q}\Big(\overline{Z}^{(q+r)}(x-b^*) + \frac {\psi'(0+)} q \Big)-\frac{\psi'(0+)}{(r+q)^2}-\frac{1}{r+q}(x-b^*)\right)\notag\\
%	&=\frac{r}{r+q}\Big(\overline{Z}^{(q+r)}(x-b^*) + \frac {\psi'(0+)}{r+q} \Big)+\frac{q}{(r+q)^2}\psi'(0+)+\frac{q}{r+q}(x-b^*)-\frac {\psi'(0+)}{r+q}\notag\\
%	&=(x-b^*)+r\overline{\overline{W}}^{(q+r)}(x-b^*).
%	\end{align}
	By the proof of Lemma 4.5 in \cite{EY}, we obtain
	\begin{align*}
	(\mathcal{L}-(q+r))\int_0^{x-b^*}W^{(q+r)}(x-b^*-y)W^{(q)}(y+b^*)\diff y&=(\mathcal{L}-(q+r))\int_{b^*}^{x}W^{(q+r)}(x-u)W^{(q)}(u) \diff u\\&=W^{(q)}(x),
	\end{align*}
%	\blue{[To adress comment 7 of the referee, how about this:] By identity \eqref{W_q_W_q_r_relation} we obtain that
%	\begin{align*}
%	(\mathcal{L}-(q+r))\int_0^{x-b^*}&W^{(q+r)}(x-b^*-y)W^{(q)}(y+b^*)\diff y\\&=(\mathcal{L}-(q+r))\left(\frac{W^{(q+r)}(x-b^*)}{r}-\frac{W^{(q)}(x-b^*)}{r}\right)=W^{(q)}(x),
%	\end{align*}
%	\red{JL: \eqref{W_q_W_q_r_relation} can be used only when $b^*$. I think we still need to stick to Lemma 4.5 in \cite{EY}.}
%	}
	which together with \eqref{W_harmonic} implies
	\begin{align}\label{v_4}
	(\mathcal{L}-q)\Big(W^{(q)}(x)&+r\int_0^{x-b^*}W^{(q+r)}(x-b^*-y)W^{(q)}(y+b^*)\diff y\Big)\notag\\
	&=r\left(W^{(q)}(x)+r\int_0^{x-b^*}W^{(q+r)}(x-b^*-y)W^{(q)}(y+b^*)\diff y\right).
	\end{align}
	Hence, by applying  \eqref{v_2}, \eqref{v_3}, and \eqref{v_4} in \eqref{vf5}, we have
	\begin{align*}
	(\mathcal{L}-q)v_{b^*}(x)&=\frac{r}{W^{(q)\prime}(b^*)}\left(W^{(q)}(x)+r\int_0^{x-b^*}W^{(q+r)}(x-b^*-y)W^{(q)}(y+b^*)\diff y\right)\notag\\&-r\frac{W^{(q)}(b^*)}{W^{(q)\prime}(b^*)}\left(r\overline{W}^{(q+r)}(x-b^*)+1\right)-r\left((x-b^*)+r\overline{\overline{W}}^{(q+r)}(x-b^*)\right)\\
	&=-r\left((x-b^*) +\frac{W^{(q)}(b^*)}{W^{(q)\prime}(b^*)} - v_{b^*}(x)\right)\\
	&=-r\left\{(x-b^*)+v_{b^*}(b^*)-v_{b^*}(x)\right\}.
	\end{align*}
	Finally, this can be extended to the case in which $x=b^*$ by taking limits as $x \rightarrow b^*$.		

	(ii) Suppose $b^* =0$.
%	We  have that by \eqref{vf0}, \eqref{v_2}, and \eqref{v_3} that
%	\begin{align*}
%	(\mathcal{L}-q)v_0(x)&=\frac{rW^{(q+r)}(x)-r(1+r\overline{W}^{(q \red{q+r?})}(x))W^{(q)}(0) }{\Phi(q+r)Z^{(q)\prime}(0,\Phi(q+r))} -r \Big( x+\overline{\overline{W}}^{(q+r)}(x) \Big) \\
%	&=-r(x+v_0(0)-v_0(x)).
%	\end{align*}
	%\red{[How about this?]
By applying \eqref{W_harmonic_q_r}, \eqref{v_2}, and \eqref{v_3} in \eqref{vf0}, we have
			\begin{align*} 
				(\mathcal{L}-q) v_0(x)&=r \frac{ (\mathcal{L}-q) W^{(q+r)}(x)-r  W^{(q)}(0)  (\mathcal{L}-q) \overline{W}^{(q+r)}(x) }{\Phi(q+r)(\Phi(q+r)-r W^{(q)}(0)) }-r (\mathcal{L}-q) \overline{\overline{W}}^{(q+r)}(x) \\
				&=r \Big[ \frac{ r W^{(q+r)}(x)-r  W^{(q)}(0)  (1+r\overline{W}^{(q+r)}(x))}{\Phi(q+r)(\Phi(q+r)-r W^{(q)}(0)) }- \Big(x+r\overline{\overline{W}}^{(q+r)}(x) \Big) \Big],
			\end{align*}
			which equals $-r(x+v_0(0)-v_0(x))$ as desired.
				\end{proof}	
	We now prove the following.
		\begin{proposition} \label{prop_slope_condition}
	For $b^* \geq 0$, we have $v_{b^*}'(x) \geq 1$ for $x \in (0, b^*)$ and $0 \leq v_{b^*}'(x) \leq 1$ for $x \in (b^*, \infty)$. 
	%\red{[JL: I added that it is nonnegative because it is used later.]}\blue{[OK.]}
	\end{proposition}

	To prove this proposition, we first rewrite the derivative \eqref{015} by using the decomposition of the scale function given in Remark \ref{remark_scale_function_completely_monotone} (1).
	
	\begin{lemma} \label{lemma_v_b_prime_alternative}%\red{[made this lemma.]} 
		For $x \geq b \geq 0$, we have
	%		\begin{align}\label{dvf_cm}
%		\frac{1}{r}v_{b}'(x)
%		&=\frac{\Phi'(q+r)}{\Phi(q+r)}-\frac{f'(x)-rf(x)W^{(q)}(0)-r\int_0^{b}f(x-u)W^{(q)\prime}(u)du}{\Phi(q+r)Z^{(q)\prime}(b,\Phi(q+r))}+\int_0^{x-b}f(u)du,
%	\end{align}	
%	where $f$ is as defined in Remark \ref{remark_scale_function_completely_monotone}(2).
%	
		\begin{align*} 
	v_{b}'(x)
	%&=r\frac{\Phi'(q+r)}{\Phi(q+r)}+r\int_0^{x-b}f(u)du-r \frac{f'(x)-rf(x)W^{(q)}(0)-r\int_0^{b}f(x-u)W^{(q)\prime}(u)du}{\Phi(q+r)Z^{(q)\prime}(b,\Phi(q+r))} \\
	%&=r\frac{\Phi'(q+r)}{\Phi(q+r)} + \int_0^\infty \frac {\mu(\diff t)} t - r \int_0^\infty \frac {e^{-t (x-b)}} t \mu (\diff t) \\
	%&+r \frac{\int_0^\infty te^{-tx} \mu (\diff t ) + r \int_0^\infty e^{-tx} W^{(q)}(0) \mu (\diff t)+ r\int_0^{b} \int_0^\infty e^{-t (x-u)} \mu (\diff t)W^{(q)\prime}(u)du}{\Phi(q+r)Z^{(q)\prime}(b,\Phi(q+r))} \\
		&= K +r \frac{\int_0^\infty e^{-tx} g(t,b) \mu^{(q+r)} (\diff t )}{\Phi(q+r)Z^{(q)\prime}(b,\Phi(q+r))},
	\end{align*}	
	%\red{[JL: How about changing $\mu$ to $\mu^{(q+r)}$?]}\blue{[Kazu:OK]}
		where
	\begin{align*}K &:= r\frac{\Phi'(q+r)}{\Phi(q+r)} + r\int_0^\infty \frac {\mu^{(q+r)}(\diff t)} t, \\
	g(t,b) &:= t  + r  W^{(q)}(0) + r\int_0^{b}  e^{ut} W^{(q)\prime}(u) \diff u -    \frac {e^{tb}} t  \Phi(q+r) Z^{(q)\prime}(b,\Phi(q+r)) \\
	&= t  + r  W^{(q)}(0) + r\int_0^{b}  e^{ut} \Big(W^{(q)\prime}(u)- \frac{\Phi(q+r)}{r}Z^{(q)\prime}(b,\Phi(q+r)) \Big) \diff u -  \frac {\Phi(q+r)} t  Z^{(q)\prime}(b,\Phi(q+r)).
	\end{align*}
	%\red{JL: I think $K := r\frac{\Phi'(q+r)}{\Phi(q+r)} + \red{r} \int_0^\infty \frac {\mu(\diff t)} t$.  Can you check the computation below?}
	Note that by considering $\int_0^\infty e^{-xt} \mu^{(q+r)}(\diff t) $ as the density of the $(q+r)$-resolvent measure of $-X$ at $x > 0$ as in Theorem 2.7 (iv) of \cite{KKR}, 
%\green{[below I think we need the division by $q+r$ and hence I changed. Could you check?]}
	\begin{align*}
	\frac 1 {q+r} > \frac {\p (X(\mathbf{e}_{q+r}) < 0)} {q+r}  = \int_0^\infty \int_0^\infty e^{-xt} \mu^{(q+r)}(\diff t) \diff x = \int_0^\infty \int_0^\infty e^{-xt}  \diff x \mu^{(q+r)}(\diff t) = \int_0^\infty  \frac {\mu^{(q+r)}(\diff t)} t.
	\end{align*}
%	\blue{[Kazu I get this]
%	\begin{align*}
%		\frac {\p (X(\mathbf{e}_{q+r}) < 0)} {q+r}  &= \mathbb{E}\left[\int_0^{\infty}e^{-(q+r)t}1_{\{X_t<0\}}dt\right]=\int_{-\infty}^{0}\left(\Phi'(q+r)e^{-\Phi(q+r)y}-W^{(q+r)}(-y)\right)dy\\&=\int_{-\infty}^0 \int_0^\infty e^{-xt} \mu^{(q+r)}(\diff t) \diff x
%	\end{align*}
%		}
	\end{lemma}
	\begin{proof}
By differentiating both sides of \eqref{W_q_W_q_r_relation}, for $x > 0$,
	\begin{align*}
	&W^{(q+r)\prime}(x)-rW^{(q+r)}(x)W^{(q)}(0)-r\int_0^bW^{(q+r)}(x-u)W^{(q)\prime}(u) \diff u \\
	&=W^{(q)\prime}(x)+r\int_b^xW^{(q+r)}(x-u)W^{(q)\prime}(u)\diff u = W^{(q)\prime}(x)+ r \int_0^{x- b} W^{(q+r)}(y)W^{(q)\prime} (x- y) \diff y.
	\end{align*}
	Hence, the equality  \eqref{015}, for $b > 0$, reduces to
	\begin{align} \label{v_b_prime_alternative}
		\frac{1}{r}v_{b}'(x)=\frac{W^{(q+r)\prime}(x)-rW^{(q+r)}(x)W^{(q)}(0)-r\int_0^bW^{(q+r)}(x-u)W^{(q)\prime}(u)\diff u}{\Phi(q+r)Z^{(q)\prime}(b,\Phi(q+r))}-\overline{W}^{(q+r)}(x-b).
	\end{align}
The same expression is obtained for $b = 0$ by differentiating \eqref{vf0}. %[added this, because we did not show \eqref{015} for the case $b = 0$.]}		
	Further, by using Remark \ref{remark_scale_function_completely_monotone} (1), we can write%\blue{[Kazu: Change identities to use $\mu$ instead of $f$]}
%	Now we can write
%	\[
%	W^{(q+r)}(x)=\Phi'(q+r)e^{\Phi(q+r)x}-f(x),\qquad x>0
%	\]
%	where $f$ is completely monotone. Therefore
	\begin{align}\label{vf_cm}
		\frac{1}{r}v_{b}'(x)&=\Phi'(q+r)\frac{\Phi(q+r)e^{\Phi(q+r)x}-re^{\Phi(q+r)x}W^{(q)}(0)-r\int_0^{b}e^{\Phi(q+r)(x-u)}W^{(q)\prime}(u)\diff u}{\Phi(q+r)Z^{(q)\prime}(b,\Phi(q+r))}\\&-\Phi'(q+r)\int_0^{x-b}e^{\Phi(q+r)u}\diff u+\int_0^{x-b}\int_0^\infty e^{-tu} \mu^{(q+r)} (\diff t )\diff u \notag\\
		&+ \frac{\int_0^\infty te^{-tx} \mu^{(q+r)} (\diff t ) + r \int_0^\infty e^{-tx} W^{(q)}(0) \mu^{(q+r)} (\diff t)+ r\int_0^{b} \int_0^\infty e^{-t (x-u)} \mu^{(q+r)} (\diff t)W^{(q)\prime}(u)\diff u}{\Phi(q+r)Z^{(q)\prime}(b,\Phi(q+r))}. \notag
	\end{align}
	Integration by parts gives
	%	\begin{align*}
	%	\int_0^{b^*}e^{\Phi(q+r)(x-u)}W^{(q)\prime}(u)du=e^{-\Phi(q+r)b^*}W^{(q)}(b^*)-W^{(q)}(0)+\Phi(q+r)\int_0^{b^*}e^{\Phi(q+r)(x-u)}W^{(q)}(u)du,
	%	\end{align*}
	%	\red{[JL: I guess it should be
	\begin{align*}
		\int_0^{b}e^{\Phi(q+r)(x-u)}W^{(q)\prime}(u) \diff u= e^{\Phi(q+r) x} \Big[ e^{-\Phi(q+r)b}W^{(q)}(b)-W^{(q)}(0)+\Phi(q+r)\int_0^{b}e^{-\Phi(q+r)u}W^{(q)}(u) \diff u \Big],
	\end{align*}
	%but I think the following computation is ok.]
	%	}
	and hence
	\begin{align*}
		\Phi(q+r)e^{\Phi(q+r)x}&-re^{\Phi(q+r)x}W^{(q)}(0)-r\int_0^{b}e^{\Phi(q+r)(x-u)}W^{(q)\prime}(u) \diff u \\
		&=e^{\Phi(q+r)(x-b)}\left(e^{\Phi(q+r)b}\Phi(q+r)\left(1-r\int_0^{b}e^{-\Phi(q+r) u}W^{(q)}(u) \diff u \right)-rW^{(q)}(b)\right)\\
		&=e^{\Phi(q+r)(x-b)}Z^{(q)\prime}(b,\Phi(q+r)).
	\end{align*}
	%\green{[shortened a bit below]} 
	Now, by using the above expression in \eqref{vf_cm}, we obtain
	\begin{align}%\label{dvf_cm}
		\frac{1}{r}v_{b}'(x)&=\frac{\Phi'(q+r)}{\Phi(q+r)}\left(e^{\Phi(q+r)(x-b)}-(e^{\Phi(q+r)(x-b)}-1)\right)+\int_0^{x-b}\int_0^\infty e^{-tu} \mu^{(q+r)} (\diff t ) \diff u \notag\\
		&+ \frac{\int_0^\infty te^{-tx} \mu^{(q+r)} (\diff t ) + r \int_0^\infty e^{-tx} W^{(q)}(0) \mu^{(q+r)} (\diff t)+ r\int_0^{b} \int_0^\infty e^{-t (x-u)} \mu^{(q+r)} (\diff t)W^{(q)\prime}(u)\diff u}{\Phi(q+r)Z^{(q)\prime}(b,\Phi(q+r))}\notag\\
		&=\frac{\Phi'(q+r)}{\Phi(q+r)} %+\int_0^{x-b}\int_0^\infty e^{-tu} \mu (\diff t )\diff u 
			+  \int_0^\infty \frac {\mu^{(q+r)}(\diff t)} t -  \int_0^\infty \frac {e^{-t (x-b)}} t \mu^{(q+r)} (\diff t) \notag\\
		&+\frac{\int_0^\infty te^{-tx} \mu^{(q+r)} (\diff t ) + r \int_0^\infty e^{-tx} W^{(q)}(0) \mu^{(q+r)} (\diff t)+ r\int_0^{b} \int_0^\infty e^{-t (x-u)} \mu^{(q+r)}(\diff t)W^{(q)\prime}(u)\diff u}{\Phi(q+r)Z^{(q)\prime}(b,\Phi(q+r))}.\notag
	\end{align}	
	Hence, the result is obtained.
%	\begin{align*} 
%	v_{b}'(x)
%	%&=r\frac{\Phi'(q+r)}{\Phi(q+r)}+r\int_0^{x-b}f(u)\diff u-r \frac{f'(x)-rf(x)W^{(q)}(0)-r\int_0^{b}f(x-u)W^{(q)\prime}(u)\diff u}{\Phi(q+r)Z^{(q)\prime}(b,\Phi(q+r))} \\
%	&=r\frac{\Phi'(q+r)}{\Phi(q+r)} + r \int_0^\infty \frac {\mu(\diff t)} t - r \int_0^\infty \frac {e^{-t (x-b)}} t \mu (\diff t) \\
%	&+r \frac{\int_0^\infty te^{-tx} \mu (\diff t ) + r \int_0^\infty e^{-tx} W^{(q)}(0) \mu (\diff t)+ r\int_0^{b} \int_0^\infty e^{-t (x-u)} \mu (\diff t)W^{(q)\prime}(u)\diff u}{\Phi(q+r)Z^{(q)\prime}(b,\Phi(q+r))} \\
%		&= K +r \frac{\int_0^\infty e^{-tx} g(t,b) \mu (\diff t )}{\Phi(q+r)Z^{(q)\prime}(b,\Phi(q+r))},
%	\end{align*}	
%	as desired.	
%		where
%	\begin{align*}K &:= r\frac{\Phi'(q+r)}{\Phi(q+r)} + \int_0^\infty \frac {\mu(\diff t)} t, \\
%	g(t,b) &:= t  + r  W^{(q)}(0) + r\int_0^{b}  e^{ut} W^{(q)\prime}(u)du -  r  \frac {e^{tb}} t  \frac{\Phi(q+r)}{r}Z^{(q)\prime}(b,\Phi(q+r)) \\
%	&= t  + r  W^{(q)}(0) + r\int_0^{b}  e^{ut} (W^{(q)\prime}(u)- \frac{\Phi(q+r)}{r}Z^{(q)\prime}(b,\Phi(q+r))) du -  \frac r t  \frac{\Phi(q+r)}{r}Z^{(q)\prime}(b,\Phi(q+r)).
%	\end{align*}
	\end{proof}	

We are now ready to prove the proposition.
	\begin{proof}[Proof of Proposition \ref{prop_slope_condition}]
	(i) Suppose $b^* > 0$.
	
	(1) Suppose $x \leq b^*$. Because  $b^*\leq \bar{b}$ and $W^{(q)\prime}$ is decreasing on $(0,\bar{b})$ as mentioned in Remark \ref{remark_scale_function_completely_monotone} (2),
	\[
	v_{b^*}'(x)=\frac{W^{(q)\prime}(x)}{W^{(q)\prime}(b^*)}\geq 1.
	\]
	(2) Suppose $x > b^*$. 
	%\green{[added this to show that the slope is nonnegative] 
	First, because the strategy $\pi^{b^*}$ pushes the process to $b^*$ at the first exponential time $T(1) =\mathbf{e}_r$ at which the process is above $b^*$, we can write
	\begin{align} \label{v_b_star_rewritten}
	v_{b^*}(x) &= \E \Big[ e^{-q \mathbf{e}_r} [(X(\mathbf{e}_r) + x - b^*) \vee 0] 1_{\{\underline{X}(\mathbf{e}_r) + x \geq 0 \}}\Big] + \E \Big[ e^{-q \mathbf{e}_r} v_{b^*}((X(\mathbf{e}_r) + x) \wedge b^*) 1_{\{\underline{X}(\mathbf{e}_r) +x  \geq 0 \}}\Big],
	\end{align}
	where $\underline{X}$ is the running infimum process of $X$.
	Because $v_{b^*}$ is nonnegative and increasing on $(0, b^*)$ according to (1), this is increasing in $x > b^*$. This shows that $v_{b^*}'(x)  \geq 0$.

	%From the computation above %\red{[I guess we should move the computation here.]} \blue{[I agree Kazu.]},
	By Lemma \ref{lemma_v_b_prime_alternative} and because the condition $\mathfrak{C}_b$ holds for $b = b^*$,
	% and because $b^*$ satisfies \eqref{opt_thres},
	\begin{align} \label{v_b_mu}
	\begin{split}
	v_{b^*}'(x)
%	&=r\frac{\Phi'(q+r)}{\Phi(q+r)}+r\int_0^{x-b^*}f(u)du-\frac{f'(x)-rf(x)W^{(q)}(0)-r\int_0^{b^*}f(x-u)W^{(q)\prime}(u)du}{W^{(q)\prime} (b^*)} \\
%	&=r\frac{\Phi'(q+r)}{\Phi(q+r)} + \int_0^\infty \frac {\mu(\diff t)} t - r \int_0^\infty \frac {e^{-t (x-b^*)}} t \mu (\diff t) \\
%	&+\frac{\int_0^\infty te^{-tx} \mu (\diff t ) + r \int_0^\infty e^{-tx} W^{(q)}(0) \mu (\diff t)+ r\int_0^{b^*} \int_0^\infty e^{-t (x-u)} \mu (\diff t)W^{(q)\prime}(u)du}{W^{(q)\prime} (b^*)} \\
		&= K +\frac{\int_0^\infty e^{-tx} g(t,b^*) \mu^{(q+r)} (\diff t )}{W^{(q)\prime} (b^*)},
		\end{split}
	\end{align}	
	where by $\mathfrak{C}_b$ for $b = b^*$,
	\begin{align*}
	%K &:= r\frac{\Phi'(q+r)}{\Phi(q+r)} + \int_0^\infty \frac {\mu(\diff t)} t, \\
	g(t,b^*) %&:= t  + r  W^{(q)}(0) + r\int_0^{b^*}  e^{ut} W^{(q)\prime}(u)du -  r  \frac {e^{tb^*}} t  W^{(q)\prime} (b^*) \\
	&= t  + r  W^{(q)}(0) + r\int_0^{b^*}  e^{ut} (W^{(q)\prime}(u)- W^{(q)\prime} (b^*)) \diff u -  \frac r t  W^{(q)\prime} (b^*).
	\end{align*}
		
We have $g(0+,b^*) = -\infty$
	and because \eqref{b_star_b_bar} and Remark  \ref{remark_scale_function_completely_monotone} (2) give $W^{(q)\prime}(u) \geq W^{(q)\prime}(b^*)$ for $u \leq b^*$,
		\begin{align*}
	\frac \partial {\partial t} g(t,b^*) 
	&= 1  + r\int_0^{b^*}  u e^{ut} (W^{(q)\prime}(u)- W^{(q)\prime} (b^*)) \diff u +  \frac r {t^2}  W^{(q)\prime} (b^*) > 1.
	\end{align*}
	Hence, there exists $p > 0$ such that 
	\begin{align} \label{signs_g}
	g(t,b^*) \leq 0 \Longleftrightarrow t \leq p. 
	\end{align}
	
	%\green{[shortened a bit here] 
	By monotone convergence,  
	\begin{align*}
	\int_0^\infty e^{-tx} g(t,b^*) \mu^{(q+r)} (\diff t ) = -\int_0^p e^{-tx} |g(t,b^*)| \mu^{(q+r)} (\diff t ) + \int_p^\infty e^{-tx} g(t,b^*) \mu^{(q+r)} (\diff t ) \xrightarrow{x \uparrow \infty}  0.
	\end{align*}
	Hence, in view of \eqref{v_b_mu}, $\lim_{x \rightarrow \infty}v_{b^*}'(x) = K$. 
	
	To show that 
	\begin{align}
	0 \leq K \leq 1, \label{K_bounds}
	\end{align}
by using \eqref{v_b_star_rewritten} and denoting $\overline{X}$ as the running supremum process of $X$,
%	, we should have $v_{b^*}(x) \sim x \E [e^{-q T_1}]$ as $x \rightarrow \infty$. More precisely, at least we can say
%\blue{[Kazu: Wrote it like this]
	\begin{align*}
	v_{b^*} (x) 
	%&\leq \E_x \Big[ e^{-q \mathbf{e}_r} (X(\mathbf{e}_r) - b^*) 1_{\{ X(\mathbf{e}_r) > b^*\}} \Big] +\E_x \Big[ e^{-q  \mathbf{e}_r} \Big] \sup_{0 \leq y \leq b^*} v_{b^*} (y) \\
	&\leq \E \Big[ e^{-q  \mathbf{e}_r} (x+\overline{X}(\mathbf{e}_r) - b^*)  \Big]  +\E_x \Big[ e^{-q \mathbf{e}_r} \Big] \sup_{0 \leq y \leq b^*} v_{b^*} (y) \\
	&\leq x \E\Big[ e^{-q \mathbf{e}_r}  \Big] + \Phi(r)^{-1} + \E \Big[ e^{-q \mathbf{e}_r} \Big] \sup_{0 \leq y \leq b^*} v_{b^*} (y),
	\end{align*}
	where in the last inequality we used $\E [ e^{-q \mathbf{e}_r} \overline{X}(\mathbf{e}_r)   ] \leq \E [  \overline{X}( \mathbf{e}_r)  ] = \Phi(r)^{-1} < \infty$ (see Exercise 3.6 in \cite{K}). Hence, $K = \lim_{x \rightarrow \infty} {v_{b^*}(x)} / x \leq \E \Big[ e^{-q \mathbf{e}_r} \Big] < 1$. Moreover, based on the aforementioned argument according to which $v_{b^*}$ is nondecreasing, we must have $K \geq 0$.
%On the other hand, if $K < 0$, $v_{b^*} (x) \xrightarrow{x \rightarrow \infty} - \infty$, which contradicts with the nonnegativity of $v_{b^*}$. Hence, we must have \eqref{K_bounds}.
%\begin{align}
%0 \leq K  < 1. \label{K_bounds}
%\end{align}

%\red{Actually, I think we do not need this moment condition, we have, for $x > b^*$,
%\begin{align*}
%	\E_x \Big[ e^{-q T_1} (X(T_1) - b^*) 1_{\{ X(T_1) > b^*\}} \Big] 	&=  \E \Big[ e^{-q T_1} (x+X(T_1) - b^*) 1_{\{ X(T_1) > b^*\}} \Big] \leq  \E \Big[ e^{-q T_1} (x+\overline{X}(T_1) - b^*)  \Big] 
%	\end{align*}
%%		\begin{align*}
%	v_{b^*} (x) &\leq x \E_x \Big[ e^{-q T_1}  \Big] + \Phi(r)^{-1} +\E_x \Big[ e^{-q T_1} \Big] \sup_{0 \leq y \leq b^*} v_{b^*} (y). \end{align*}
%	 This does not require $\psi'(0+) > -\infty$?
%}

%	 is negative on $(0, p)$ and positive on $(p, \infty)$.	
	
%		In view of \eqref{v_b_mu}, $v_{b^*}'(\infty)$ exists and becomes $K$.
%	From the discussion above, we have
%	\begin{align*}
%	K = v_{b^*}'(\infty) = \E [e^{-q T_1}] \in (0, 1).
%	\end{align*}
	
		Because $v_{b^*}'(b^*) = 1$ according to \eqref{value_function_below_b_star} and the smoothness at $b^*$ as stated in Lemma \ref{smooth_fit_prob1},
	\begin{align*}
	1 = v_{b^*}'(b^*)
		&= K +\frac{\int_0^\infty e^{-t b^*} g(t,b^*) \mu^{(q+r)} (\diff t )}{W^{(q)\prime} (b^*)}.
		\end{align*}
In addition, \eqref{signs_g} gives  $e^{-(x-b^*)p}g(t, b^*)\geq e^{-(x-b^*)t}g(t,b^*)$ for all $t>0$.
Therefore, by using these and \eqref{K_bounds},
%, and because $K \in (0,1)$,
	\begin{align*}
	v_{b^*}'(x)
		&= K +\frac{\int_0^\infty e^{-(x-b^*) t}e^{-t b^*} g(t,b^*) \mu^{(q+r)} (\diff t )}{W^{(q)\prime} (b^*)} \\
		&\leq K +e^{-(x-b^*) p} \frac{\int_0^\infty e^{-t b^*} g(t,b^*) \mu^{(q+r)} (\diff t )}{W^{(q)\prime} (b^*)} = K +e^{-(x-b^*) p} (1-K) \leq 1.
	\end{align*}
	(ii) Suppose $b^* = 0$.
		%\green{[added this to show that the slope is nonnegative] 
		First, we can write
	\begin{align*}
	v_0(x) &= \E \Big[ e^{-q \mathbf{e}_r} (X(\mathbf{e}_r) + x) 1_{\{\underline{X}(\mathbf{e}_r) + x \geq 0 \}}\Big] + \E \Big[ e^{-q \mathbf{e}_r} v_0(0) 1_{\{\underline{X}(\mathbf{e}_r) +x  \geq 0 \}}\Big].
	\end{align*}
	%where $\underline{X}$ is the running infimum process.
	Because $v_0$ is nonnegative, this increases in $x$, showing that $v_{0}'(x)  \geq 0$. 
	%}

By Lemma \ref{lemma_v_b_prime_alternative} and the fact that $Z^{(q)\prime}(0,\Phi(q+r))=\Phi(q+r)-rW^{(q)}(0)$, we have
%\eqref{dvf_cm}
%	\begin{align*}
%	v_{b^*}'(x)
%	&=r\frac{\Phi'(q+r)}{\Phi(q+r)}-\frac{f'(x)-rf(x)W^{(q)}(0)}{\Phi(q+r)Z^{(q)\prime}(0,\Phi(q+r))}+r\int%_0^{x}f(u)du\\
%	&=K+\frac{\int_0^\infty e^{-t b^*} g(\blue{t,0}) \mu (\diff t )}{\Phi(q+r)Z^{(q)\prime}(0,\Phi(q+r))},
%\end{align*}	
%where 
%\[
%\tilde{g}(t)= t  + r  W^{(q)}(0) -\frac{r}{t}\Phi(q+r)(\Phi(q+r)-rW^{(q)}(0)).
%\]
%\red{[Maybe this?]
		\begin{align}  \label{v_zero_prime_rewritten}
	v_{0}'(x)
		&= K +r \frac{\int_0^\infty e^{-tx} g(t,0) \mu^{(q+r)} (\diff t )}{\Phi(q+r)(\Phi(q+r)-rW^{(q)}(0))},
	\end{align}	
where
%	\begin{align*}	g(t,0) &= t  + r  W^{(q)}(0) - r   \frac{\Phi(q+r)}{t} (\Phi(q+r)-rW^{(q)}(0)).	\end{align*}
%	\red{Is this
	\begin{align*}
	g(t,0) &= t  + r  W^{(q)}(0) -    \frac {1} t  \Phi(q+r) (\Phi(q+r)-rW^{(q)}(0)).
	\end{align*}%}

%	}

Recall that $b^*=0$ if and only if (i) or (ii) of Remark \ref{remark_cond_b_zero} holds.
%\begin{itemize}
%	\item[(i)] $\eta>0$ and $\displaystyle r\leq\frac{\eta^2}{2}\Phi^2(q+r)$ or
%	\item[(ii)] $\eta=0$, $\Pi(-\infty,0)<\infty$ and $\displaystyle\Phi^2(q+r)\geq r\left(\frac{(q+\Pi(-\infty,0))}{c^2}+\frac{\Phi(q+r)}{c}\right)$.
%\end{itemize}
For case (i), we have $\Phi(q+r)-rW^{(q)}(0)=\Phi(q+r)>0$ by \eqref{eq:Wqp0}, while in  case (ii), by Remark \ref{remark_cond_b_zero} (ii), we have
\[
\Phi(q+r)-rW^{(q)}(0)=\Phi(q+r)-\frac{r}{c}\geq r\frac{q+\Pi(-\infty,0)}{c^2\Phi(q+r)}>0.
\]
This implies that $g(0+,0)=-\infty$ and 
%\[
%\tilde{g}(t)=1+\frac{r}{t^2}\Phi(q+r)(\Phi(q+r)-rW^{(q)}(0))>1.
%\]
%\red{Is this
\begin{align*}	\frac \partial {\partial t}g(t,0) &= 1 +    \frac{\Phi(q+r)}{t^2} (\Phi(q+r)-rW^{(q)}(0)) > 1.	\end{align*}%}
Hence, there exists $p > 0$ such that $g$ is negative on $(0, p)$ and positive on $(p, \infty)$. Then, $e^{-xp}g(t)\geq e^{-xt}g(t)$ for all $t>0$.

%Now noting that $b^*=0$ if and only if $\Phi^2(q+r)\geq rW^{(q)\prime}(0)+r\Phi(q+r)W^{(q)}(0)$ (in view of \eqref{b_zero_criteria} and Proposition \ref{prop_b_existence}) then
%\[
%v_0'(0)=K+\frac{\int_0^{\infty}g(t)\mu (\diff t )}{\Phi(q+r)Z^{(q)\prime}(0,\Phi(q+r))}=\frac{rW^{(q)\prime}(0)}{\Phi(q+r)(\Phi(q+r)-rW^{(q)}(0))}\leq1.
%\]
%\red{[I think the second equality is confusing? How about this?] 

By \eqref{v_b_prime_alternative},
\begin{align*}
		v_{0}'(0+)= r \frac{W^{(q+r)\prime}(0+)-rW^{(q+r)}(0)W^{(q)}(0)}{\Phi(q+r)(\Phi(q+r)-rW^{(q)}(0))} =\frac{rW^{(q)\prime}(0+)}{\Phi(q+r)(\Phi(q+r)-rW^{(q)}(0))},\end{align*}
	where, under (ii) of Remark \ref{remark_cond_b_zero}, the second equality holds by \eqref{eq:Wqp0}.
	%can be confirmed by differentiating  \eqref{W_q_W_q_r_relation} and taking limits. 
Noting that $b^*=0$ if and only if $\Phi^2(q+r)\geq rW^{(q)\prime}(0+)+r\Phi(q+r)W^{(q)}(0)$ (in view of \eqref{b_zero_criteria} and Proposition \ref{prop_b_existence}), we have
\begin{align*}
		v_0'(0+) 
		%=\frac{r}{\Phi(q+r)}\frac{W^{(q)\prime}(0+)}{Z^{(q)\prime}(0,\Phi(q+r))}
		 =\frac{rW^{(q)\prime}(0+)}{\Phi(q+r)(\Phi(q+r)-rW^{(q)}(0))}\leq1.
\end{align*}
On the other hand, \eqref{v_zero_prime_rewritten} gives
\[
v_0'(0+)=K+ r\frac{\int_0^{\infty}g(t,0)\mu^{(q+r)} (\diff t )}{\Phi(q+r)(\Phi(q+r)-rW^{(q)}(0))}, \]
and hence
%}
%Then
\[
 r\frac{\int_0^{\infty}g(t,0)\mu^{(q+r)} (\diff t )}{\Phi(q+r)(\Phi(q+r)-rW^{(q)}(0))}\leq 1-K.
\]
Using these and \eqref{K_bounds}, for $x > 0$,
%, and because $K \in (0,1)$ as proved above,
\begin{align*}
	v_{0}'(x)
	&= K +  r \frac{\int_0^\infty e^{-x t} g(t,0) \mu^{(q+r)} (\diff t )}{\Phi(q+r)(\Phi(q+r)-rW^{(q)}(0))} \\
	&\leq K +e^{-x p}  r\frac{\int_0^\infty g(t,0) \mu^{(q+r)} (\diff t )}{\Phi(q+r) (\Phi(q+r)-rW^{(q)}(0))} \leq K +e^{-x p} (1-K) \leq 1.
\end{align*}
	\end{proof}
Next, by the application of Proposition \ref{prop_slope_condition} the following result is immediate.
\begin{lemma}\label{cond_max_v}
	For $b^*\geq0$ we have
	\begin{equation*} %\label{max_cond}
	\max_{0\leq l\leq x} \{ l+v_{b^*}(x-l)-v_{b^*}(x) \} =
	\begin{cases} 0 &\mbox{if } x \in[0,b^*], \\ 
	x-b^*+v_{b^*}(b^*)-v_{b^*}(x) & \mbox{if } x \in (b^*, \infty).%\\
%	(x-b^*+v_{b^*}(b^*)-v_{b^*}(x))^+ & \mbox{if } x \in (b'',\infty). 
\end{cases}
	\end{equation*}
\end{lemma}

By Lemmas \ref{generator_on_v} and \ref{cond_max_v}, $v_{b^*}$ satisfies the variational inequality \eqref{HJB-inequality}.  Hence, by Lemma \ref{verificationlemma}, we have the optimality of the periodic barrier strategy $\pi^{b^*}$, and the value function is given by $v = v_{b^*}$.

\section{Convergence to the classical case} \label{section_convergence}%\red{[made this section]}

In this section, we analyze the behavior of the optimal barrier $b^*$ and the value function $v_{b^*}$ with respect to the parameter $r$. Solely in this section, we write $h^{(r)}$, $b^*_r$, and $v^{(r)} = v_{b^*_r}^{(r)}$, to stress the dependence on $r > 0$.

%\red{[moved this here.]}

\begin{lemma} \label{lemma_convergence_r_inf} %\red{[JL: added this. Could you check?]} 
%\red{[combined the remark and lemma.]}
\begin{enumerate}
	\item The optimal periodic barrier $b^*_r$ is increasing in $r$.   
	\item We have $b^*_r \rightarrow \bar{b}$ as $r \rightarrow \infty$.
	\item When $W^{(q)\prime}(0+) < \infty$, $b^*_r$ is zero for sufficiently small $r$.  When $W^{(q)\prime}(0+) = \infty$, $b^*_r \rightarrow 0$ as $r \rightarrow 0$.
	\end{enumerate}
\end{lemma}
\begin{proof}
%Because $b^* \leq \bar{b}$. We can assume without loss of generality that $\bar{b} > 0$.

%Solely in this proof, we write $h^{(r)}$ and $b^*_r$ to stress the dependence on $r > 0$.

(1) For $b > 0$, integration by parts applied to \eqref{039} and the use of Remark \ref{remark_scale_function_completely_monotone} (which implies $e^{-\Phi(q+r)y} W^{(q)\prime \prime}(y) \xrightarrow{y \rightarrow \infty} 0$) give
\begin{align*}
h^{(r)}(b)
%&=-r\int_b^\infty e^{-\Phi(q+r) y} W^{(q)\prime \prime} (y) \diff y \\
&= - \frac r {\Phi(q+r)} \Big[ e^{-\Phi(q+r) b} W^{(q)\prime \prime} (b) +\int_b^\infty e^{-\Phi(q+r) y} W^{(q)\prime \prime \prime} (y) \diff y\Big].
\end{align*}
Because the third derivative $W^{(q)\prime \prime \prime}$ is always positive as described in Remark \ref{remark_scale_function_completely_monotone} (2),
\begin{align*}
\tilde{h}^{(r)} (b) := \frac  {\Phi(q+r)} r   e^{\Phi(q+r) b} h^{(r)}(b)
&= -  \Big[ W^{(q)\prime \prime} (b) +\int_0^\infty e^{-\Phi(q+r) y} W^{(q)\prime \prime \prime} (y+b) \diff y\Big] \\ &< - W^{(q) \prime \prime} (b), \quad b > 0.
\end{align*}
By Remark \ref{remark_scale_function_completely_monotone} (2) and because $r \mapsto \Phi(q+r)$ is increasing, $r \mapsto \tilde{h}^{(r)}(b)$ is increasing for all $b > 0$.
%Differentiating this, 
%\begin{align*}
%\frac \partial {\partial r}\tilde{h}^{(r)} (b) 
%&=  \Phi'(q+r) \int_0^\infty e^{-\Phi(q+r) y} W^{(q)\prime \prime \prime} (y+b) \diff y > 0, \quad b > 0, \quad r > 0.
%\end{align*}
This directly implies that $b^*_r = \sup \{ b > 0: \tilde{h}^{(r)}(b) > 0 \}$ (with $\sup \emptyset = 0$) is increasing in $r$. 

(2) When $\bar{b} = 0$, the convergence is immediate because $b^*_r = 0$ for all $r > 0$.  Hence, we assume that $\bar{b} > 0$.

By considering (1) and because $b_r^* < \bar{b}$ for all $r  > 0$ as in \eqref{b_star_b_bar}, there exists 
\begin{align*}
0 \leq b^*_\infty := \lim_{r \rightarrow \infty} b^*_r \leq \bar{b}.
\end{align*}

To show $b^*_\infty = \bar{b}$, assume, to derive a contradiction, that $b^*_\infty < \bar{b}$. This and Remark \ref{remark_scale_function_completely_monotone} (2) imply that $W^{(q)\prime \prime} (b^*_\infty) < 0$. 

(i) Suppose $b^*_\infty > 0$.  
%Then, it can be shown by \eqref{W_q_limit} that $\int_0^\infty e^{-\Phi(q+r) y} W^{(q)\prime \prime \prime} (y+b^*_\infty) \diff y$ is finite for any $r > 0$ and vanishes in the limit as $r \rightarrow \infty$. 
By Remark \ref{remark_scale_function_completely_monotone} (1) and an application of Fubini's Theorem, we have 
\begin{align*}
	\int_0^{\infty}e^{-\Phi(q+r)y}W^{(q)\prime\prime\prime}(y+b^*_{\infty}) \diff y &=\Phi'(q)\Phi^3(q)\int_0^{\infty}e^{-\Phi(q+r)y}e^{\Phi(q)(y+b^*_{\infty})}\diff y\\
	&+\int_0^{\infty}\int_0^{\infty}t^{3}e^{-t(y+b^*_{\infty})}e^{-\Phi(q+r)y}\mu^{(q)}(\diff t ) \diff y \\
&=\frac{\Phi'(q)\Phi^3(q)}{(\Phi(q+r)-\Phi(q))}e^{\Phi(q)b^*_{\infty}}+\int_0^{\infty}\frac{t^3}{t+\Phi(q+r)}e^{-tb^*_{\infty}}\mu( \diff t),
\end{align*}	
which is finite for any $r>0$ and vanishes in the limit as $r\to\infty$, because $b^*_{\infty}\in(0,\infty)$ and $\mu$ is a finite measure.
Hence, we can take a sufficiently large $r'$, such that
\begin{align*}
\tilde{h}^{(r')} (b^*_\infty) &= -  \Big[ W^{(q)\prime \prime} (b^*_\infty) +\int_0^\infty e^{-\Phi(q+r') y} W^{(q)\prime \prime \prime} (y+b^*_\infty) \diff y\Big]
\end{align*}
is positive. However, this contradicts with $\tilde{h}^{(r)} (b^*_\infty) \leq 0$ for all $r > 0$ (which is implied by $b^*_\infty \geq b^*_r$ and Proposition \ref{prop_b_existence} (i)).
Hence, we must have $b^*_\infty = \bar{b}$ for the case $b^*_\infty > 0$.

(ii) Suppose $b_\infty^* = 0$. In this case, $h^{(r)}$ (and hence $\tilde{h}^{(r)}$ as well) is uniformly negative for all $r > 0$ by  Proposition \ref{prop_b_existence} (ii). 
Then, the assumption ($0 = b^*_\infty < \bar{b}$) implies $W^{(q)\prime \prime} (x) < 0$ on $(0, \bar{b})$ in view of Remark  \ref{remark_scale_function_completely_monotone} (2). Take any $0 < \epsilon < \bar{b}$, then we have
% We have \blue{[How about: "Take any $0 < \epsilon < \bar{b}$, then we have]}
\begin{align*}
\tilde{h}^{(r')} (\epsilon) &= -  \Big[ W^{(q)''} (\epsilon) +\int_0^\infty e^{-\Phi(q+r') y} W^{(q)\prime \prime \prime} (y+\epsilon) \diff y\Big],
\end{align*}
which can be shown to be positive for a sufficiently large $r'$, using the same argument as that in (i); this contradicts with the uniform negativity of $\tilde{h}^{(r')}$. Hence, we must have $\bar{b}= 0 = b^*_\infty$.

(3) In the case $W^{(q)\prime}(0 +) < \infty$, by taking $r\to 0$ in  \eqref{b_zero_criteria}, we see that $h^{(r)}(0+) < 0$, and hence $b_r^*=0$ for small enough $r>0$. 
%This suggests
%to take all the money and run at the first dividend payment opportunity if one needs to expect a long time until the next dividend-decision time. 
%\red{[JL: I guess this holds only when $W^{(q)\prime}(0 +) < \infty$?  But $b^* \xrightarrow{r \downarrow 0} 0$ because $g(b)$ is negative for sufficiently small $r$ for any fixed $b>0$?]}
\par In the case $W^{(q)\prime}(0 +) = \infty$, by using the first equality of \eqref{ot_cond},  we have, for $b > 0$,
\begin{align*}
\lim_{r\to0}h^{(r)}(b)=-\Phi(q)^{2}<0.	
\end{align*}	
Hence, for any fixed $b >0$, we have $h^{(r)}(b)<0$ (and hence $b_r^* < b$ by the form of $h^{(r)}$) for sufficiently small $r>0$; this shows that $b_r^* \xrightarrow{r \downarrow 0} 0$.
\end{proof}

We now show the convergence of $v^{(r)}$ to the value function in the classical case $\bar{v}$ as described in \eqref{classical_value_function}.
	\begin{proposition}
	As $r \rightarrow \infty$, $v^{(r)}(x)$ converges to  $\bar{v}(x)$ for all $x \geq 0$.
	\end{proposition}
	\begin{proof}
	%In order to spell out the dependence on $r$, let us write $v^{(r)} = v_{b^*_r}^{(r)}$.

	(i) Suppose $\bar{b} > 0$.  
	
	Fix $0 \leq x < \bar{b}$.  Because $b^*_r$ increases to $\bar{b}$ by Lemma \ref{lemma_convergence_r_inf}, we can choose a sufficiently large $\bar{r}$ such that for all $r > \bar{r}$, $b_r^* > x$, and hence $v^{(r)}(x) = W^{(q)}(x)/ W^{(q) \prime}(b^*_r)$.  This converges to $\bar{v}(x) = W^{(q)}(x)/ W^{(q) \prime}(\bar{b})$ by Lemma \ref{lemma_convergence_r_inf} and because $W^{(q)\prime}$ is continuous by Remark \ref{remark_scale_function_completely_monotone}.
	
	Fix $x = \bar{b}$.  Then we have $v^{(r)} (b^*_r) = W^{(q)}(b^*_r)/ W^{(q) \prime}(b^*_r)$, which is increasing in $r$ by the monotonicity of $b^*_r$ (as in Lemma \ref{lemma_convergence_r_inf}) and that of the mapping $y \mapsto W^{(q)}(y)/ W^{(q) \prime}(y)$ (by (8.22) of \cite{K}).  This together with the monotonicity of $v^{(r)}$ in $x$ (by Proposition \ref{prop_slope_condition}) gives
	\begin{align*}
	v^{(r)} (\bar{b})  > v^{(r)}(b^*_r) =  \frac {W^{(q)}(b^*_r)} {W^{(q) \prime}(b^*_r)} \xrightarrow{r \uparrow \infty} \frac {W^{(q)}(\bar{b})} {W^{(q) \prime}(\bar{b})} = \bar{v}(\bar{b}).
	\end{align*}
	On the other hand, $v^{(r)} (\bar{b}) \leq \bar{v}(\bar{b})$ for all $r > 0$ (because $\mathcal{A}_r \subset \mathcal{A}_\infty$), %\blue{[Do you think is a good idea to write for each $r>0$ $\mathcal{A}_r$ instead of a generic $\mathcal{A}$?]} 
	and thus $v^{(r)} (\bar{b}) \xrightarrow{r \uparrow \infty} \bar{v} (\bar{b})$.

	Fix $x > \bar{b} > 0$ and $r$ sufficiently large such that $b^*_r > 0$. First, we can write \eqref{vf5} as
	\begin{align*}
	v^{(r)}(x)
%&=\frac{rW^{(q)}(b^*)}{\Phi(q+r)Z^{(q)\prime}(b^*,\Phi(q+r))}I_{-b}^{(q,r)}(x-b^*)-r\overline{\overline{W}}^{(q+r)}(x-b^*)\\
	&=\frac{r A^{(r)}(x, b^*_r)}{\Phi(q+r)Z^{(q)\prime}(b^*_r,\Phi(q+r))}  + B^{(r)}(x-b^*_r) =\frac{A^{(r)}(x, b^*_r)}{W^{(q)\prime}(b^*_r)}+ B^{(r)}(x-b^*_r) 
	\end{align*}
where
\begin{align*}
A^{(r)}(y, b) &:=W^{(q)}(y)+r\int_0^{y-b}W^{(q+r)}(y-b-z)W^{(q)}(z+b) \diff z \\ &-r  W^{(q)}(b)  \overline{W}^{(q+r)}(y-b) -W^{(q+r)}(y-b)\frac{Z^{(q)\prime}(b,\Phi(q+r))}{\Phi(q+r)}, \\
B^{(r)}(y) &:=r\left(\frac{1}{\Phi^2(q+r)}W^{(q+r)}(y)-\overline{\overline{W}}^{(q+r)}(y)\right).
\end{align*}
We also set $\tilde{A}^{(r)}(y, b) := A^{(r)}(y, b)/W^{(q)}(b)$.

%We shall first show that, for $b_0 < \bar{b} < b_1 < x$,
%\begin{align*}
%\sup_{b_0 < b < b_1}|\tilde{A}(x, b) - 1| \xrightarrow{r \uparrow \infty} 0.
%\end{align*}

%\green{[changed from $b_1$ to $\bar{b}$ as we already know that $b^*$ is bounded from above by $\bar{b}$.]}
Fix $0 < b_0 < \bar{b} < x$. By using the limiting case of Lemma 5.1 in \cite{PYM} (where the limit can be easily obtained by Lemma B.3 of \cite{PYM} and monotone convergence), we have 
%\red{[I think there was a misshifting between $x$ and $x-b$ below and so I changed as follows.]}\green{Agreed.}
	\begin{align} \label{identity_A}
	\tilde{A}^{(r)}(x, b) = \E_{x-b}(e^{-q\mathbf{e}_r};\mathbf{e}_r<\tau_0^-)+\E_{x-b}\left(e^{-(q+r)\tau_0^-}\frac{W^{(q)}(X(\tau_0^-)+b)}{W^{(q)}(b)};\tau_0^-<\infty\right).
%=I_a^{(q,r)}(x)-W^{(q+r)}(x)\frac{Z^{(q)\prime}(-a,\Phi(q+r))}{W^{(q)}(-a)\Phi(q+r)}.
	\end{align}
	%where $\mathbf{e}_r$ is an independent exponential random variable with parameter $r$.
This gives a bound:
%	\begin{align*}
%	\E_{\red{x-b_0}}(e^{-q\mathbf{e}_r};\mathbf{e}_r<\tau_0^-)\leq \tilde{A}^{(r)}(\red{x-b}, b)
%\leq \E_{\red{x-b_1}}(e^{-q\mathbf{e}_r};\mathbf{e}_r<\tau_0^-)+\E_{\red{x-b_1}}\left(e^{-(q+r)\tau_0^-} \right), \quad \green{b_0 < b < b_1}.
%	\end{align*}
%	\blue{Given that $x-b_1<x-b<x-b_0$ should it be
	\begin{align*}
		\E_{x-\bar{b}}(e^{-q\mathbf{e}_r};\mathbf{e}_r<\tau_0^-)\leq \tilde{A}^{(r)}(x, b)
		\leq \E_{x-b_0}(e^{-q\mathbf{e}_r};\mathbf{e}_r<\tau_0^-)+\E_{x-\bar{b}}\left(e^{-(q+r)\tau_0^-} \right), \quad b_0 < b < \bar{b}.
	\end{align*}		
	Notice that the dominated convergence theorem gives
	\begin{align} \label{lim_e_r_less_tau}
	\E_y(e^{-q\mathbf{e}_r};\mathbf{e}_r<\tau_0^-) 
	%= \E_y(e^{-q\mathbf{e}_1/r};\mathbf{e}_1 / r<\tau_0^-) 
	\xrightarrow{r \uparrow \infty} 1, \quad y > 0,
	\end{align}
	%\green{[talked with Kouji and deleted the middle equality with $e_1/r$ as it is obvious.]}
and $\E_{x-\bar{b}} (e^{-(q+r)\tau_0^-} )  \xrightarrow{r \uparrow \infty} \p_{x - \bar{b}} (\tau_0^-  = 0)= 0$.
Hence, $\sup_{b_0 \leq b \leq \bar{b}}\tilde{A}^{(r)}(x, b) \xrightarrow{r \uparrow \infty}1$. %\red{$\sup_{b_0 \leq b \leq \bar{b}}\tilde{A}^{(r)}(x, b) \xrightarrow{r \uparrow \infty}1$?}%	We have
%	\begin{align*}
%	\sup_{\underline{x} \leq x \leq \overline{x}}\E_x\left(e^{-(q+r)\tau_0^-}\frac{W^{(q)}(X(\tau_0^-)-a)}{W^{(q)}(-a)};\tau_0^-<\infty\right) \leq \E_{\underline{x}}\left(e^{-(q+r)\tau_0^-} \right) \xrightarrow{r \uparrow \infty} 0.
%	\end{align*}
%	On the other hand,
%	\begin{align*}
%	\inf_{\underline{x} \leq x \leq \overline{x}} \E_x(e^{-q\mathbf{e}_r};\mathbf{e}_r<\tau_0^-) \geq   \E_{\underline{x}}(e^{-q\mathbf{e}_r};\mathbf{e}_r<\tau_0^-) \xrightarrow{r \uparrow \infty} 1.
%	\end{align*}
%	Hence,
%	\begin{align*}
%\sup_{\underline{x} \leq x \leq \bar{x}} A(x, b^*) \xrightarrow{r \uparrow \infty} W^{(q)}(b^*).
%	\end{align*}
%	%converges to $1$ uniformly in $\underline{x} \leq x \leq \overline{x}$.
	
	In contrast, by (the limiting case of) Lemma 5.2 in \cite{PYM}, which holds for any spectrally negative L\'evy process, we have
	\begin{align*}
	B^{(r)}(x-b) = \E_{x-b}\left(e^{-q\mathbf{e}_r}X(\mathbf{e}_r);\mathbf{e}_r<\tau_0^-\right).
	%=r\left(\frac{W^{(q+r)}(x-b)}{\Phi^2(q+r)}-\overline{\overline{W}}^{(q+r)}(x-b)\right).
	\end{align*}
This gives
	\begin{align*}
	B^{(r)}(x-b) - (x-b) &= \E_{x-b}\left(e^{-q\mathbf{e}_r}(X(\mathbf{e}_r)-(x-b));\mathbf{e}_r<\tau_0^-\right) - (x-b) (1 - \E_{x-b}\left(e^{-q\mathbf{e}_r};\mathbf{e}_r<\tau_0^-\right)) \\
	&= \E\left(e^{-q\mathbf{e}_r} X(\mathbf{e}_r);\mathbf{e}_r<\tau_{-(x-b)}^-\right) - (x-b) (1 - \E_{x-b}\left(e^{-q\mathbf{e}_r};\mathbf{e}_r<\tau_0^-\right)).
	\end{align*}
	Hence, % for $b_0 \leq b \leq b_1$,
	\begin{align*}
	&\inf_{b_0 \leq b \leq \bar{b}} (B^{(r)}(x-b) - (x-b)) \\ &\geq \E \left(e^{-q\mathbf{e}_r} X(\mathbf{e}_r); X(\mathbf{e}_r) < 0, \mathbf{e}_r<\tau_{-(x-b_0)}^-\right) - (x-b_0) (1 - \E_{x-\bar{b}}\left(e^{-q\mathbf{e}_r};\mathbf{e}_r<\tau_0^-\right)) \\ &\xrightarrow{r \uparrow \infty} 0,
	\end{align*}
	where the convergence holds by \eqref{lim_e_r_less_tau} and because the fact that $|X(\mathbf{e}_r)|\leq (x-b_0)$ on the event $\{X(\mathbf{e}_r) < 0, \mathbf{e}_r<\tau_{-(x-b_0)}^-\}$, implies, by dominated convergence, that	\begin{align*} 
	\E \left(e^{-q\mathbf{e}_r} X(\mathbf{e}_r); X(\mathbf{e}_r) < 0, \mathbf{e}_r<\tau_{-(x-b_0)}^-\right) 
	%= \E \left(e^{-q\mathbf{e}_1/r} X(\mathbf{e}_1/r); X(\mathbf{e}_1/r) < 0, \mathbf{e}_1/r<\tau_{-(x-b_0)}^-\right)  
	\xrightarrow{r \uparrow \infty} 0.
	\end{align*}
	%\green{[talked with Kouji and deleted the middle equality with $e_1/r$ as it is obvious.]}
	On the other hand, 
	\begin{align}
	\sup_{b_0 \leq b \leq \bar{b}} (B^{(r)}(x-b) - (x-b)) \leq \E\left(e^{-q\mathbf{e}_r} \overline{X}(\mathbf{e}_r);\mathbf{e}_r<\tau_{-(x-b_0)}^-\right) \leq  \E\left(\overline{X}(\mathbf{e}_r) \right) \xrightarrow{r \uparrow \infty} 0, \label{B_conv}
	\end{align}
where the last limit holds because $\overline{X}(\mathbf{e}_r)$ is exponentially distributed with parameter $\Phi(r)$.
	Hence, $\sup_{b_0 \leq b \leq \bar{b}} |B^{(r)}(x-b)-(x-b)| \xrightarrow{r \uparrow \infty} 0$.	
Now, by using these uniform convergence results together with $b^*_r \nearrow \bar{b}$, 
	\begin{align*}
	v^{(r)}(x)
	 =\frac{W^{(q)}(b^*_r)}{W^{(q)\prime}(b^*_r)} \tilde{A}^{(r)}(x, b^*_r)+ B^{(r)}(x-b^*_r) \xrightarrow{r \uparrow \infty} \frac{W^{(q)}(\bar{b})}{W^{(q)\prime}(\bar{b})} + x-\bar{b}, 
	\end{align*}
	as desired.

%Now, suppose the case $x = \bar{b}$. Assume for contradiction that $v_{b^*}(\bar{b})$ does not converge to $\bar{v}(\bar{b})$. This means together with the fact that $v_{b^*}(\bar{b}) \leq \bar{v}(\bar{b})$ that we can take $\epsilon > 0$ such that
%\begin{align*}
%\liminf_{r \rightarrow \infty} v_{b^*}(\bar{b}) + \epsilon < \frac {W^{(q)} (\bar{b})} {W^{(q) \prime} (\bar{b})}.
%\end{align*}
%Thanks to the continuity of $ \frac {W^{(q)} (\bar{b})} {W^{(q) \prime} (\bar{b})}$, we can take sufficiently small $\delta > 0$ such that 
%\begin{align*}
%	\frac {W^{(q)} (\bar{b} - \delta)} {W^{(q) \prime} (\bar{b} - \delta)} > \frac {W^{(q)} (\bar{b})} {W^{(q) \prime} (\bar{b})} - \epsilon
%	\end{align*}
%	On the other hand, because $b^*_r \xrightarrow{r \uparrow \infty} \bar{b}$, we can take $r'$ such that $b^*_r > \bar{b} - \delta$ and
%\begin{align*}
%v_{b^*}^{(r)}(b_r^*) > v_{b^*}^{(r)}(\bar{b}) 
%\end{align*}
%which contradicts with the monotonicity.

	(ii) Suppose $\bar{b} = 0$.  Notice in this case that $b^*_r = 0$ for all $r > 0$, and thus it suffices to show the (pointwise)  convergence of 
		\begin{align*}
	v_{0}^{(r)}(x)
%&=\frac{rW^{(q)}(b^*)}{\Phi(q+r)Z^{(q)\prime}(b^*,\Phi(q+r))}I_{-b}^{(q,r)}(x-b^*)-r\overline{\overline{W}}^{(q+r)}(x-b^*)\\
	&:=\frac{r A^{(r)}(x, 0) }{\Phi(q+r)(\Phi(q+r) - r W^{(q)}(0))} + B^{(r)}(x).	\end{align*}
Here, $B^{(r)}(x) \xrightarrow{r \uparrow \infty} x$ as a special case of \eqref{B_conv} when $x > 0$.  This also holds for $x = 0$ because $0 \leq B^{(r)}(0) = \E\left(e^{-q\mathbf{e}_r}X(\mathbf{e}_r);\mathbf{e}_r<\tau_0^-\right) \leq \E\left(\overline{X}(\mathbf{e}_r) \right) \xrightarrow{r \uparrow \infty} 0$.

On the other hand, by using \eqref{W_q_W_q_r_relation}, we have
\begin{align*}
A^{(r)}(x, 0) &=W^{(q)}(x)+r\int_0^{x}W^{(q+r)}(x-y)W^{(q)}(y) \diff y -r  W^{(q)}(0)  \overline{W}^{(q+r)}(x) \\ &-W^{(q+r)}(x)\frac{\Phi(q+r) - r W^{(q)}(0)}{\Phi(q+r)} \\
&=r W^{(q)}(0)  \Big(- \overline{W}^{(q+r)}(x) +\frac{ W^{(q+r)}(x)}{\Phi(q+r)} \Big).
\end{align*}
Suppose $X$ is of unbounded variation.  Then, $A^{(r)}(x,0) = 0$, and hence it is clear that $v_0^{(r)}(x) \xrightarrow{r \uparrow \infty} x = \bar{v}(x)$, as desired.

%Suppose $X$ is of bounded variation.  Because $q+r = \psi(\Phi(q+r)) = c \Phi(q+r) +\int_{(-\infty,0)}\big( {\rm e}^{\Phi(q+r) z}-1\big)\Pi(\ud z)$ by \eqref{psi_bounded_var}, we have
%\begin{equation*}
%	   \Phi(q+r) -  \frac r c = \frac q c-  c^{-1}\int_{(-\infty,0)}\big( {\rm e}^{\Phi(q+r) z}-1\big)\Pi(\ud z) \xrightarrow{r \uparrow \infty} \frac {q + \Pi(-\infty, 0)} {c}.
%\end{equation*}
%\blue{[How about changing it like this to adress comment 16.]
Suppose $X$ is of bounded variation (then by Remark \ref{remark_cond_b_zero}, we have $\Pi(-\infty,0)<\infty$).  Since $q+r = \psi(\Phi(q+r)) = c \Phi(q+r) +\int_{(-\infty,0)}\big( {\rm e}^{\Phi(q+r) z}-1\big)\Pi(\ud z)$ by \eqref{psi_bounded_var}, we have by monotone convergence
	\begin{equation*}
		\Phi(q+r) -  \frac r c = \frac q c+  c^{-1}\int_{(-\infty,0)}\big( 1-{\rm e}^{\Phi(q+r) z}\big)\Pi(\ud z) \xrightarrow{r \uparrow \infty} \frac {q + \Pi(-\infty,0)} {c}.
	\end{equation*}
%\red{[I don't think the following is necessary? Maybe we can just say that in the response?]}\blue{[Added this for comment 17 of the referee] Similarly we note that
%		\begin{equation*}
%			\frac{\Phi(q+r)}{r} -  \frac 1 c = \frac{q}{rc}+  (rc)^{-1}\int_{(-\infty,0)}\big( 1-{\rm e}^{\Phi(q+r) z}\big)\Pi(\ud z) \xrightarrow{r \uparrow \infty} 0.
%		\end{equation*} }
By this and because $\Phi(q+r) c \sim r$ as $r \uparrow \infty$ (see also Remark \ref{remark_smoothness_zero}), we have
\begin{align}
	\begin{split}
	\frac{r}{\Phi(q+r)(\Phi(q+r) - r W^{(q)}(0))}=\frac{r}{\Phi(q+r)(\Phi(q+r) - \frac r c))} \xrightarrow{r \uparrow \infty} \frac {1} {W^{(q)\prime}(0+)}.
\end{split}
	\end{align}
	On the other hand, by noting that $W^{(q)}(0) > 0$, shifting the process by $b$
	%\red{performing} shifting \red{the process} by $b$ \blue{[How about "performing a shift to the process by $b$..." or "shifting the process by $b$"?]} 
	in \eqref{identity_A} and taking $b \downarrow 0$, together with the dominated convergence theorem, give
	\begin{align*}
	\E_x(e^{-q\mathbf{e}_r};\mathbf{e}_r<\tau_0^-)+\E_x\left(e^{-(q+r)\tau_0^-}\frac{\lim_{b \downarrow 0}W^{(q)}(X(\tau_b^-))}{W^{(q)}(0)};\tau_0^-<\infty\right) = \frac {A^{(r)}(x,0)} {W^{(q)}(0)}.
%=I_a^{(q,r)}(x)-W^{(q+r)}(x)\frac{Z^{(q)\prime}(-a,\Phi(q+r))}{W^{(q)}(-a)\Phi(q+r)}.
	\end{align*}
Because $X$ does not creep downward ($\p_x ( X(\tau_0^-) = 0, \tau_0^- < \infty ) = 0$ for all $x \geq 0$) for the case of bounded variation (see Exercise 7.6 of \cite{K}), the second expectation on the left hand side is zero.  Now taking $r \rightarrow \infty$ on both sides, we have $ {A^{(r)}(x,0)} / {W^{(q)}(0)} \xrightarrow{r \rightarrow \infty} 1$. In conclusion, we have  $v_{0}^{(r)}(x) \xrightarrow{r \uparrow \infty}  W^{(q)}(0)/W^{(q)\prime}(0+)+ x = \bar{v}(x)$, as desired.

% and hence $\E_x\left(e^{-(q+r)\tau_0^-}\frac{W^{(q)}(X(\tau_0^-))}{W^{(q)}(0)};\tau_0^-<\infty\right) = 0$.
%
%	
%	$v_0(x)$. Then using the above,
%	\begin{align*}
%	v_{0}(x)
%	&=\frac{rW^{(q)}(0)}{\Phi(q+r)Z^{(q)\prime}(0,\Phi(q+r))}\left(I_{0}^{(q,r)}(x)-W^{(q+r)}(x)\frac{Z^{(q)\prime}(0,\Phi(q+r))}{W^{(q)}(0)\Phi(q+r)}\right)\\
%	&+r\left(\frac{1}{\Phi^2(q+r)}W^{(q+r)}(x)-\overline{\overline{W}}^{(q+r)}(x)\right).
%	\end{align*}
%
%If $W^{(q)}(0) = 0$.  This is zero. 
%
%
%	
%	Hence the above converges to $1$
	\end{proof}

\section{Numerical Examples} \label{section_numerics}

%We have
%\begin{align*}
%Z^{(q)}(x) = 1+ q \frac {e^{\Phi(q) x}-1} {\Phi(q) \psi'(\Phi(q))} + q \sum_{i \in \mathcal{I}_q} \frac {C_{i,q}} {\xi_{i,q}} (e^{-\xi_{i,q}x}-1)
%\end{align*}
%
%We have
%\begin{align*}
%\int_{0}^{x} W^{(q)}(y) e^{-\Phi(q+r) y} \ud y &= \int_{0}^{x} \Big[ \frac {e^{\Phi(q) y}} {\psi'(\Phi(q))} - \sum_{i \in \mathcal{I}_q} C_{i,q} e^{-\xi_{i,q}y}  \Big] e^{-\Phi(q+r) y} \ud y \\
%&= \int_{0}^{x} \Big[ \frac {e^{(\Phi(q)-\Phi(q+r)) y}} {\psi'(\Phi(q))} - \sum_{i \in \mathcal{I}_q} C_{i,q} e^{-(\xi_{i,q}+\Phi(q+r))y}  \Big] \ud y \\
%&=  \frac {e^{(\Phi(q)-\Phi(q+r)) x}-1} {(\Phi(q)-\Phi(q+r))\psi'(\Phi(q))} + \sum_{i \in \mathcal{I}_q} \frac {C_{i,q}} {\xi_{i,q}+\Phi(q+r)} (e^{-(\xi_{i,q}+\Phi(q+r))x}-1)  
%\end{align*}

We conclude this paper with a sequence of numerical experiments.  Here, to better understand the sensitivity with respect to each parameter describing the underlying process, we consider a simple case using a (drifted) compound Poisson process with i.i.d.\ exponential-size jumps, which satisfies Assumption \ref{remark_scale_function_completely_monotone}.  Both cases with and without Brownian motions are considered. 

More specifically, we assume, for some $c \in \R$ and $\sigma \geq 0$, % be a spectrally positive process with i.i.d.\ phase-type distributed jumps \cite{Asmussen_2004} of the form
\begin{equation}
 X(t) - X(0)= c t+\sigma B(t) - \sum_{n=1}^{N(t)} Z_n, \quad 0\le t <\infty, \label{X_phase_type}
\end{equation}
where $B=( B(t); t\ge 0)$ is a standard Brownian motion, $N=(N(t); t\ge 0 )$ is a Poisson process with arrival rate $\kappa$, and  $Z = ( Z_n; n = 1,2,\ldots )$ is an i.i.d.\ sequence of exponential random variables with parameter $\lambda$. The processes $B$, $N$, and $Z$ are assumed mutually independent.  This is a special case of the  spectrally negative version of the \emph{phase-type} \lev process in \cite{Asmussen_2004}, which admits an analytical form of the scale function, as in \cite{Egami_Yamazaki_2010_2}.  We refer the reader to \cite{Egami_Yamazaki_2010_2, KKR} for the forms of the corresponding scale functions.

\subsection{Computation of the value function} We first illustrate the computation scheme of the optimal barrier $b^*$ and the value function $v = v_{b^*}$.  Here, for $X$ in \eqref{X_phase_type}, we consider the following sets of parameters:
\begin{align*}
\textrm{\textbf{Case 1}: } \sigma = 0.2, c=1.5, \quad \textrm{\textbf{Case 2}: } \sigma = 0.2, c=0.1, \quad \textrm{\textbf{Case 3}: } \sigma = 0.2, c=0,
\end{align*}
and 
\begin{align*}
\textrm{\textbf{Case 1'}: } \sigma = 0, c=1.5, \quad \textrm{\textbf{Case 2'}: } \sigma = 0, c=1.15, \quad \textrm{\textbf{Case 3'}: } \sigma = 0, c=0.1.
\end{align*}
For other parameters, we set $\kappa = \lambda = 1$, $r = 0.5$, and $q = 0.05$.  These parameters are chosen so that $b^* > 0$ for \textbf{Cases 1} and \textbf{1'}, $0 = b^* < \bar{b}$ for \textbf{Cases 2} and \textbf{2'}, and $0 = b^* = \bar{b}$ for \textbf{Cases 3} and \textbf{3'}, where we recall that $\bar{b}$ is the optimal barrier in the classical case, as defined in Remark \ref{remark_scale_function_completely_monotone}.

Recall that the optimal barrier $b^*$ is the unique  root of $h=0$ if $h(0+) > 0$ and  zero otherwise.   Figure \ref{figure_h} plots the function $h$ along with the points at $b^*$ and $\bar{b}$.  For \textbf{Cases 1} and \textbf{1'}, $h$ starts at a strictly positive value, decreases until $\bar{b}$, and then increases to zero;  $b^*$ becomes the unique point at which  $h$ vanishes.  For \textbf{Cases 2} and \textbf{2'}  (where $\bar{b} > 0$), $h$ starts at a negative value and then behaves similarly to \textbf{Cases 1} and \textbf{1'}; because $h$ is uniformly negative, we set $b^* = 0$. For \textbf{Cases 3} and \textbf{3'}  (where $\bar{b} = 0$), $h$ starts at a negative value and monotonically increases to zero; again we set $b^* = 0$.

\begin{figure}[htbp]
\begin{center}
\begin{minipage}{1.0\textwidth}
\centering
\begin{tabular}{ccc}
 \includegraphics[scale=0.35]{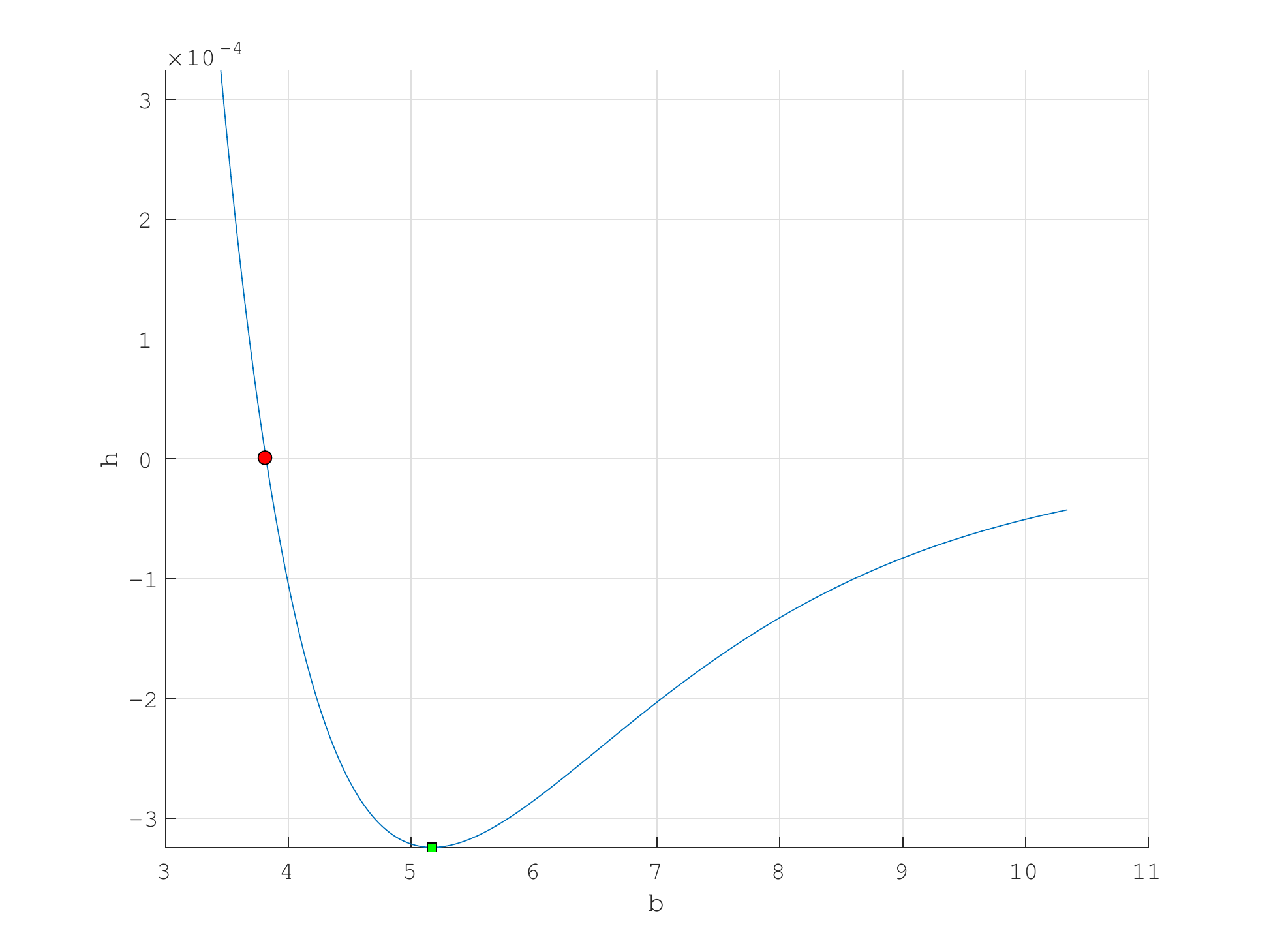} & \includegraphics[scale=0.35]{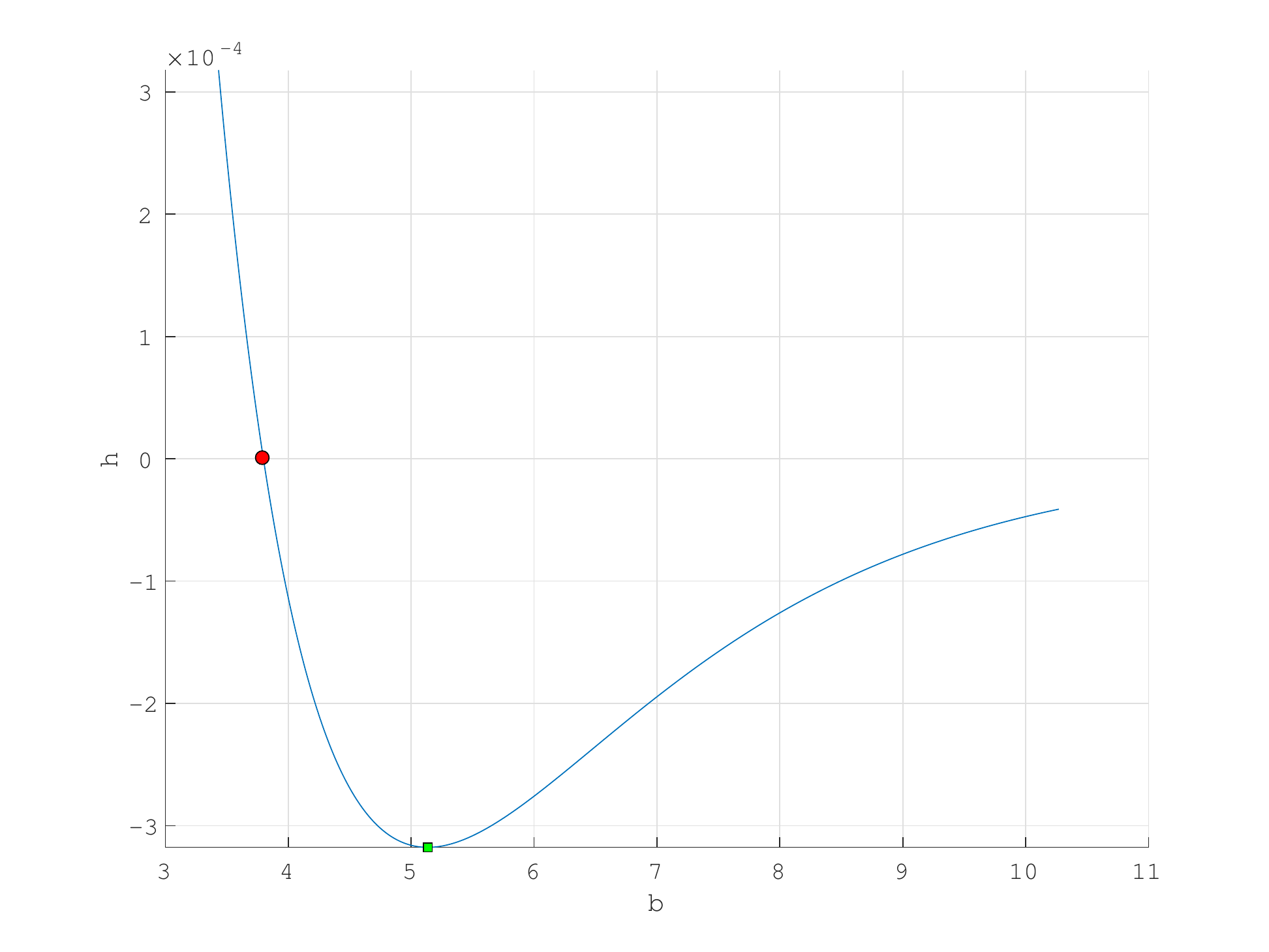}  \\
 Case 1 & Case 1' \\
 \includegraphics[scale=0.35]{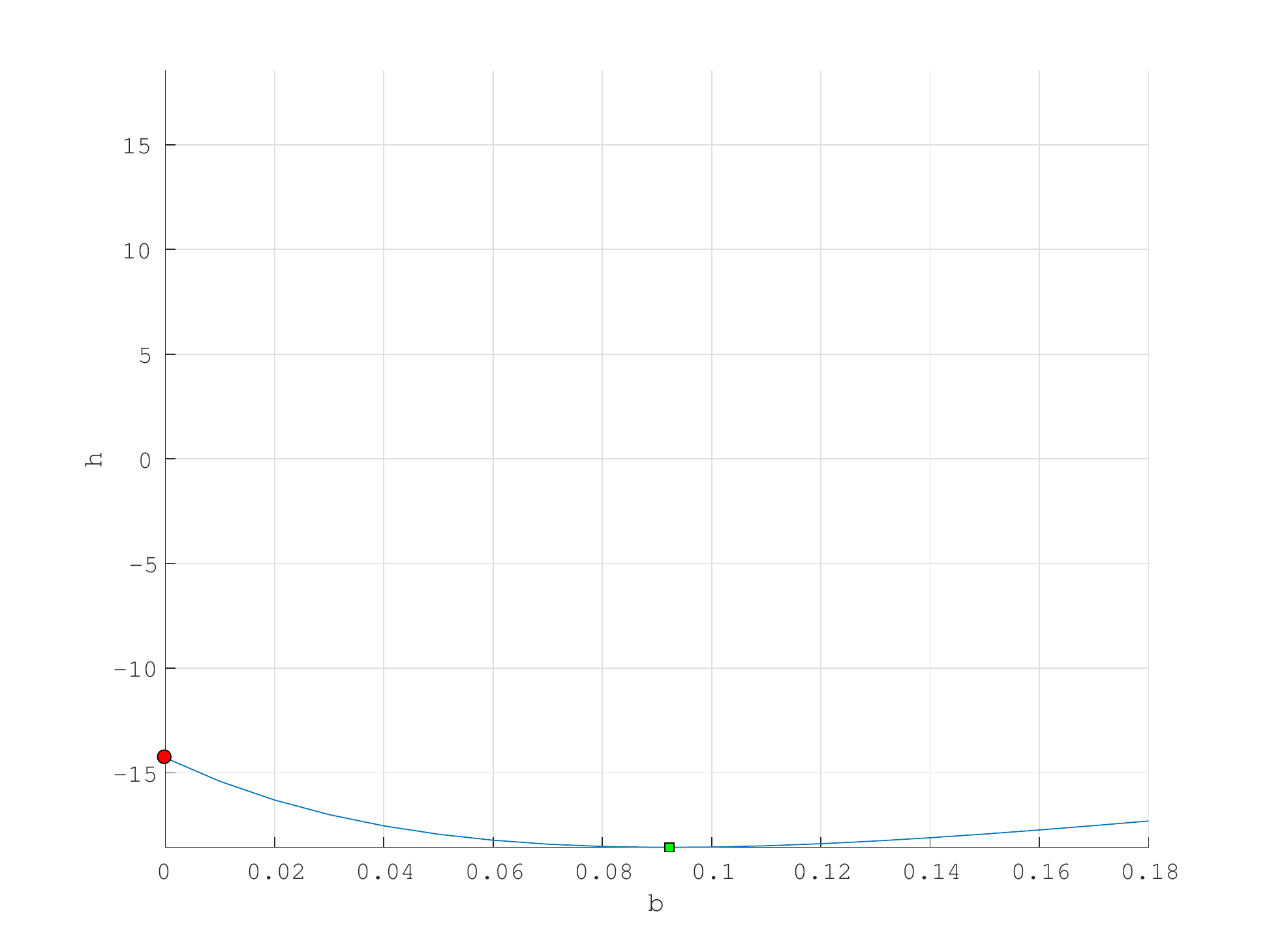} & \includegraphics[scale=0.35]{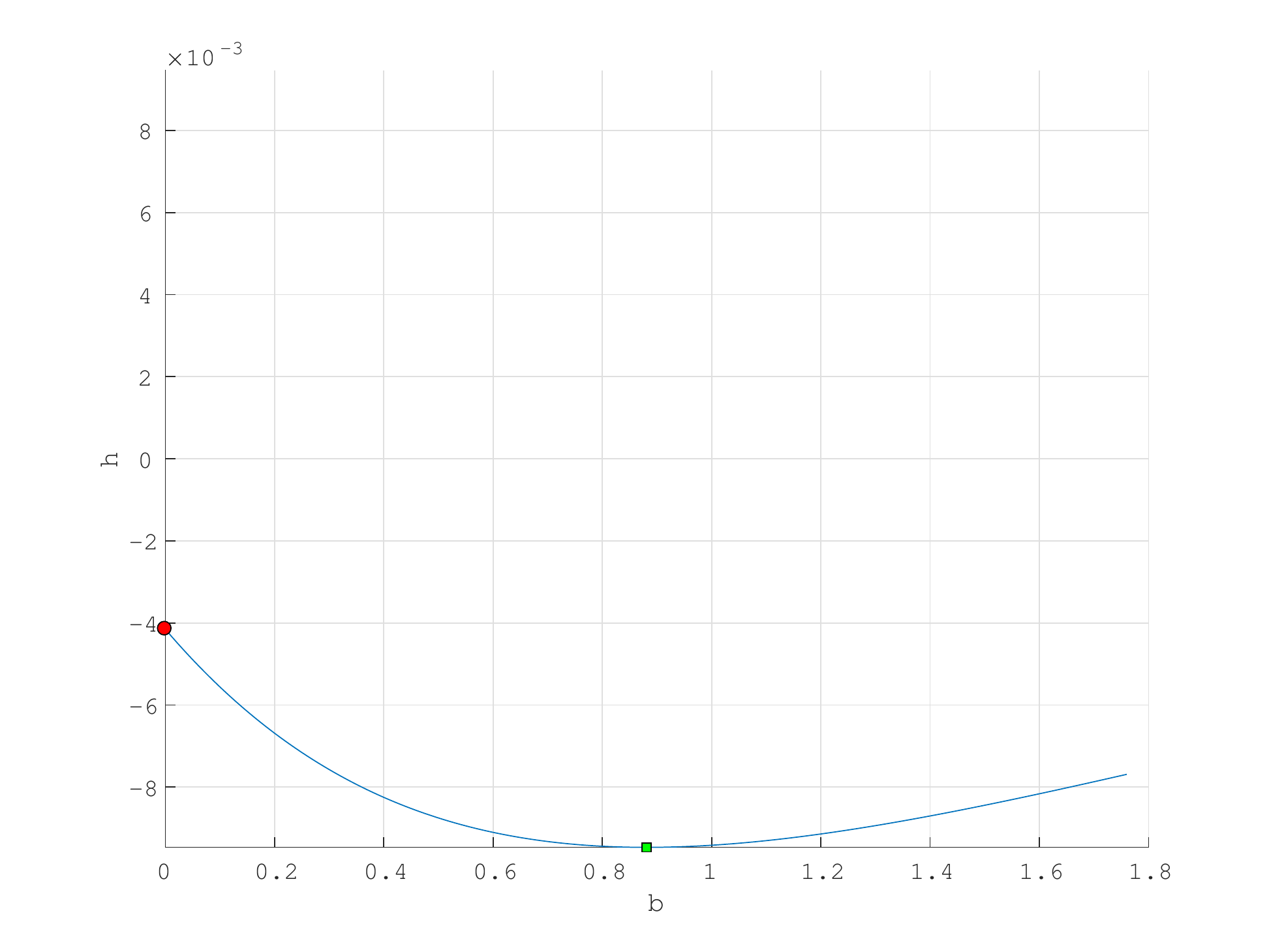}  \\
 Case 2 & Case 2' \\
 \includegraphics[scale=0.35]{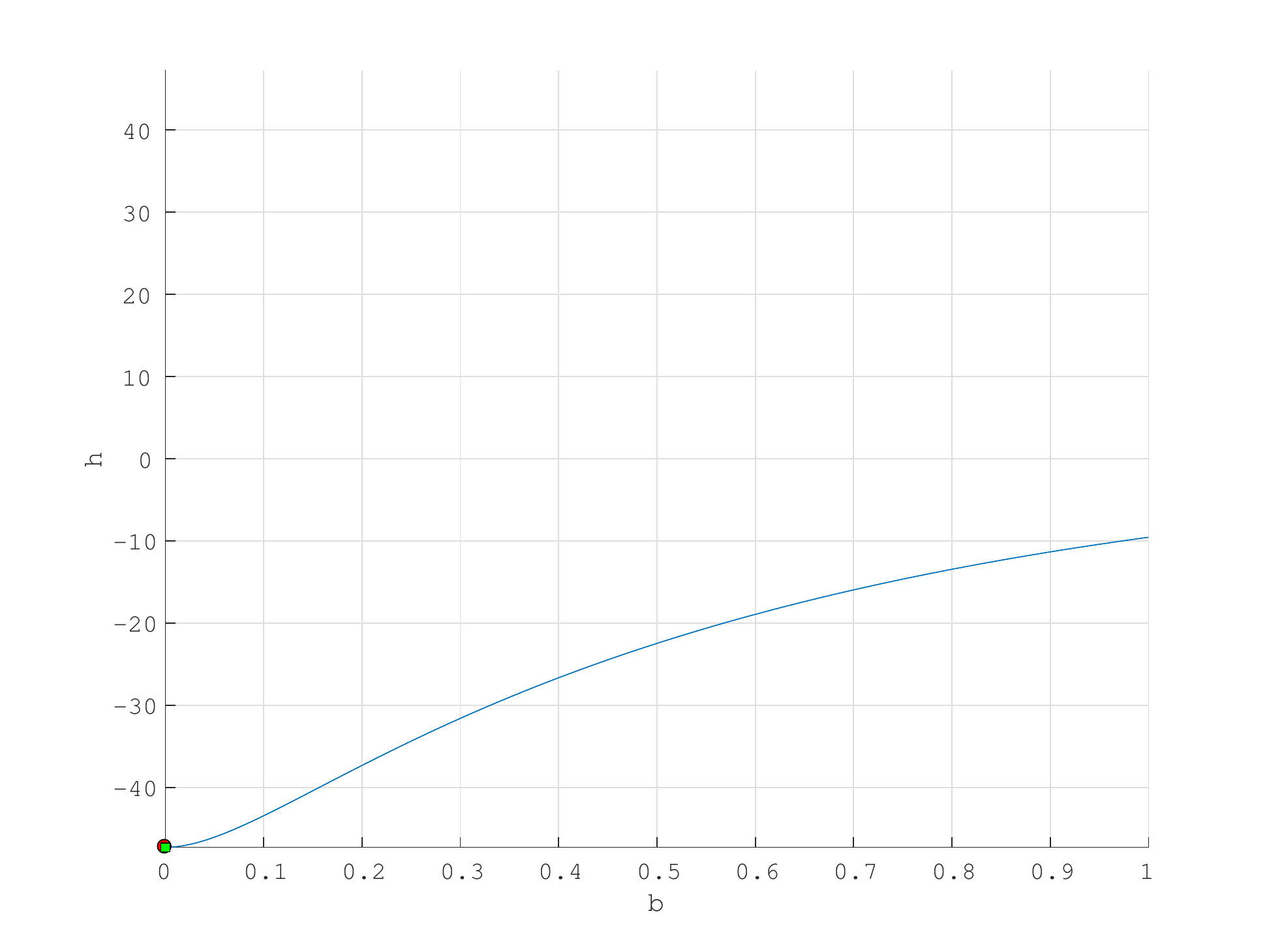} & \includegraphics[scale=0.35]{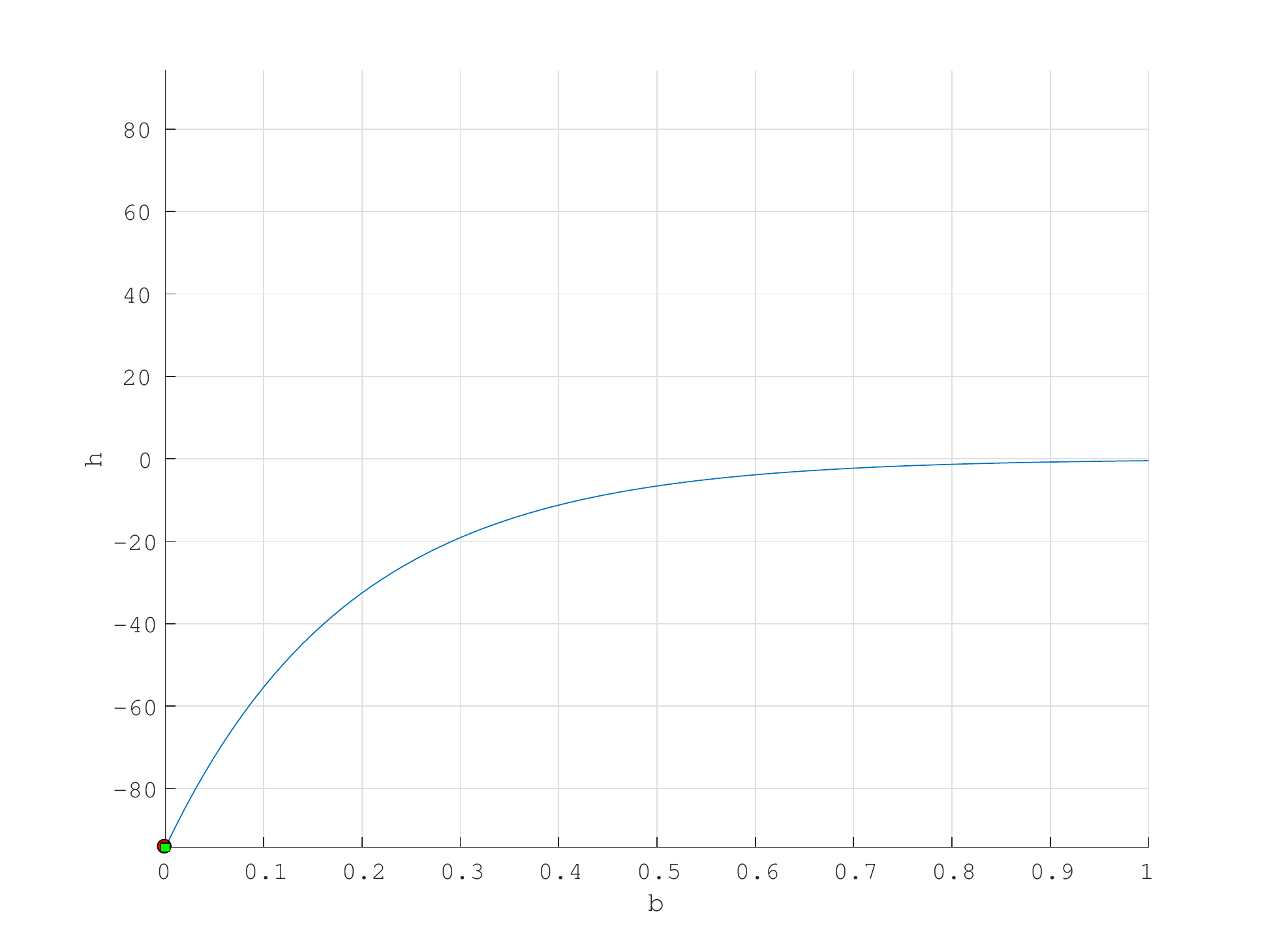}  \\
 Case 3 & Case 3' 
%  \includegraphics[scale=0.35]{fig_f_zero} & \includegraphics[scale=0.35]{fig_value_zero}  \\
 %$b \mapsto h(b)$ & $x \mapsto v$ and $x \mapsto v_b(x)$ 
 \end{tabular}
\end{minipage}
\caption{Plots of $h$ and the points at $b^*$ and $\bar{b}$ indicated by circles and squares, respectively. 
} \label{figure_h}
\end{center}
\end{figure}

% and 3, $h$ starts at negative values and hence it is uniformly negative -- hence we set $b^* = 0$. The difference between these two cases is that, while in Case 2 (where $\bar{b} > 0$), the function $h$ first decreases until $\bar{b}$ and then increases to zero,  in Case 3 (where $\bar{b} = 0$), it is monotonically increasing.  
%Similarly, we plot in  the left column of Figure \ref{optimality_bounded} for the bounded variation cases (Cases 1', 2', and 3').

With the computed values of $b^*$,  the value function $v = v_{b^*}$ is obtained using \eqref{vf5} and \eqref{vf0} for the cases $b^* > 0$ and $b^* = 0$, respectively.  To confirm the optimality, in Figure \ref{figure_v}, we plot $v_{b^*}$ along with suboptimal NPVs $v_{b}$ with $b \neq b^*$.  It can be confirmed in all cases that $v_{b^*}$ dominates $v_{b}$,  for $b \neq b^*$,
 uniformly in $x$.   As shown in Proposition \ref{prop_slope_condition}, $v_{b^*}$ is smooth and its slope is larger than $1$ if and only if $x < b^*$. %\blue{[Kazu deleted the statement about concavity.]}%The value function appears to be concave when $\bar{b} > 0$ (\textbf{Cases 1}, \textbf{1'}, \textbf{2}, and \textbf{2'}) but it is not so when $\bar{b} = 0$ (\textbf{Cases 3} and \textbf{3'}). 
 
 Regarding the comparison between the unbounded and bounded variation cases, the main differences include the degree of smoothness at $b^*$ and the behavior of $v_{b^*}$ in the vicinity of zero.  The degree of smoothness is not visually clear as it is at least twice continuously differentiable in both cases. On the other hand, the difference in the vicinity of zero can be observed:  as the starting value $x$ decreases to zero, $v$ converges to zero for the unbounded variation case (\textbf{Cases 1}, \textbf{2}, and \textbf{3}), but not for the bounded variation case (\textbf{Cases 1'}, \textbf{2'}, and \textbf{3'}).  
 \begin{figure}[htbp]
\begin{center}
\begin{minipage}{1.0\textwidth}
\centering
\begin{tabular}{cc}
 \includegraphics[scale=0.35]{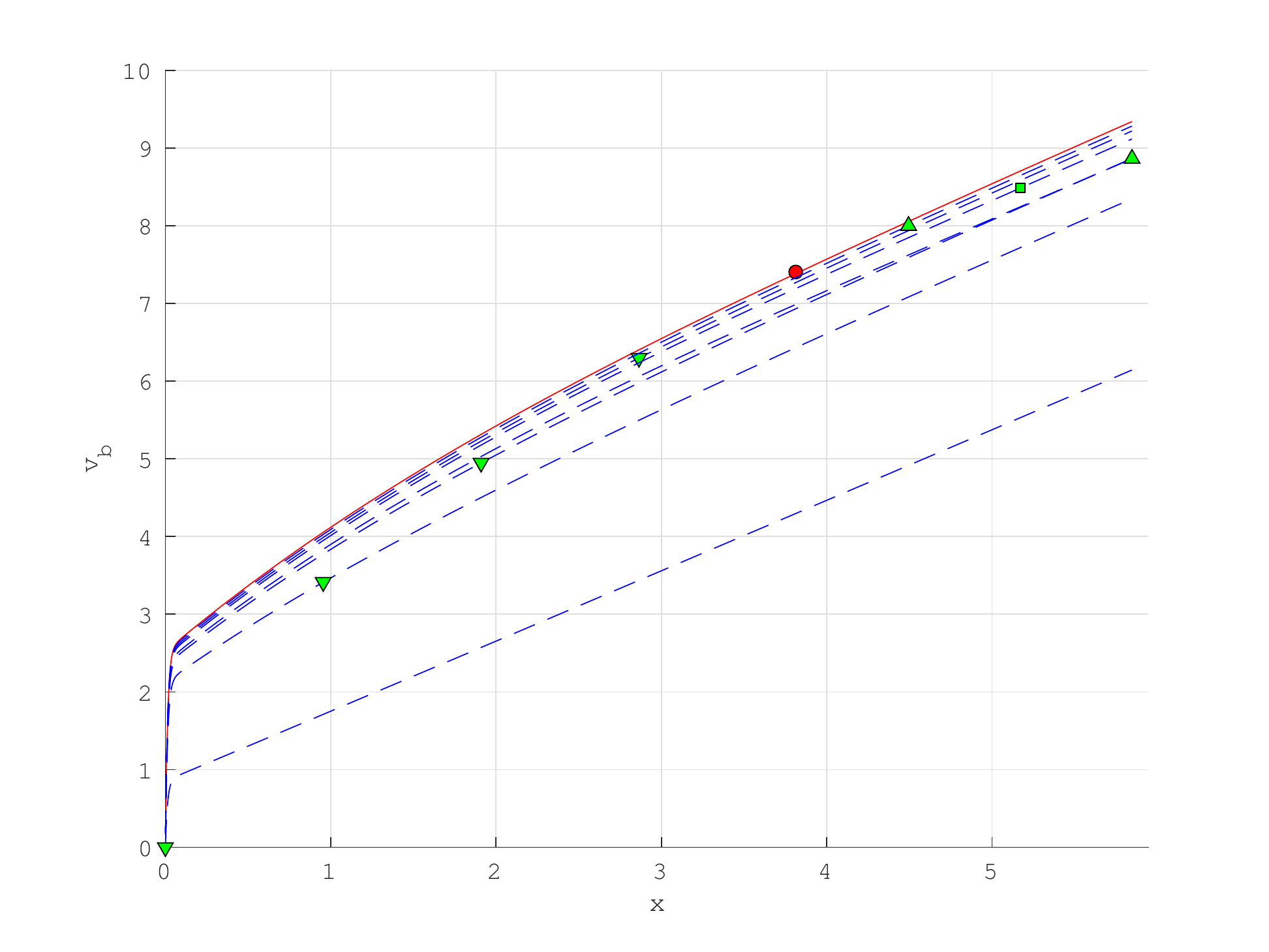} & \includegraphics[scale=0.35]{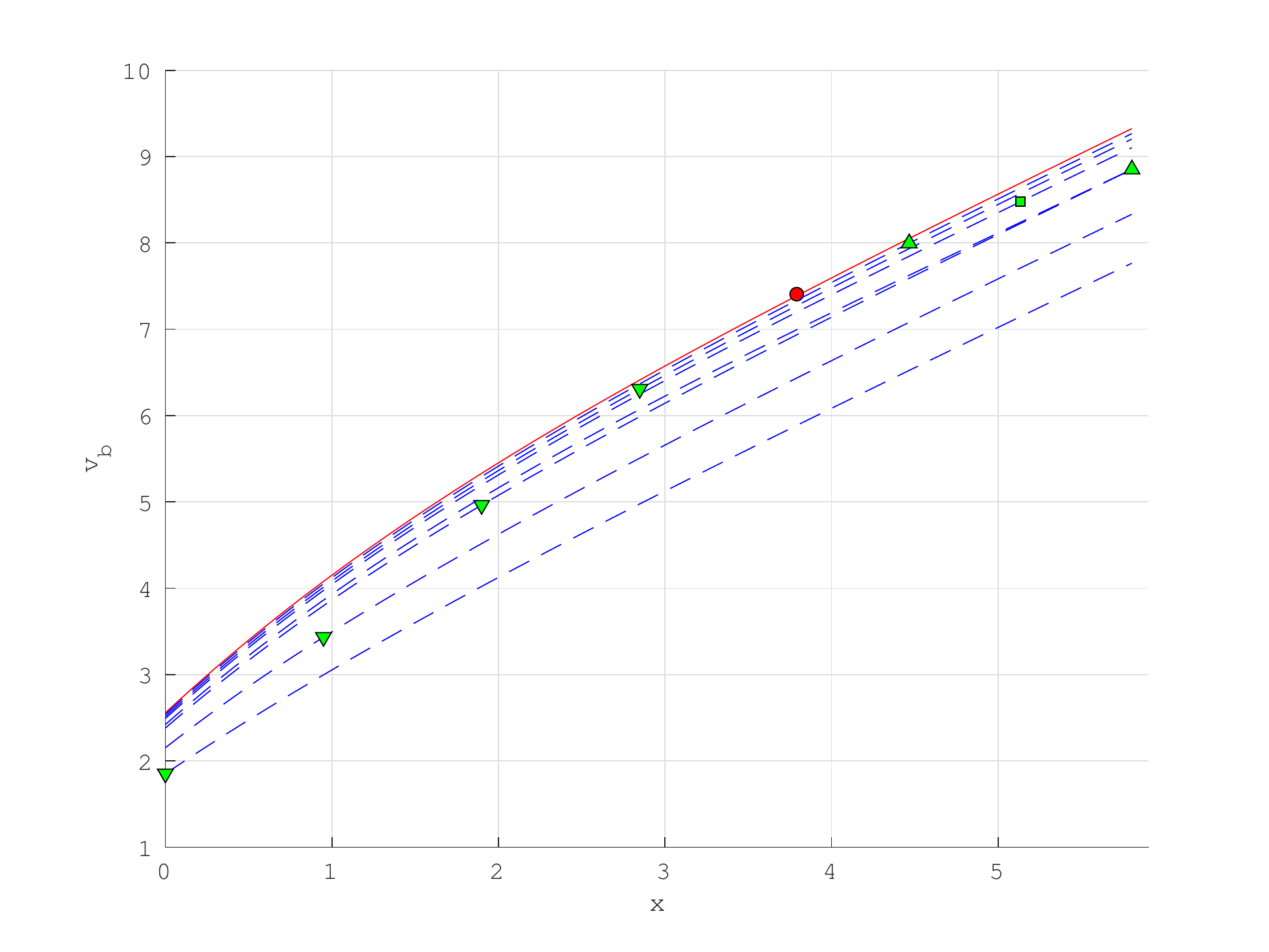}  \\
 Case 1 & Case 1' \\
 \includegraphics[scale=0.35]{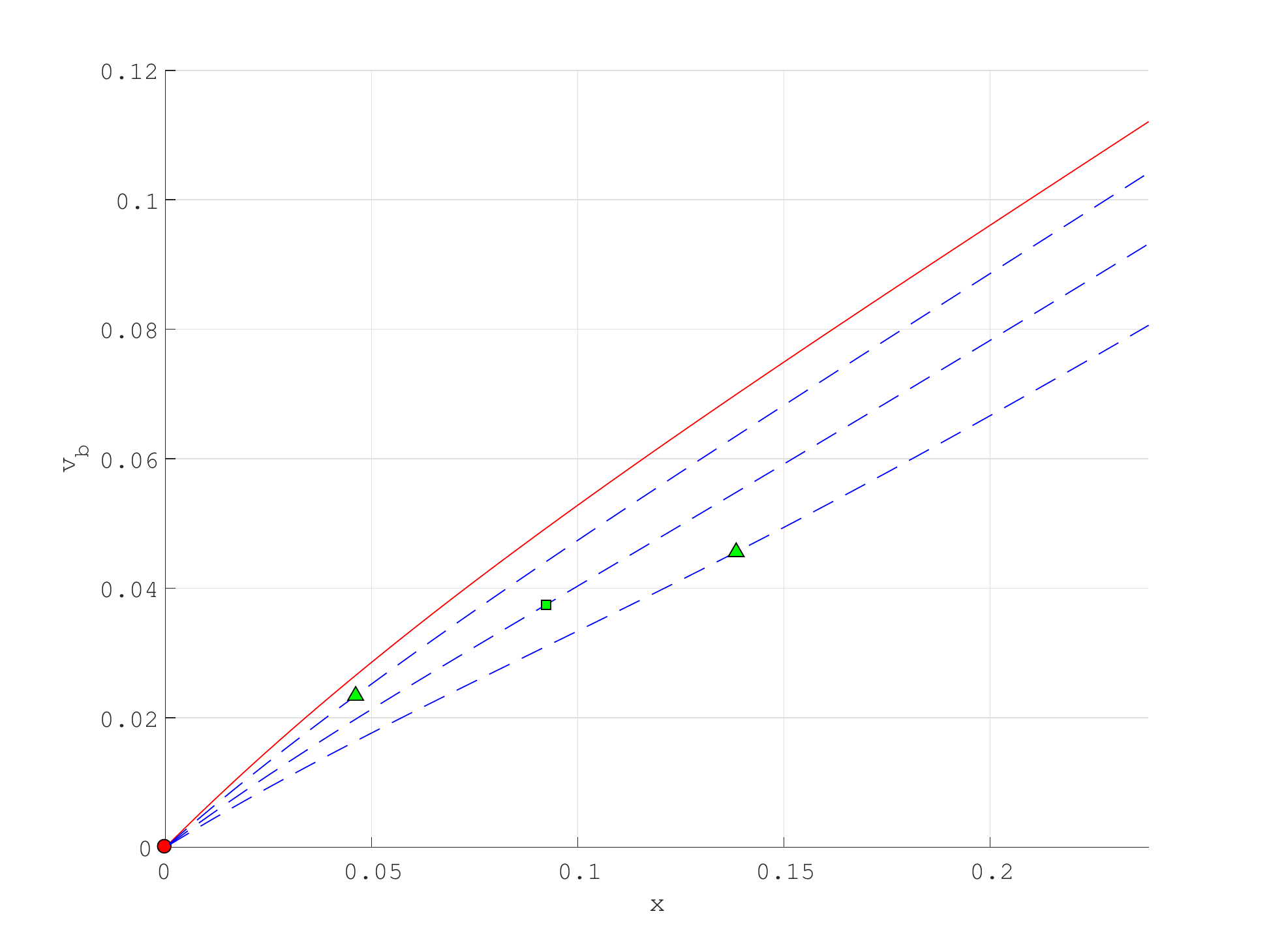} & \includegraphics[scale=0.35]{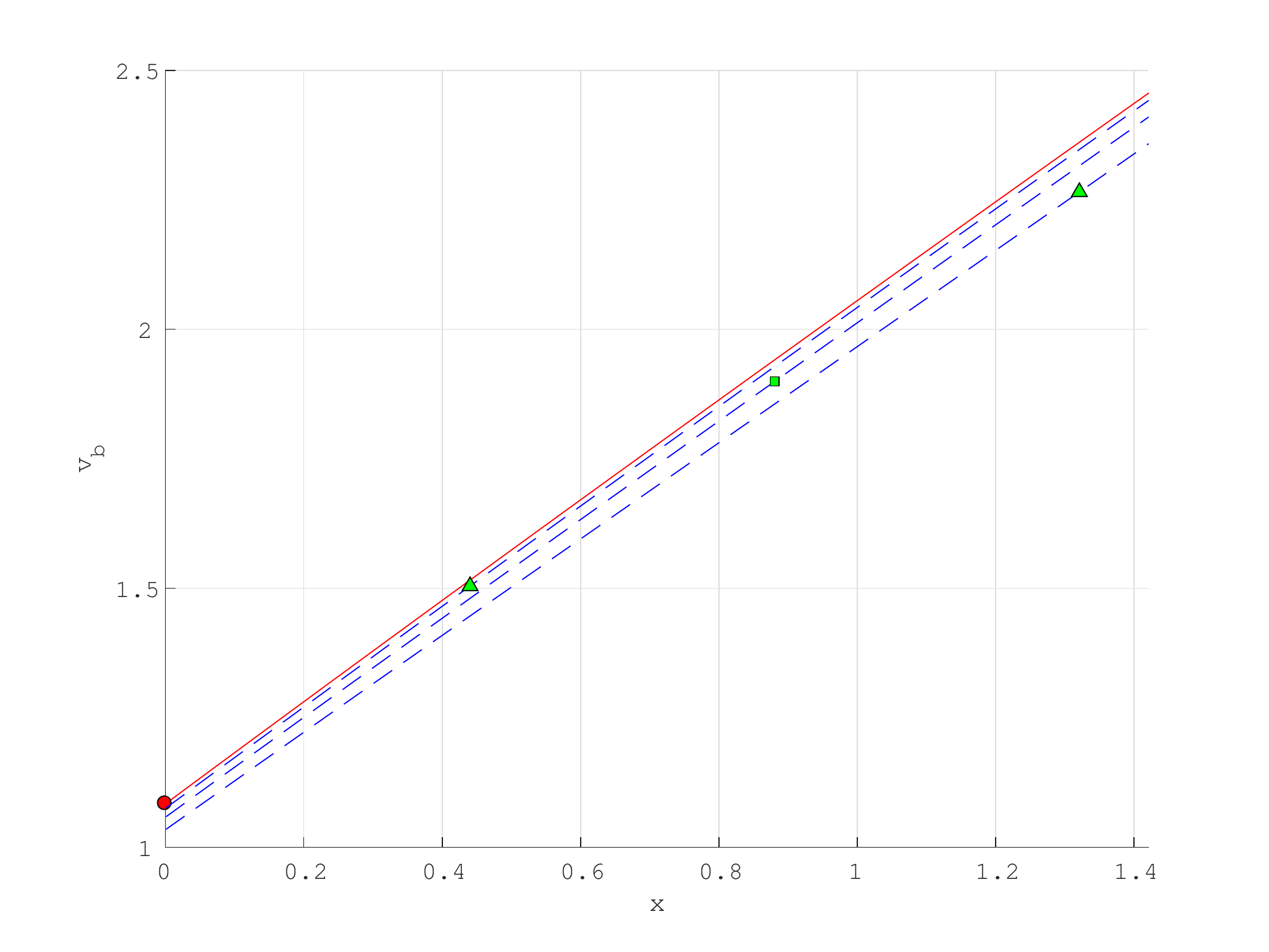}  \\
 Case 2 & Case 2' \\
 \includegraphics[scale=0.35]{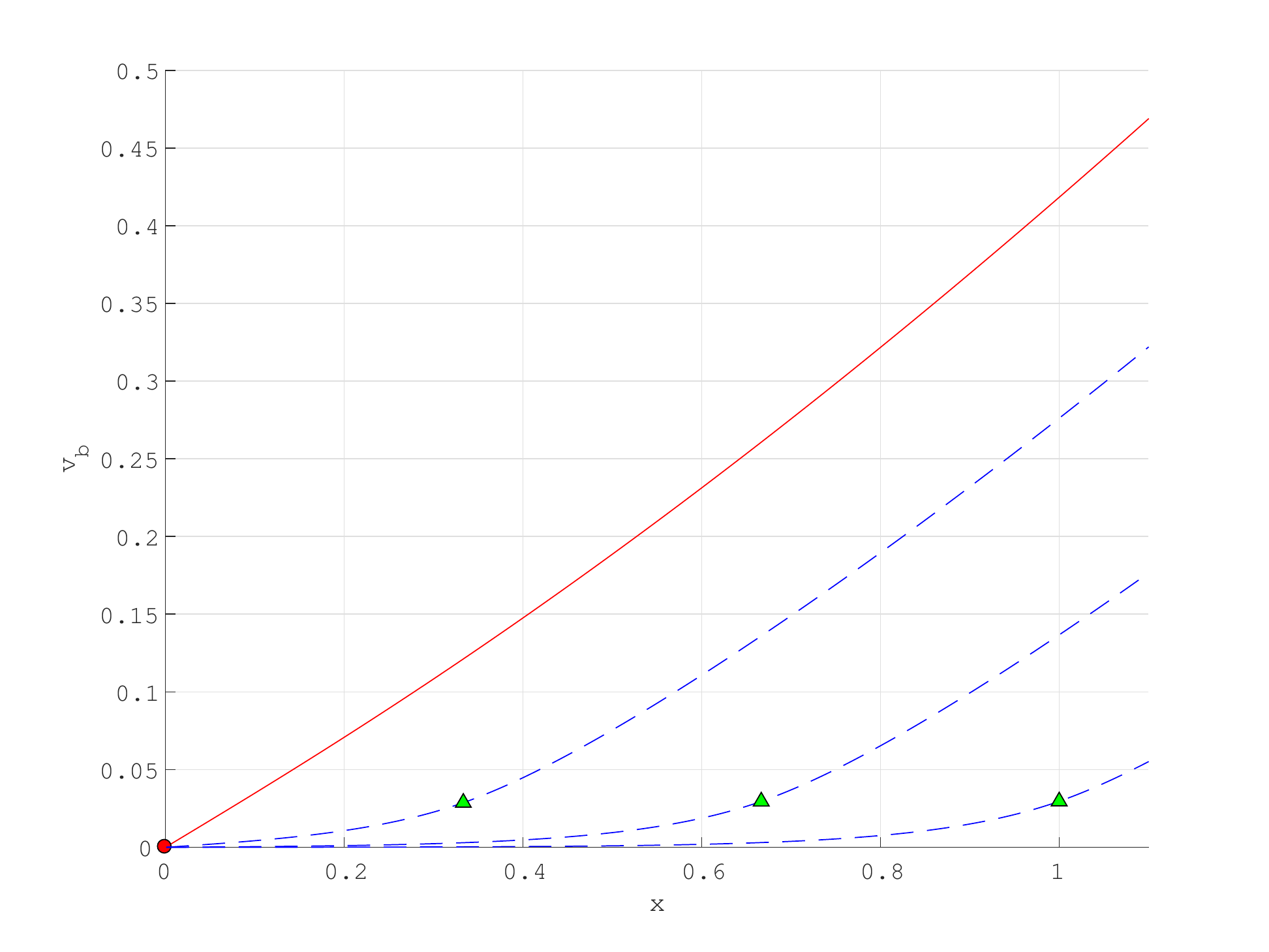} & \includegraphics[scale=0.35]{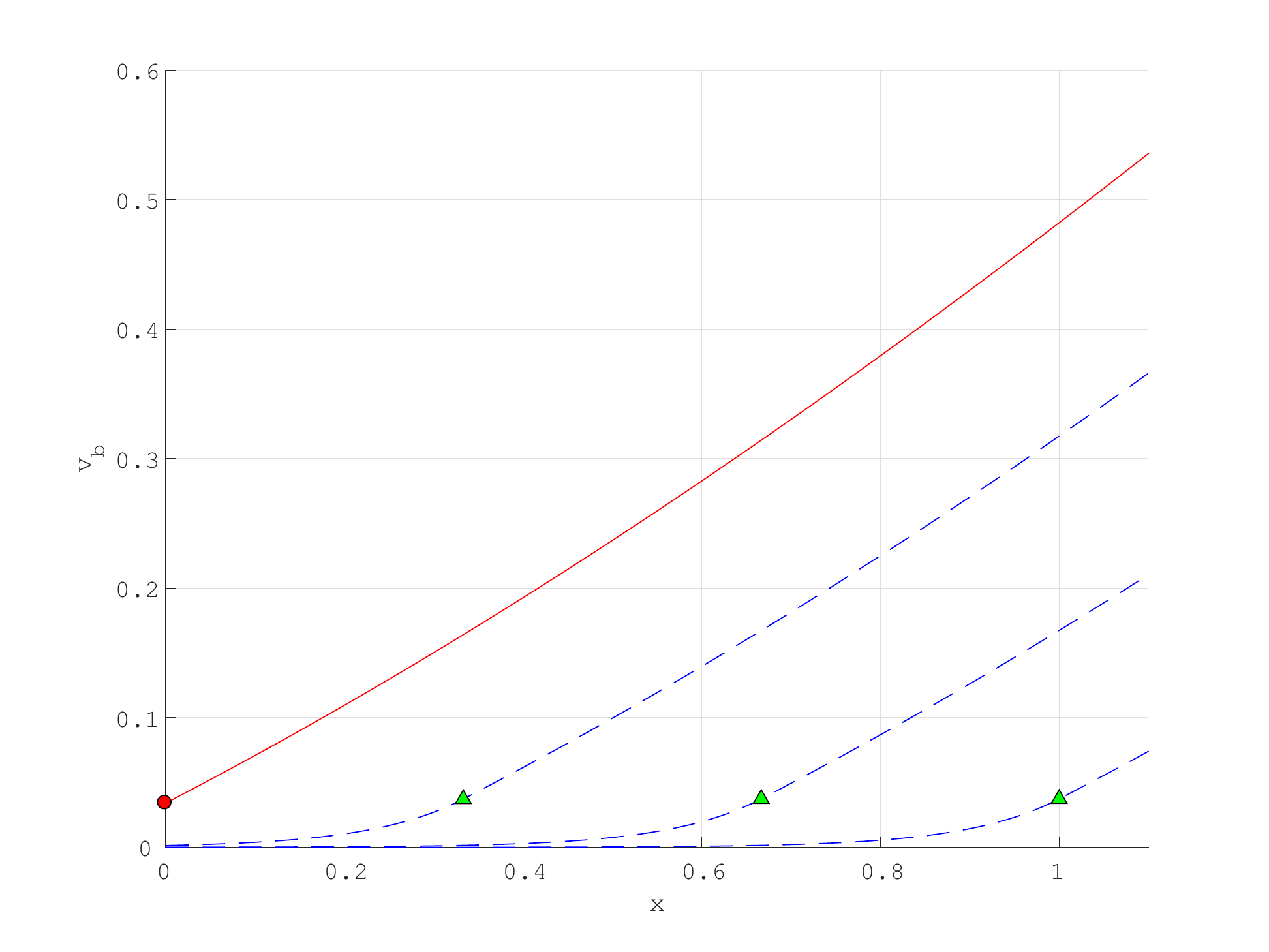}  \\
 Case 3 & Case 3' 
%  \includegraphics[scale=0.35]{fig_f_zero} & \includegraphics[scale=0.35]{fig_value_zero}  \\
 %$b \mapsto f(b)$ & $x \mapsto v_{b^*}(x)$ and $x \mapsto v_b(x)$ 
 \end{tabular}
\end{minipage}
\caption{The corresponding value function $v_{b^*}(x)$ (solid) along with suboptimal expected NPVs $v_b$ (dotted) for $ b = 0, b^*/4, b^*/2, 3b^*/4, (b^*+\bar{b})/2, \bar{b}$, 
    and $\bar{b}+(\bar{b}-b^*)/2$ for \textbf{Cases 1} and \textbf{1'}; $b =\bar{b}/2, \bar{b}$, and  $3\bar{b}/2$ for \textbf{Cases 2} and \textbf{2'}; and $b =1/3,2/3$, and $1$ for \textbf{Cases 3} and \textbf{3'}.  The values at $b^*$ are indicated by circles. Those at the suboptimal barriers $b > b^*$ (resp.\ $b < b^*$)  are indicated by up-pointing (resp.\ down-pointing) triangles when $b \neq \bar{b}$ and those at $b = \bar{b}$
 are indicated by squares.} \label{figure_v}
\end{center}
\end{figure}
 
 \subsection{Sensitivity analysis} We now numerically study the behaviors of the optimal barrier $b^*$ and the value function $v_{b^*}$ with respect to the parameters describing the problem.  In the remaining numerical results,  we set $\kappa = \lambda = 1$, $c = 1.5$, $r = 0.5$, and $q = 0.05$, unless stated otherwise. Both unbounded and bounded variation cases with $\sigma = 0.2$ and $\sigma = 0$ are considered.
 
 \begin{figure}[htbp]
\begin{center}
\begin{minipage}{1.0\textwidth}
\centering
\begin{tabular}{cc}
 \includegraphics[scale=0.35]{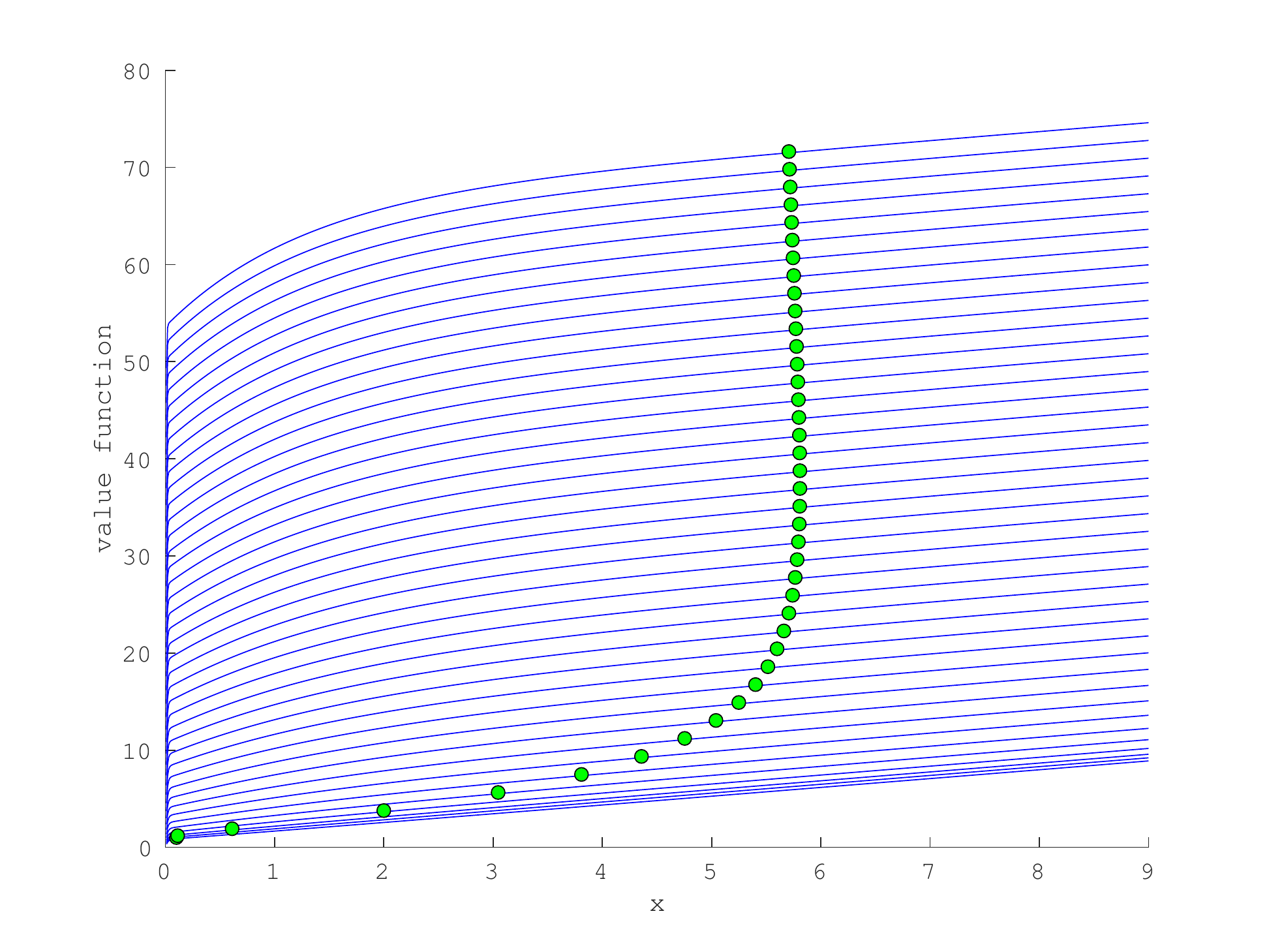} & \includegraphics[scale=0.35]{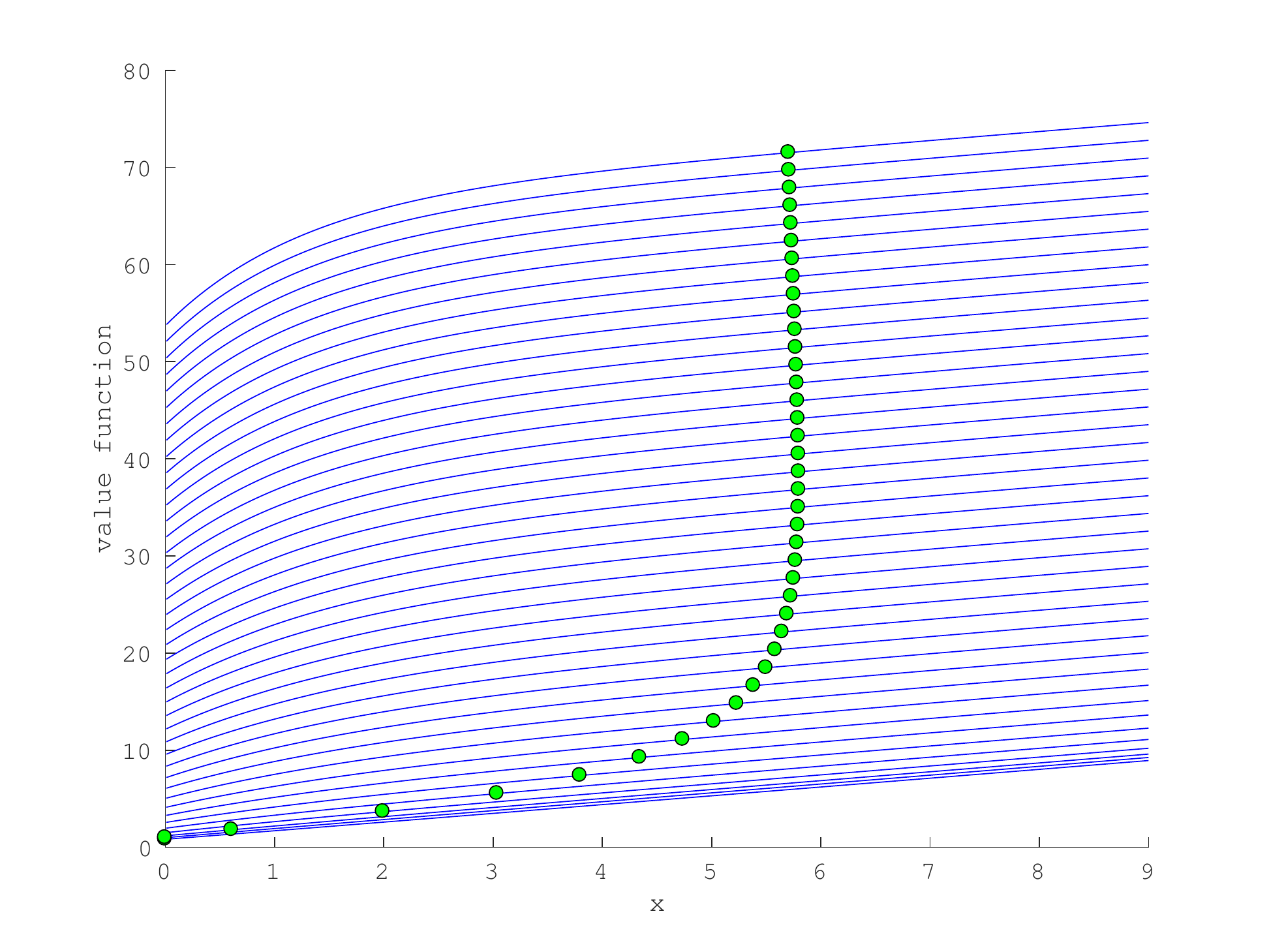}  \\
 $\sigma > 0$ & $\sigma = 0$ \end{tabular}
\end{minipage}
\caption{Plots of $v_{b^*}$ for $c=1,1.1,\ldots, 4.9$, and $5$ with $\sigma = 0.2$ (left) and $\sigma = 0$ (right). The values at $b^*$ are indicated by circles. The function $v_{b^*}$ is increasing in $c$ uniformly in $x$.
} \label{sensitivity_c}
\end{center}
\end{figure}

Figure \ref{sensitivity_c} plots $v_{b^*}$ and the points at $b^*$ for various values of the drift parameter $c$.  Naturally, the value function $v_{b^*}$ is increasing in $c$ uniformly in $x$. It is also observed that $b^* = 0$ for sufficiently small values of $c$, and increases in $c$ to some finite limit.
 
 \begin{figure}[htbp]
\begin{center}
\begin{minipage}{1.0\textwidth}
\centering
\begin{tabular}{cc}
 \includegraphics[scale=0.35]{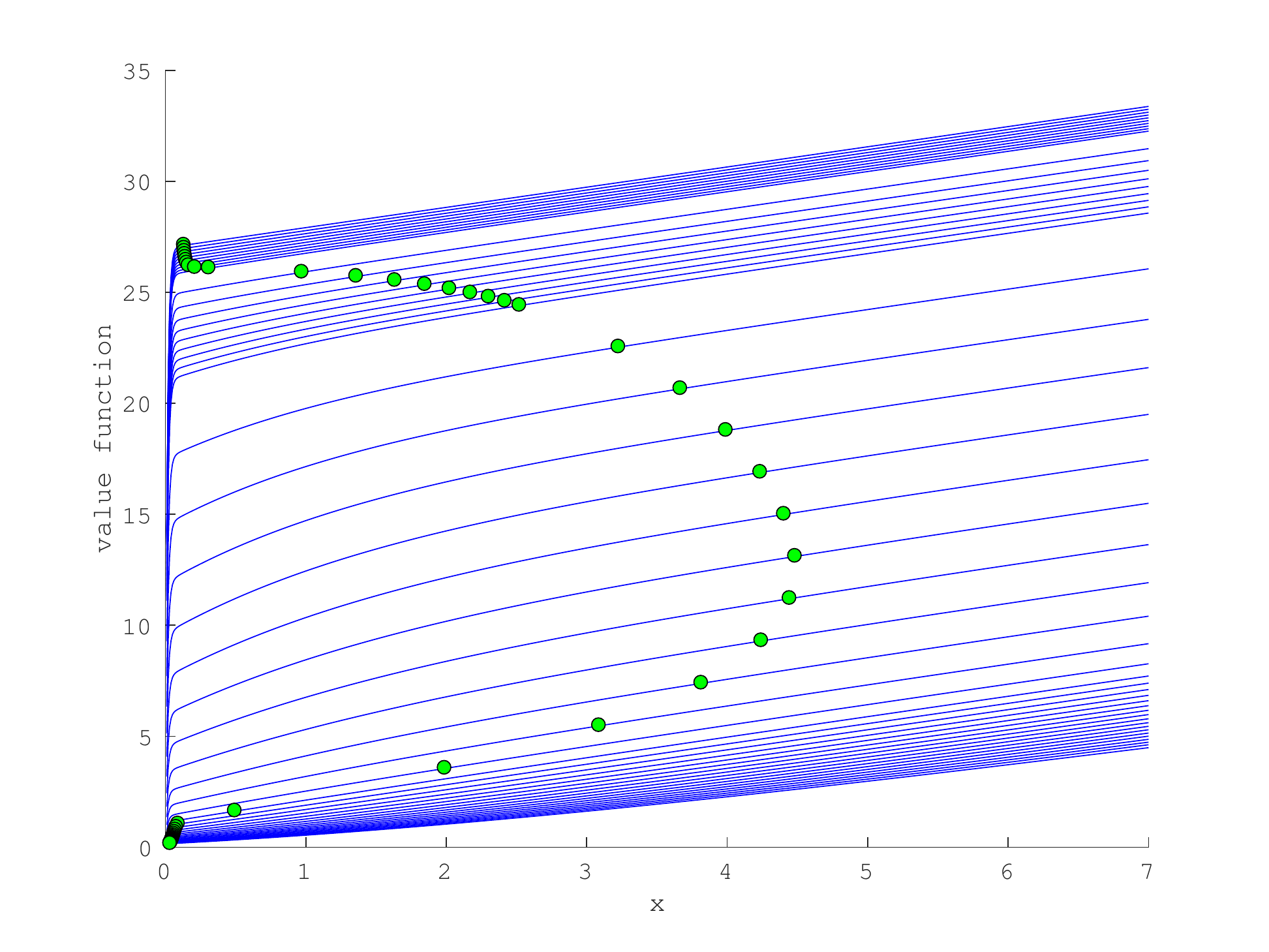} & \includegraphics[scale=0.35]{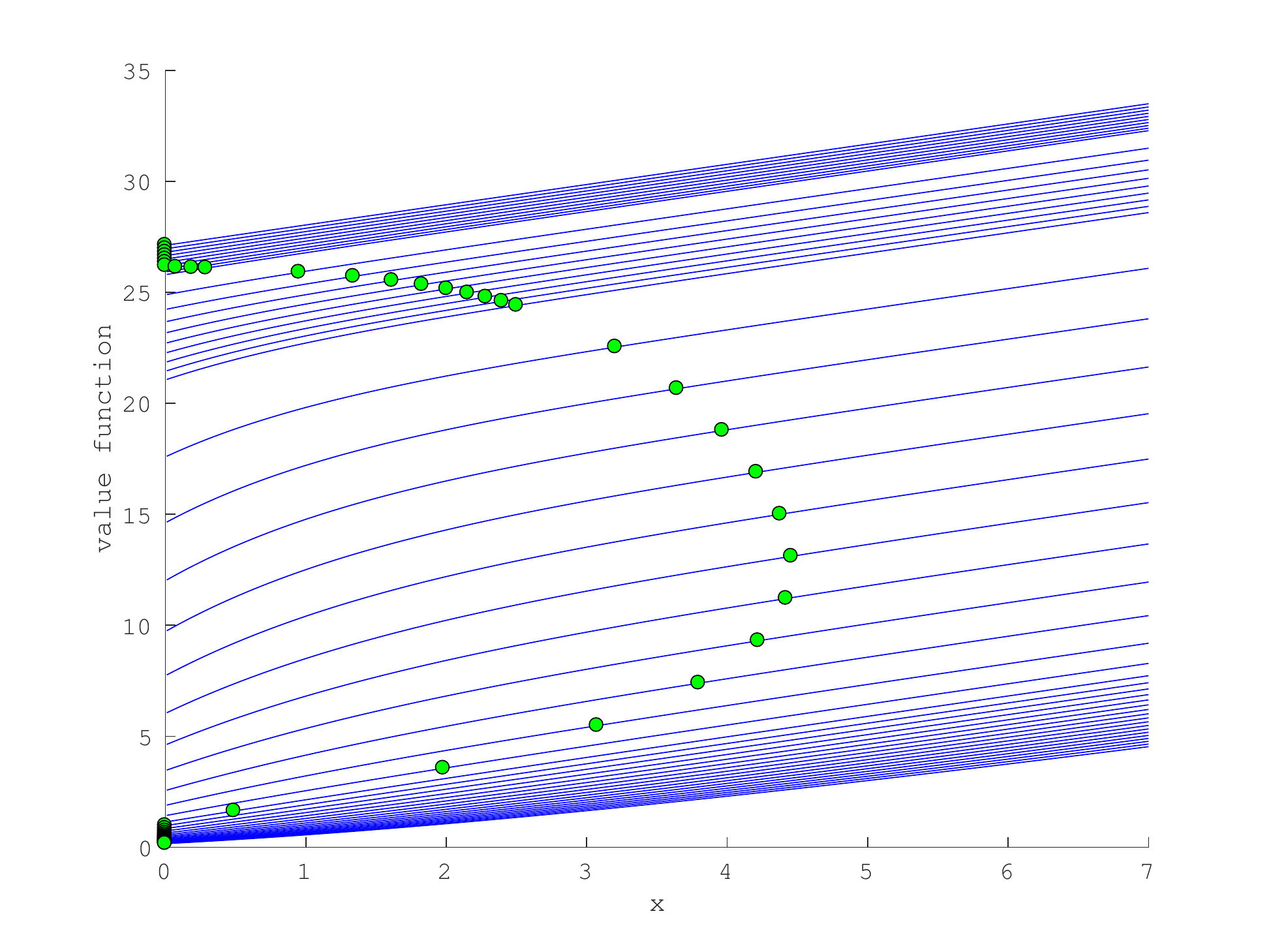}  \\
 $\sigma > 0$ & $\sigma = 0$ \end{tabular}
\end{minipage}
\caption{Plots of $v_{b^*}$ for $\kappa= 0.001$, $0.002$, $\ldots$, $0.008$,$0.009$,$0.01$,$0.02$, $\ldots$, $0.08$, $0.09$, $0.1$, $0.2$, $\ldots$, $2.9$, and $3$ with $\sigma = 0.2$ (left) and $\sigma = 0$ (right). The values at $b^*$ are indicated by circles. The function $v_{b^*}$ is decreasing in $\kappa$ uniformly in $x$.
} \label{sensitivity_kappa}
\end{center}
\end{figure}

Figures \ref{sensitivity_kappa} and \ref{sensitivity_lambda} plot the results for various values of the jump rate $\kappa$ and the jump-size parameter $\lambda$, respectively. It is confirmed that $v_{b^*}$ decreases in $\kappa$ and increases in $\lambda$ (uniformly in $x$). Interestingly, $b^*$ is not monotone in these parameters (contrary to what we observed in Figure \ref{sensitivity_c}).  When $\kappa$ is sufficiently large, the future aspect is negative, and hence $b^* =0$. As $\kappa$ decreases, $b^*$ departs from zero and starts increasing.  However, the lower the value of $\kappa$, the easier it is to avoid ruin.  Therefore, with sufficiently small $\kappa$, $b^*$ can be set low. Owing to these tradeoffs, $b^*$ is not monotone in $\kappa$.
%However, when $\kappa$ is sufficiently small, it gets easier to avoid ruin and hence $b^*$ can be set lower. 
 The same observation applies to the analysis for $\lambda$.

\begin{figure}[htbp]
\begin{center}
\begin{minipage}{1.0\textwidth}
\centering
\begin{tabular}{cc}
 \includegraphics[scale=0.35]{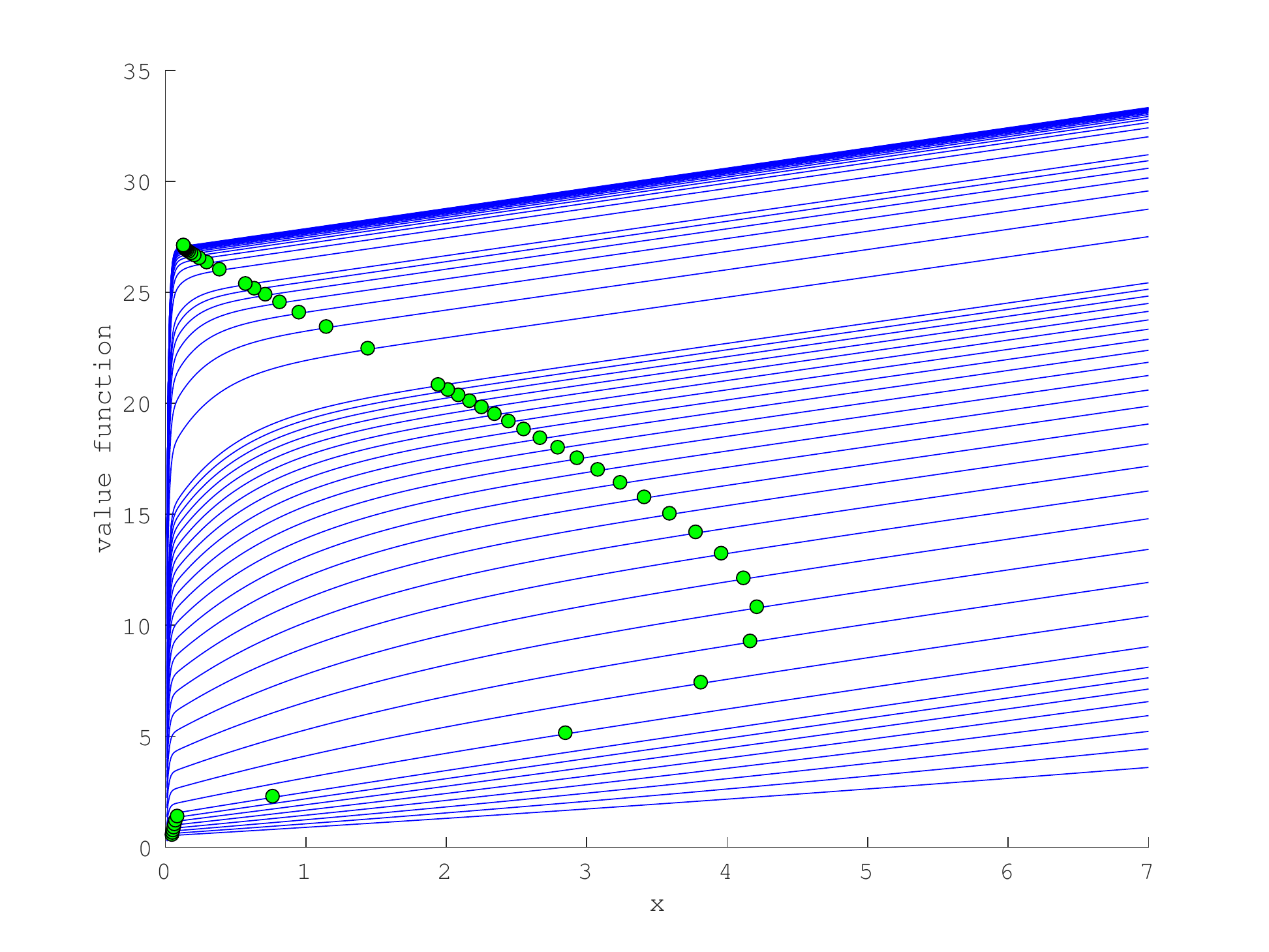} & \includegraphics[scale=0.35]{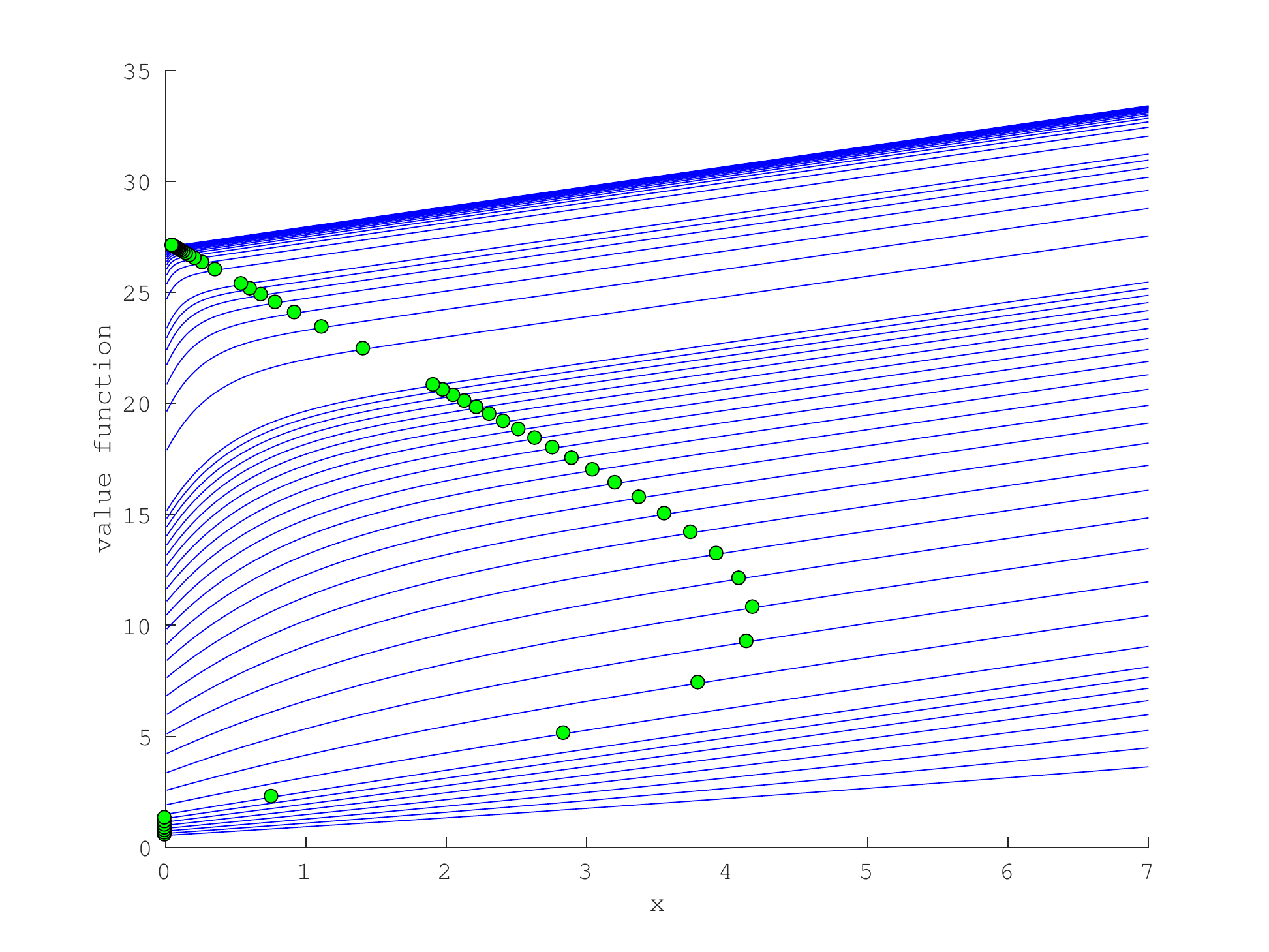}  \\
 $\sigma > 0$ & $\sigma = 0$ \end{tabular}
\end{minipage}
\caption{Plots of $v_{b^*}$ for $\lambda= 0.1, 0.2, \ldots, 2.9,3, 4, \ldots, 9,15, 20, \ldots, 95,$ and $100$ 
with $\sigma = 0.2$ (left) and $\sigma = 0$ (right). The values at $b^*$ are indicated by circles. The function $v_{b^*}$ is increasing in $\lambda$ uniformly in $x$.
} \label{sensitivity_lambda}
\end{center}
\end{figure}

\begin{figure}[htbp]
\begin{center}
\begin{minipage}{1.0\textwidth}
\centering
\begin{tabular}{cc}
 \includegraphics[scale=0.35]{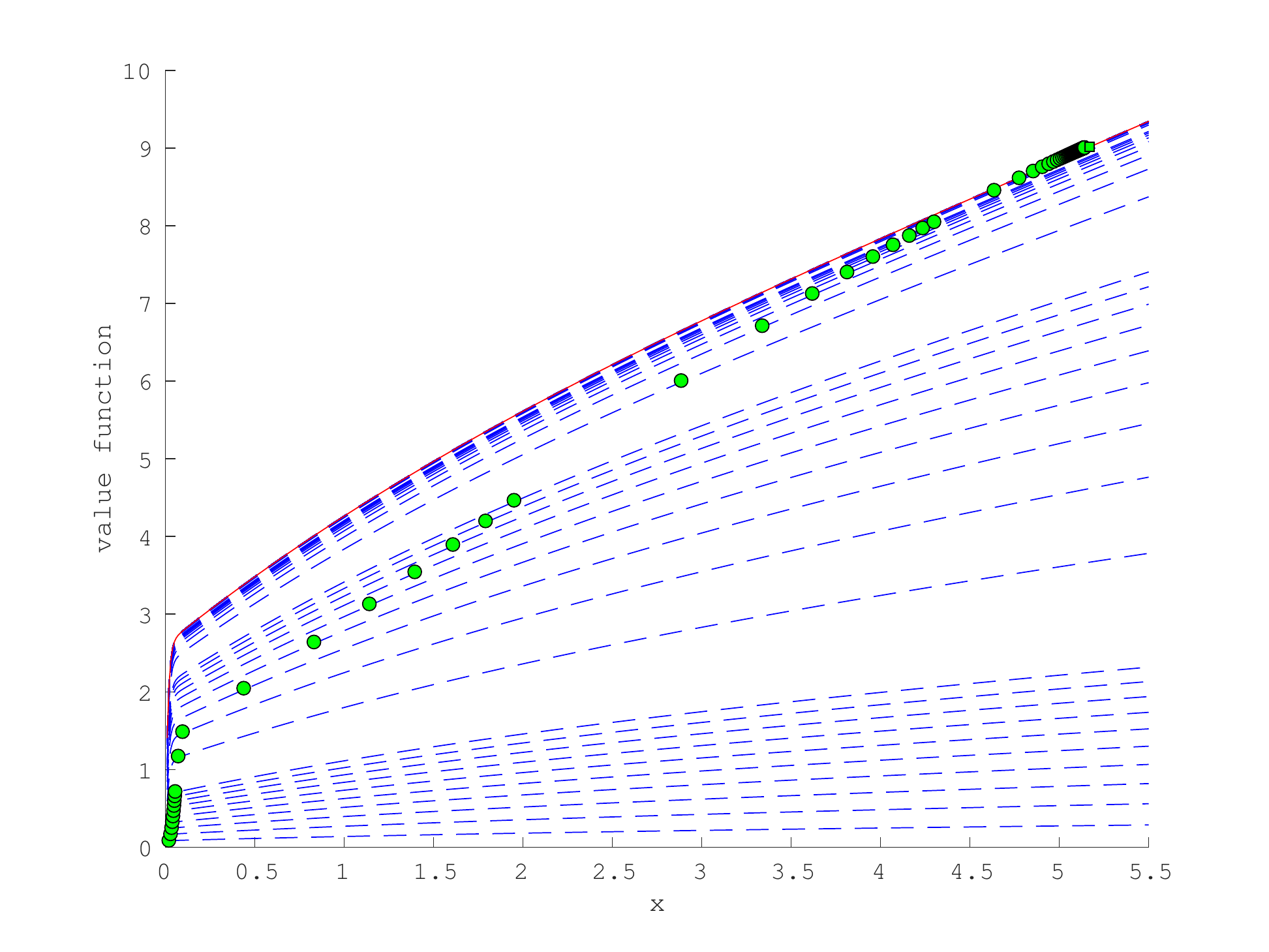} & \includegraphics[scale=0.35]{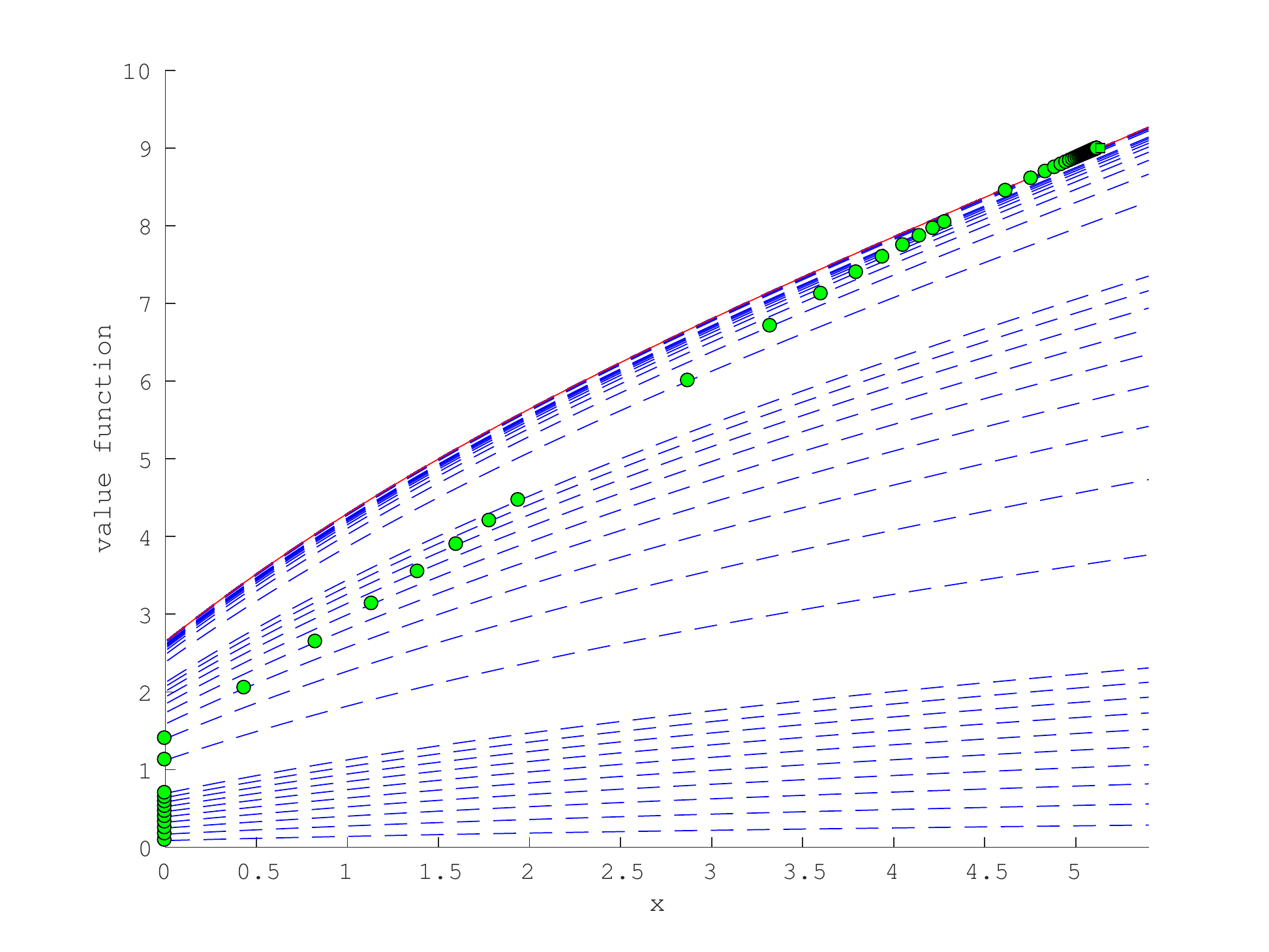}  \\
 $\sigma > 0$ & $\sigma = 0$ \end{tabular}
\end{minipage}
\caption{Plots of $v_{b^*}$ (dotted) for $r= 
0.001,0.002, \ldots, 0.01,0.02, \ldots, 0.09, 0.1, 0.2$, $\ldots, 0.9, 1, 2, \ldots, 99,$ and $100$ along with the classical value function $\bar{v}$ (solid) with $\sigma = 0.2$ (left) and $\sigma = 0$ (right). The values of $v_{b^*}$ at $b^*$ are indicated by circles and those of $\bar{v}$ at $\bar{b}$ are indicated by squares. The function $v_{b^*}$ is increasing in $r$ uniformly in $x$.
% (Bottom) Plots of $r \mapsto b^*$ for the value of $r$ up to $100$ with $\sigma = 0.2$ (left) and $\sigma = 0$ (right).
} \label{sensitivity_r}
\end{center}
\end{figure}

Finally, we study the behaviors of $v_{b^*}$ and $b^*$ with respect to the rate of dividend payment opportunities $r$.  Figure \ref{sensitivity_r} plots $v_{b^*}$ and the points at $b^*$ for various values of $r$ along with those in the classical case \eqref{classical_value_function}.  It is confirmed that $v_{b^*}$ is monotonically increasing in $r$ (uniformly in $x$)  to the classical case. 
%Note that, for the computation of $v_{b^*}$, due to the dependence on $\exp(\Phi(q+r))$, the value of $r$ cannot be chosen arbitrarily large for accurate results (in Matlab).
%Note that we were not able to compute (using Matlab) the value function for these values of $r$ as the form is (note that the value function is dependent on the exponential of $\Phi(q+r)$).
As studied in Lemma \ref{lemma_convergence_r_inf}, $b^*$ is monotone in $r$, and converges to zero as $r \downarrow 0$ and to $\bar{b}$ as $r \uparrow \infty$; this confirms the results in  Lemma \ref{lemma_convergence_r_inf}.  While the convergence to zero is relatively fast, we find that the convergence to $\bar{b}$ is rather slow.  In \cite{PY}, the same numerical analysis was obtained for the spectrally positive case; in their case $b^*$ was shown to accurately approximate the optimal barrier for the classical case even for a moderate value of $r$. We conjecture that this difference is due to the chance of jumping to ruin between Poisson observation times (which can be made negligible in the absence of downward jumps). With downward jumps, $b^*$ is more sensitive to the choice of $r$.  

\section*{Acknowledgements}
The authors are grateful to the  anonymous referees for helpful comments.
K.\ Noba and K.\ Yano were supported by JSPS-MAEDI Sakura program.
K. Yano was supported by KAKENHI 26800058 and partially by KAKENHI
15H03624 and KAKENHI 16KT0020. J. L. P\'erez  was  supported  by  CONACYT,  project  no.\ 241195.
K. Yamazaki was supported by MEXT KAKENHI grant no.\  	17K05377.

\end{document}